\documentclass[a4paper,reqno]{amsart}

\usepackage[utf8]{inputenc}
\usepackage{amssymb}
\usepackage{mathrsfs}
\usepackage{mathtools}
\usepackage{enumerate}
\usepackage{amsmath}
\usepackage{comment}
\usepackage{xcolor}
\usepackage{hyperref}
\usepackage{enumitem}

\newtheorem{theorem}{Theorem}[section]

\newtheorem{lemma}[theorem]{Lemma}
\newtheorem{proposition}[theorem]{Proposition}

\theoremstyle{definition}
\newtheorem{definition}[theorem]{Definition}
\newtheorem{remark}[theorem]{Remark}

\numberwithin{equation}{section}
\newcommand\RR{\mathbb{R}}
\newcommand\NN{\mathbb{N}}

\renewcommand{\le}{\leqslant}
\renewcommand{\ge}{\geqslant}
\allowdisplaybreaks[4]

\begin{document}
	\parindent=0pt
	\title[Rigidity of the Multi-Bubble Solutions]{Rigidity  of the multi-bubble solutions to the energy critical wave equation in dimension five}
\author[J. Jendrej]{Jacek Jendrej}
\address{Institut de Math\'{e}matiques de Jussieu, Sorbonne Universit\'{e}, Universit\'{e} Paris Cit\'{e}
4 place Jussieu, 75005 Paris, France,\ \& Faculty of Applied Mathematics, AGH University of Krak\'ow, al. Adama Mickiewicza 30, 30-059 Krak\'ow, Poland.}
\email{jendrej@imj-prg.fr}
\author[C.Zhang]{Chencheng Zhang}
\address{School of Mathematical Sciences,
University of Science and Technology of China, Hefei 230026, Anhui, China}
\email{zccmaths@mail.ustc.edu.cn}
 \author[L.Zhao]{Lifeng Zhao}
\address{School of Mathematical Sciences,
University of Science and Technology of China, Hefei 230026, Anhui, China}.
\email{zhaolf@ustc.edu.cn}	

\begin{abstract}
We study the asymptotic dynamics of multi-bubble solutions to the focusing energy-critical wave equation in five dimensions. Assuming that the solution asymptotically decomposes into a finite superposition of spatially separated bubbles with comparable scales, we prove a rigidity result that describes the precise long-time behavior of these scales.

More precisely, we show that all scaling parameters are necessarily of order $t^{-2}$, and that the corresponding renormalized modulation vector converges to a connected component of a finite-dimensional algebraic set determined by the limiting spatial configuration of the bubbles. This algebraic system encodes the strong interactions between the polynomial tails of the bubbles and governs the effective asymptotic dynamics of the multi-bubble regime.
\end{abstract}

	\maketitle

    \section{Introduction}
    We consider the energy-critical focusing wave equation in 5 dimensions:
    \begin{equation}\label{NLW}
        \partial^{2}_{t}u(t,x)=\Delta u(t,x)+f(u(t,x)),\quad t\in\RR,\ x\in\RR^{5},
    \end{equation}
    where $f(u):=|u|^{\frac{4}{3}}u$. Let $F(u):=\frac{3}{10}|u|^{\frac{10}{3}}.$ The energy functional related to this equation 
    \[
    E(\vec{u}(t)):=\int\limits_{\RR^{5}}\left(\frac{1}{2}|\partial_{t}u|^{2}+\frac{1}{2}|\nabla u|^{2}-F(u)\right){\rm{d}}x\]
    is well-defined for $\vec{u}(t):=(u(t),\partial_{t}u(t))\in\dot{H}^{1}(\RR^{5})\times L^{2}(\RR^{5})$ by the Sobolev inequality 
    \begin{equation}\label{sobolev}
    \lVert u\rVert_{L^{\frac{10}{3}}}\leq C \lVert\nabla u\rVert_{L^{2}}.
    \end{equation}
    We will write vectors with two components as $\vec{v}=(v,\dot{v})$, noting that the notation $\dot{v}$ will not, in general, refer to a time derivative of $v$ but rather just to  the second component of $\vec{v}$. With this notation (\ref{NLW}) can be rephrased as a Hamiltonian system
    \begin{equation}\label{NLW1}
		\left\{\begin{aligned}
      & \frac{d}{dt}\vec{u}(t)=J\circ {\rm{D}}E(\vec{u}(t)),\\
      & \vec{u}(t_0)=(u_{0},\dot{u}_{0})\in\dot{H}^{1}(\RR^{5})\times L^{2}(\RR^{5}),
		\end{aligned}\right.
	\end{equation}
    where
    \begin{equation*}
        J=\begin{pmatrix}
            0 &1\\
            -1 &0
        \end{pmatrix}, \quad {\rm{D}}E(\vec{u}(t))=\begin{pmatrix}
            -\Delta u(t)-f(u(t))\\
            \partial_{t}u(t)
        \end{pmatrix}.
    \end{equation*}
    We recall that this equation is locally well-posed in the energy space $\dot{H}^{1}\times L^{2}$ (see for example \cite{GSV,SS,SS1} and references therein). In particular, for any initial data $(u_{0},\dot{u}_{0})$ there exists a maximal interval of existence $(T_{-},T_{+})$, $-\infty\leq T_{-}<t_{0}<T_{+}\leq{+\infty}$, and a unique solution $\vec{u}(t)\in C((T_{-},T_{+});\dot{H}^{1}\times L^{2})\cap L^{\frac{7}{3}} _{\rm{loc}}((T_{-},T_{+});L^{\frac{14}{3}}(\RR^{5}))$. For such solutions, the energy $E(\vec{u}(t))$ is constant in time.\\
    For a function $v:\RR^{5}\rightarrow\RR$ and $\lambda>0$, $z\in\RR^{5}$, we denote
    \begin{equation*}
        v_{\lambda,z}(x):=\frac{1}{\lambda^{\frac{3}{2}}}v\left(\frac{x-z}{\lambda}\right),\quad v_{\underline{\lambda,z}}(x):=\frac{1}{\lambda^{\frac{5}{2}}}v\left(\frac{x-z}{\lambda}\right).
    \end{equation*}
    A change of variables shows that 
    \begin{equation*}
        E\left((u_{0})_{\lambda,z},(\dot{u}_{0})_{\underline{\lambda,z}}\right)=E(u_{0},\dot{u}_{0}).
    \end{equation*}
    Equation (\ref{NLW1}) is invariant under the same scaling. If $(u,\partial_{t}u)$ is a solution of (\ref{NLW1}) and $\lambda>0$, $z\in\RR^{5}$, then
    \[
    t\mapsto\left(u\left(\frac{t}{\lambda}+t_{0}\right)_{\lambda,z},\partial_{t}u\left(\frac{t}{\lambda}+t_{0}\right)_{\underline{\lambda,z}}\right)
    \]
    is also a solution with initial data $\left((u_{0})_{\lambda,z},(\dot{u}_{0})_{\underline{\lambda,z}}\right)$ at time $t=0$. This is why equation (\ref{NLW1}) is called energy-critical.\\
    We also introduce the infinitesimal generators of scale change:
    \[\Lambda v:=-\frac{\partial}{\partial\lambda}v_{\lambda,0}\bigg|_{\lambda=1}=\left(\frac{3}{2}+x\cdot\nabla \right)v\quad (\dot{H}^{1}\ {\rm{scaling}}),\]
    \[\underline{\Lambda}v:=-\frac{\partial}{\partial\lambda}v_{\underline{\lambda,0}}\bigg|_{\lambda=1}=\left(\frac{5}{2}+x\cdot\nabla\right)v\quad (L^{2}\ {\rm{scaling}}).\]
    A fundamental object in the study of (\ref{NLW1}) is the family of solutions $\vec{u}(t)=(W_{\lambda,z},0)$, where 
    \[W(x):=\left(1+\frac{|x|^{2}}{15}\right)^{-\frac{3}{2}},\quad x\in\RR^{5}\]
    is the ground state solution of the elliptic equation
    \begin{equation*}
        \Delta W=-W^{\frac{7}{3}}\quad {\rm{on}}\ \RR^{5}.
    \end{equation*}
    Up to scaling and translation, $W(x)$ is the unique positive solution to this elliptic equation.\\
    It is well-known that the ground state $W$ attains  the optimal constant in the critical Sobolev inequality \eqref{sobolev}, see \cite{Aub1976,Tal1976}. 
 Moreover, $W$ is the threshold element for global existence and scattering in the Kenig--Merle theory for the focusing energy-critical wave equation, see \cite{KenMer2008}. Beyond this critical threshold, the global-in-time behavior ceases to be purely dispersive, opening the door to a rich variety of nonlinear phenomena that demand a systematic global classification.\\
 The definitive guiding principle for describing such complex long-time dynamics is the soliton resolution conjecture, which predicts that, as time approaches the endpoint of the lifespan (which may be finite or infinite), a solution should decompose into a finite sum of decoupled solitons and a dispersive radiation term. For the focusing energy-critical wave equation, the radial case of this conjecture has been completely settled, see \cite{DKM2012,DKM2013,Rod2016,DKM2023,DKMM2022,JL2023a,CDKM2024}. In the
non-radial setting, the full continuous-in-time resolution remains open. The main general result beyond symmetry is a sequential soliton resolution theorem, which gives such a decomposition along a sequence of times, see \cite{DJKM2017}. \\
Although soliton resolution gives the expected asymptotic form of a general bounded solution, it usually does not determine the number of solitons nor the precise asymptotic behavior of their geometric parameters, apart from the decoupling between the bubbles and the radiation.  This leaves open a more quantitative problem: to understand which soliton configurations are actually realized by solutions, and how the corresponding parameters evolve.\\
The first situation in which such quantitative questions can be addressed is the dynamics near a single copy of the ground state. Solutions with initial data close to $(W,0)$ were studied by Krieger--Schlag \cite{KS2007} and Beceanu \cite{Bec2014}, who constructed, in suitable topologies, manifolds of initial data for which the solution converges to the soliton family and gave a precise description of these solutions. The long-time behavior of all the solutions with initial data close to $(W,0)$ in the energy space was later classified by Duyckaerts--Merle \cite{DM} and Krieger--Nakanishi--Schlag \cite{KNS2013}.  These works exhibit a rich variety of one-bubble type-II dynamics and concentration behaviors. Simultaneously, one-bubble type-II blow-up and concentration regimes have been constructed for the energy-critical wave equation.  In dimension $3$, this includes finite-time blow-up solutions with polynomial or more exotic scaling laws, as well as infinite-time concentration and non-scattering regimes, see \cite{KST2009,DK2013,KS2014,DHKS2014}.  In dimension $4$,
Hillairet--Rapha\"el \cite{HR2012} constructed finite-time type-II blow-up solutions with a logarithmically corrected rate.  In dimension $5$, the first author \cite{Jen2017a} constructed radial one-bubble type-II blow-up solutions with prescribed asymptotic profiles and explicit concentration rates. More recently, Samuelian \cite{Sam2026} constructed finite-time one-bubble type-II blow-up solutions with prescribed polynomial rates in dimensions $4$ and $5$.  These works demonstrate that even the one-bubble regime exhibits delicate modulation dynamics.\\ 
Beyond the intricate mechanics of a single blow-up core, a natural progression is to explore the multi-bubble regime. In this multi-bubble regime, the dynamics are no longer governed only by the instability of a single ground state, but also by the nonlinear interactions between different bubbles.  For the energy-critical wave equation, several multi-bubble regimes have been constructed. In high dimensions, the first author \cite{Jen2019} constructed radial two-bubble solutions that decompose asymptotically into a concentrating bubble superposed with a standing soliton of the same sign.  Another important regime is given by the traveling multi-soliton constructions of Martel and Merle. They first constructed, in dimension $5$, solutions containing an arbitrary number of bounded traveling solitons under suitable restrictions on their
speeds \cite{MM2016}, and later proved the inelasticity of soliton
interactions in the same setting \cite{MM2018}.  More recently, they
constructed multi-solitons with arbitrary parameters for the
$5$-dimensional energy-critical wave equation \cite{MM2025}.  We also mention the recent construction of multi-soliton solutions in dimension $3$ by Kadar \cite{Kad2024}, which is based on a different approach. A further regime, introduced by the first author and Martel \cite{JM}, consists of
non-radial pure multi-bubbles concentrating at distinct fixed spatial points. This regime is different from the same-center two-bubble dynamics, where the interaction is essentially radial, and differs from traveling multi-bubbles, where the leading parameters are the velocities and relative trajectories of the solitons.\\
These constructions show that multi-bubble dynamics for the energy-critical wave equation are highly diverse.  Depending on the interaction regime, the dominant parameters may be the relative scales, velocities, signs, or spatial configuration of the bubbles.  Thus a general rigidity theory covering all possible multi-bubble dynamics seems far out of reach at present.  We therefore restrict our attention to the dynamical framework introduced by the first author and Martel \cite{JM}: non-radial pure multi-bubbles whose centers converge to distinct spatial points and whose scaling parameters remain comparable. This framework represents an important class of multi-bubble dynamics.  Indeed, similar fixed-center, same-scale pure multi-bubble configurations have also been constructed in both parabolic and dispersive models, see for instance
\cite{Mer1990, MR2018, CPM2020, CSZ2024, RSZ2024}.  It is therefore natural to study the rigidity of pure multi-bubble solutions constructed in \cite{JM}.\\
More precisely, we study the rigidity of the pure multi-bubble solutions as $t\to+\infty$.  Namely, we consider solutions $  (u,\partial_tu):[0,+\infty) \to \dot H^1(\RR^5)\times L^2(\RR^5)$ satisfying
\begin{equation}\label{Bub1}
    \lim_{t\to+\infty}\left(\left\|u(t)-\sum_{k=1}^{K}W_{\mu_k(t),y_k(t)}
\right\|_{\dot H^1(\RR^5)}+ \|\partial_tu(t)\|_{L^2(\RR^5)}
 \right)=0,
\end{equation}
where $K\geq2$, and $\mu_k:[0,+\infty)\to(0,+\infty),\       y_k:[0,+\infty)\to\RR^5
$ are continuous functions.\\
A complete classification of all solutions satisfying \eqref{Bub1} alone seems to be a very difficult problem.  We therefore impose the additional assumptions corresponding to the fixed-center, comparable-scale concentrating regime:
\begin{equation}\label{asp}
\begin{aligned}
    &\lim_{t\to+\infty}\mu_k(t)=0,
    \qquad 1\leq k\leq K,\\
    &C_1\leq \frac{\mu_k(t)}{\mu_j(t)}\leq C_2,
    \qquad 1\leq k,j\leq K,\quad t\geq0,\\
    &\lim_{t\to+\infty}y_k(t)=z_k,
    \qquad 1\leq k\leq K,
\end{aligned}
\end{equation}
where $C_1,C_2>0$ are fixed constants and $z_1,\ldots,z_K$ are pairwise distinct points in $\RR^5$.  Thus the bubbles concentrate at comparable scales and their centers converge to distinct limiting points.  The second condition excludes tower-type regimes and scale cascades, while the third one rules out collisions of the limiting centers.  Thus all bubbles remain in a collective interaction regime determined by their spatial configuration.
\begin{remark}
The condition $\mu_k(t)\to0$ in \eqref{asp} is included to emphasize that we
work in a concentrating multi-bubble regime.  It is not an independent
assumption: under \eqref{Bub1} together with the comparable-scale and
non-colliding-center assumptions in \eqref{asp}, one necessarily has
$\mu_k(t)\to0$ for all $1\leq k\leq K$.  Equivalently, the bubbles become
pairwise asymptotically orthogonal with respect to the distance $\delta$
introduced below.  The proof is given in Appendix~B.
\end{remark}
The existence of solutions satisfying \eqref{Bub1}--\eqref{asp} was proved by the first author and Martel \cite{JM}.  More precisely, for any integer $K\geq2$ and any pairwise distinct points $   z_1,\ldots,z_K\in\RR^5$, they constructed a global forward-in-time solution which satisfies
\begin{equation}\label{intro:JM-solution}
 \left\|
u(t)-\sum_{k=1}^{K}
 \frac{1}{(c_k t^{-2})^{\frac32}}
 W\left(\frac{\cdot-z_k}{c_k t^{-2}}\right)
\right\|_{\dot H^1(\RR^5)}
+\|\partial_tu(t)\|_{L^2(\RR^5)}   \lesssim t^{-\frac13},   \qquad t\to+\infty,
\end{equation}
where $c_1,\ldots,c_K>0$ are constants depending on the mutual distances
between the points $z_1,\ldots,z_K$.\\
Motivated by this construction, the first author and Martel conjectured that the rate $t^{-2}$ is rigid in this regime.  The present paper proves this rigidity statement.  Under  assumptions \eqref{Bub1}--\eqref{asp}, we show that the scaling parameters of the bubbles must be of order $t^{-2}$.
Beyond this leading-order rigidity, we also obtain the convergence of the suitably rescaled modulation parameters to a finite-dimensional algebraic set determined by the limiting centers $z_1,\ldots,z_K$.
To the best of our knowledge, this is the first rigidity result for non-radial pure multi-bubble solutions of the energy-critical wave equation in which the scaling parameters and the spatial translation parameters enter the asymptotic dynamics. 
\subsection{Main result}
We now state the main result.
 \begin{theorem}\label{main theorem}
Let $\vec{u}(t):[0,+\infty)\rightarrow\dot{H}^{1} (\RR^{5})\times L^{2}(\RR^{5})$ be a solution to (\ref{NLW1}) satisfying (\ref{Bub1}), (\ref{asp}). Then for $t$ large enough and some fixed constant $C_{1}>0,C_{2}>0$, there exist  modulated parameters $\lambda_{k}(t),b_{k}(t),x_{k}(t)\in C^{1}$  satisfying
\begin{equation}\label{estimates for parameters}
    C_{1}t^{-2}\leq \lambda_{k}(t)\leq C_{2}t^{-2},\ C_{1}t^{-3}\leq b_{k}(t)\leq C_{2}t^{-3},\ \lim_{t\rightarrow+\infty}t^{2}|x_{k}(t)-z_{k}|=0,
\end{equation}
and 
\begin{equation}\label{estimate for remainder terms}
   \lim_{t\rightarrow+\infty}t^{3}\left(\bigg\lVert u(t)-\sum_{k=1}^{K}W_{\lambda_{k},x_{k}}\bigg\rVert_{\dot{H}^{1}(\RR^{5})}+
       \bigg\lVert\partial_{t}u(t)-\sum_{k=1}^{K}b_{k}(\Lambda W)_{\underline{\lambda_{k},x_{k}}}\bigg\rVert_{L^{2}(\RR^{5})}\right)=0.
\end{equation}
Moreover,  the set \begin{equation*}
      Eq(F)=\left\{(\vec{a},\vec{c})=((a_{k})_{1\leq k \leq K},(c_{k})_{1\leq k\leq K})\left|\begin{aligned}
      & 2a_{k}=c_{k},\\
      & 3c_{k}=\kappa\sum_{j\neq k, 1\leq j\leq K}|z_{j}-z_{k}|^{-3}a_{k}^{\frac{1}{2}}a^{\frac{3}{2}}_{j},\\
      & a_{k}>0,\ c_{k}>0,\ \forall\ 1\leq k\leq K.
		\end{aligned}\right.\right\}
 \end{equation*}
is non-empty, and if we denote $\vec{\lambda}(t)=((\lambda_{1}(t),\cdots,\lambda_{K}(t)),\ \vec{b}(t)=(b_{1}(t),\cdots, b_{K}(t))$, then $(t^{2}\vec{\lambda}(t),t^{3}\vec{b}(t))$ converges to a connected component of Eq(F).
\end{theorem}
\begin{remark}
The introduction of the auxiliary parameters $b_k$ is a key point in the
analysis.  It allows us to construct a refined modulation decomposition whose remainder is much smaller than the critical size $t^{-3}$.  This gain is crucial: with such a precise expansion, one can derive the asymptotic modulation equations for the parameters and then study the convergence of the suitably rescaled parameters to the finite-dimensional algebraic system in the main theorem.
\end{remark}
\begin{remark}
		The emergence of the algebraic system $Eq(F)$ highlights a profound structural property of the multi-bubble dynamics in the five-dimensional energy-critical setting. Rather than being merely a technical byproduct of the modulation estimates,, $Eq(F)$ represents a finite-dimensional reduction of the infinite-dimensional flow as $t \to +\infty$. Specifically, the system quantifies the strong interactions between bubbles: the coupling term $\sum_{j\neq k} |z_j - z_k|^{-3} a_k^{1/2} a_j^{3/2}$ exactly captures the overlap effects of the polynomial tails of the ground states $W$. Furthermore, the convergence of the rescaled parameters to a connected component of $Eq(F)$ establishes  strong asymptotic rigidity, indicating that such multi-bubble configurations can only exist if their spatial centers and blow-down rates satisfy these precise algebraic constraints.
\end{remark}

	The algebraic set $Eq(F)$ exhibits a highly complex topological structure that heavily depends on the number of bubbles $K$ and their spatial configuration $\{z_k\}_{k=1}^K$. As $K$ increases, the geometry of $Eq(F)$ transitions from a completely rigid singleton to a potentially non-discrete manifold.

\begin{remark}
    When $K=2$, a direct computation shows that $Eq(F)$ is a singleton:
\[
Eq(F)=\left\{\left(\frac{6|z_{1}-z_{2}|^{3}}{\kappa},\frac{6|z_{1}-z_{2}|^{3}}{\kappa}, \frac{12|z_{1}-z_{2}|^{3}}{\kappa},\frac{12|z_{1}-z_{2}|^{3}}{\kappa}\right)\right\}.
\]
Consequently, we have
	\[\lambda_{1}(t)=\lambda_{2}(t)=t^{-2}\left(\frac{6|z_{1}-z_{2}|^{3}}{\kappa}+o_{t\rightarrow+\infty}(1)\right),
	\]
	\[b_{1}(t)=b_{2}(t)=t^{-3}\left(\frac{12|z_{1}-z_{2}|^{3}}{\kappa}+o_{t\rightarrow+\infty}(1)\right).\]
	
	When $K=3$, we show that every point of the set $Eq(F)$ is isolated  (see Proposition A.1). Hence, in this case, the vector $(t^{2}\vec{\lambda}(t),t^{3}\vec{b}(t))$ converges to a fixed point.
\end{remark}

\begin{remark}
	For general $K\geq 3$, it follows from Theorem \ref{main theorem} that the set of limit points of the $t^{-2}$-renormalized vector $t^{2}\vec{\lambda}(t)$ depends only on the choice of the points $z_{1},z_{2},\cdots,z_{K}$, which is in perfect agreement with the form of the solutions constructed by Jendrej and Martel \cite{JM}. In particular, if every point of $Eq(F)$ is isolated, then $(t^{2}\vec{\lambda}(t),t^{3}\vec{b}(t))$ converges to a fixed point.
\end{remark}

\begin{remark}
    However, for larger $K$, the nonlinear algebraic constraints governing $Eq(F)$ become highly complex and degenerate. The set $Eq(F)$ is \emph{not} necessarily discrete; in fact, we can explicitly construct a configuration for $K=10$ for which $Eq(F)$ forms a non-discrete set (see Proposition A.2). In such cases, rather than settling at a fixed point, the rescaled parameters may, in principle, drift along a non-trivial connected component, reflecting an intricate dynamical instability of large bubble clusters. Moreover, motivated by Proposition A.2, we {\it conjecture} that there exist $K$-bubble solutions to the five-dimensional energy-critical wave equation whose rescaled parameter vectors $(t^{2}\vec{\lambda}(t),t^{3}\vec{b}(t))$ fail to stabilize to a unique fixed point. Instead, as time $t \to +\infty$, these parameters exhibit sustained oscillations and drift perpetually along the closed one-dimensional manifold. 
\end{remark}

\begin{remark}
The convergence of the rescaled parameters to the algebraic set $Eq(F)$ has a striking conceptual parallel with the classical $n$-body problem in celestial mechanics, see \cite[Chapter V]{wintner1941}. It is a well-known phenomenon that when a system of $n$ bodies undergoes a simultaneous collision---collapsing to a single point at some finite time $T$---the spatial configuration exhibits strict polynomial decay. Specifically, the polar moment of inertia decays asymptotically as $(T-t)^{4/3}$, meaning the spatial configuration must be dynamically rescaled by the factor $(T-t)^{-2/3}$ to capture the asymptotic geometry. Under this rescaling, the collapsing system converges to the family of \emph{central configurations}, which, much like our set $Eq(F)$, are stationary states determined by a finite-dimensional system of algebraic equations balancing the mutual interaction forces. 
%TFurthermore, in the $n$-body problem, it remains a notoriously difficult open question whether a collapsing orbit must converge to a \emph{unique} central configuration (up to natural invariances), or if it can drift along a continuum of such configurations. This classical uncertainty in celestial mechanics provides a profound historical context for our results: while we prove convergence to a unique fixed point for $K \le 3$ due to the discreteness of $Eq(F)$, our explicit construction of a non-discrete continuum for $K=10$ rigorously demonstrates that, for larger systems, the rescaled parameters might only converge to a connected component, theoretically allowing for non-trivial geometric drift in the asymptotic limit.
\end{remark}

\begin{remark} In the classical \(n\)-body problem, it remains a notoriously difficult open question---encompassing Wintner's limit conjecture and Smale's 6th problem---whether a collapsing orbit must converge to a \emph{unique} central configuration (up to natural invariances), or whether it may drift along a continuum of such configurations. Our result suggests a striking analogy with this phenomenon. In the present setting, the spatial poles \(z_k\) play the role of the masses, while the algebraic set \(Eq(F)\) naturally corresponds to the set of central configurations governing the effective asymptotic dynamics. From this perspective, it is natural to expect that the long-time behavior of multi-bubble solutions should exhibit a similar rigidity phenomenon. More precisely, one may conjecture that for a generic configuration of spatial centers \(\{z_k\}_{k=1}^K\subset\mathbb R^5\), the equilibrium set \(Eq(F)\) is discrete, and that for generic initial data in the manifold of multi-bubble solutions, the renormalized modulation parameters
\[
(t^{2}\vec{\lambda}(t),\, t^{3}\vec{b}(t))
\]
converge to a unique equilibrium point of \(Eq(F)\) as \(t\to+\infty\).    
\end{remark}

\begin{remark}    
Beyond the rigidity results, a fundamental question remains regarding the complete classification of all solutions satisfying \eqref{Bub1}. The family of solutions constructed in this work constitutes a $5K$-dimensional manifold, which is naturally parameterized by the asymptotic spatial centers $z_k \in \mathbb{R}^5$ of the $K$ bubbles. Taking into account that the linearized operator around each bubble possesses exactly one stable direction, we conjecture that the set of all solutions satisfying these asymptotic assumptions forms a $6K$-dimensional manifold.
\end{remark}
\begin{remark}[Other works on multi-bubble rigidity and classification]\label{rem:related-rigidity} Here we collect several rigidity and classification results for pure
multi-bubble or multi-soliton dynamics in related models.
For the energy-critical wave equation itself, we mention the recent radial
results of Shen \cite{Shen2026a,Shen2026b} in dimension $3$: pure multi-bubble type-II blow-up solutions do not exist, nor do radial global or type-II blow-up solutions with two or more bubbles.  These results are different in nature from our result, but they also reflect the rigidity of multi-bubble dynamics.\\
Related problems have also been studied in other critical 
models.  For equivariant wave maps, we refer to
\cite{JL2018,JL2023b,JLR2022} for threshold two-bubble dynamics, uniqueness
of two-bubble solutions, and bubbling dynamics with prescribed radiation.
For scalar field models in dimension $1+1$, strongly interacting
kink-antikink pairs and more general kink clusters were studied in
\cite{JKL2022,JL2024}.  For generalized KdV equations, classification and
uniqueness-type results for solutions converging to multi-solitons, including
strongly interacting two-soliton dynamics, were obtained in
\cite{Combet2011,Jen2025}.  For the mass-critical nonlinear Schr\"odinger
equation, uniqueness results for multi-bubble blow-up solutions and
multi-solitons were proved in \cite{CSZ2023}.\\
We also mention recent parabolic results.  Kim--Merle \cite{KM2025}
classified bubble-tree dynamics for high-equivariance harmonic map heat flows
and radial energy-critical nonlinear heat equations.  In a subsequent work,
they proved rigidity results for non-radial multi-bubble dynamics of the
high-dimensional energy-critical nonlinear heat equation \cite{KM2026}.
\end{remark}
\subsection{Strategy of the proof.}
\mbox{} \par 

We first present a brief outline of the paper, focusing on the proof of Theorem~\ref{main theorem}.\\

In Section~\ref{preliminaries}, we recall some coercivity estimates and then establish several technical lemmas, including estimates for spatial integrals, pointwise estimates of the nonlinear terms, and bounds on the distances between distinct bubbles. \\

The proof of Theorem~\ref{main theorem} is then divided into two main parts.

The first part spans from Section~\ref{modulation around the multi-solitons} to Section~\ref{order of the modulation parameters and the remainder terms}, where we establish the order of the modulation parameters, derive refined estimates for the remainder terms, and obtain the asymptotic ODE systems governing the modulation parameters. In Section~3,  we modulate the initial parameters and derive the modulation equations. In Section~4, we refine the choice of the modulation parameters. In Section~5, using conservation of energy, we establish energy estimates and then derive a precise relationship between the modulation parameters and the remainder terms. In Section~6, we establish control over the stable and unstable directions and show that these terms are all of order strictly smaller than the critical exponent, as anticipated. Section~7 concludes the first part of the proof.

The second part consists of Section~\ref{ode systems for the parameters}, where we analyze the asymptotic ODE system and prove that the rescaled modulation parameters converge to the set of solutions of a system of algebraic equations.

Finally, in Appendix~A, we study the structure of the solution set of the
aforementioned algebraic system. We show that for $K=2,3$, all such solutions
are isolated, while for $K\geq 4$, there may exist nontrivial continuous
families of solutions, and we provide an explicit example for $K=10$.
In Appendix~B, we show that the concentrating condition
$\mu_k(t)\to0$ in \eqref{asp} is redundant.\\

In what follows, we describe the core proof strategy and outline the key technical steps used to establish Theorem~\ref{main theorem}.\\
\textbf{Step 1: Modulation argument.} We assume that $\vec{u}(t)=(u(t),\partial_{t}u(t))$ is a solution to the equation \eqref{NLW1} which satisfies \eqref{Bub1} and \eqref{asp}. By standard modulation arguments, we obtain unique $C^{1}$ parameters $\lambda_{k}(t)$, $x_{k}(t)$ and $b_{k}(t)$ and error terms $\vec{g}(t)=(g(t),\dot{g}(t))$ with
 \begin{align*}
        & g(t):=u(t)-\sum_{k=1}^{K}W_{\lambda_{k},x_{k}},\\
         & \dot{g}(t):=\partial_{t}u(t)-\sum_{k=1}^{K}b_{k}\left(\Lambda W\right)_{\underline{\lambda_{k},x_{k}}}.
     \end{align*}
     \begin{equation*}
      \langle (\Delta \Lambda W)_{\lambda_{k},x_{k}},g\rangle=0,\
         \langle (\nabla W)_{\lambda_{k},x_{k}},g\rangle=0,\
         \langle (\Lambda W)_{\underline{\lambda_{k},x_{k}}},\dot{g}\rangle=0,\ \text{for}\ k=1,\cdots,K.
\end{equation*}
It should be noted that we additionally introduce the new parameter $b_{k}$ in the modulation procedure, which formally satisfies $b_{k}(t)\simeq -\lambda'_{k}$. The key advantage of introducing this parameter is that it enables a significant improvement of the decay properties of the remainder. Indeed, as we can see from the refined energy estimate \eqref{refined energy estimate}, the introduction of $b_{k}$ effectively reduces the energy loss.\\
Differentiating the orthogonality conditions and using the equation satisfied by $\vec{g}(t)$ give preliminary estimates for $\lambda_{k}'(t)$ and $b_{k}'(t)$ (see \eqref{lamkbk} and \eqref{bk'}). However, as one might expect, these standard arguments are  not sufficient to understand the dynamics in a useful way due to the choice of the orthogonality condition $ \langle (\Delta \Lambda W)_{\lambda_{k},x_{k}},g\rangle=0 $  and the presence of terms of critical size. In fact, since $\Delta\Lambda W\notin \text{ker} L$, the modulation parameters $\lambda_{k}(t)$ are imprecise proxies for the true dynamics. To account for such imprecision, we introduce a correction to $\lambda_{k}(t)$ as follows, defining
\begin{equation*}
         \zeta_{k}(t):=\lambda_{k}(t)-\frac{1}{\lVert \Lambda W\rVert^{2}_{L^{2}}}\bigg\langle\chi\left(\frac{\cdot-x_{k}(t)}{\lambda_{k}(t)M}\right)\left(\Lambda W\right)_{\underline{\lambda_{k},x_{k}}},g(t)\bigg\rangle,
     \end{equation*}
      where $\chi\in C^{\infty}_{c}(\RR^{5})$ satisfies $\chi\equiv 1$ on $|x|\leq 1$, ${\rm{supp}}\chi\subset \{|x|\leq 2\}$,  $M>0$ is a sufficiently large constant. To cancel the terms with critical size, we introduce a localized virial correction to $b_{k}(t)$, defining
     \begin{align*}
         p_{k}(t):=b_{k}(t)
         &-\frac{1}{\lVert\Lambda W\rVert^{2}_{L^{2}}}\frac{b_{k}(t)}{\lambda_{k}(t)}\bigg\langle \chi\left(\frac{\cdot-x_{k}(t)}{\lambda_{k}(t)M}\right)({\underline{\Lambda}}\Lambda W)_{\underline{\lambda_{k},x_{k}}},g(t)\bigg\rangle \\
         &+\frac{1}{\lVert \Lambda W\rVert^{2}_{L^{2}}}\langle \dot{g}(t),{\underline{A}}_{k}g(t)\rangle,
         \end{align*}
     where ${\underline{A}}_{k}$ is a truncated (to scale $\lambda_{k}(t)$ and center at $x_{k}(t)$) version of $\underline{\Lambda}=x\cdot \nabla+\frac{5}{2}$, the generator of $L^{2}$ scaling. The use of refined modulation parameters to obtain dynamical control of interacting bubbles for energy-critical equations was introduced by the first author in the context of a two-bubble construction for NLS in \cite{Jen2017}, see also \cite{Jen2019,Jen2025,JKL2022,JL2018,JL2022,JL2023a,JL2023b,JL2025,JM} for  other applications. Compared with existing results for wave maps and wave equations, the additional introduction of the parameter $b_{k}$ in our modulation procedure gives rise to a corresponding extra term in the corrections. The notion of local virial corrections to modulation parameters was first introduced by Rapha\"{e}l and Szeftel in \cite{RS2011} in a different context.\\
     \mbox{} \par 
\textbf{Step 2: Energy estimates.}
After completing the modulation arguments, we first use energy conservation to prove that under assumptions \eqref{Bub1} and \eqref{asp}, the energy of the solution satisfies $E(u(t),\partial_{t}u(t))= KE(W,0)$. We then expand the energy as follows:
\[
E\left(u,\partial_{t}u\right)=E\left(\vec{U}\right)+\left\langle DE\left(\vec{U}\right),\vec{g}\right\rangle+\frac{1}{2}\left\langle D^{2}E(\vec{U})\vec{g},\vec{g}\right\rangle+O\left(\lVert g\rVert^{3}_{\dot{H}^{1}}\right),
\]
where \[\displaystyle \vec{U}=\left(\sum_{k=1}^{K}W_{\lambda_{k},x_{k}},\sum_{k=1}^{K}b_{k}\left(\Lambda W\right)_{\underline{\lambda_{k},x_{k}}}\right).\]
Combining this with the coercivity estimate for the linearized operator, and up to terms that are much smaller than the critical size, we derive the following energy estimate:
\[
\lVert \vec{g}\rVert^{2}_{\dot{H}^{1}\times L^{2}}+\sum_{k=1}^{K}|b_{k}|^{2}\leq \frac{4\kappa}{3}\sum_{1\leq j<k\leq K}|z_{j}-z_{k}|^{-3}\lambda_{j}^{\frac{3}{2}}\lambda_{k}^{\frac{3}{2}}+\sum_{k=1}^{K}(a_{k}^{+})^{2}+(a_{k}^{-})^{2}.
\]
Unlike the treatments of the radially symmetric wave map problem and the kink-antikink problem in \cite{JKL2022,JL2018}, where the linearized operators of both equations admit no negative directions, the linearized operator of the energy-critical wave equation possesses negative directions. This necessitates additional control over the terms corresponding to these negative directions, which constitutes one of the core difficulties of this part. Among existing results on the rigidity characterization of energy-critical wave equations, there is no effective method suitable for controlling the negative directions in our problem. To address this, we adapt the idea of handling negative directions from the first author's work \cite{Jen2025} on the characterization of two-bubble solutions for the generalized Korteweg-de Vries (GKDV) equation,  and apply it, for the first time, to the study of energy-critical wave equations (see Section~\ref{stable and unstable directions} for details). This yields the expected estimate:
\[
\sum_{k=1}^{K}(a_{k}^{+})^{2}+(a_{k}^{-})^{2}\ll \lambda_{k}^{3}.
\]
Combining the above negative direction control with the energy estimate established earlier, we immediately deduce a priori bound on the remainder term and the modulation parameters:
\[
\lVert \vec{g}\rVert^{2}_{\dot{H}^{1}\times L^{2}}+\sum_{k=1}^{K}|b_{k}|^{2}\lesssim \sum_{k=1}^{K}\lambda_{k}^{3}.
\]
     \mbox{} \par
\textbf{Step 3: Order estimates for modulation parameters and decay of the remainder term.} With the preparations from the first two steps, we are ready to prove the first core estimate of this paper:
\[
\lambda_{k}(t)\sim t^{-2},\quad b_{k}(t)\sim t^{-3},\quad \lVert\vec{g}(t)\rVert_{\dot{H}^{1}\times L^{2}}\ll t^{-3}.
\]
The order estimates for the first two modulation parameters $\lambda_{k}(t)$ and $b_{k}(t)$ are derived by combining the previously established modified parameter evolution equations with the energy estimate. The decay estimate for the remainder term, however, relies crucially on the additional parameters $b_k(t)$ introduced during the modulation procedure. These parameters compensate for the energy loss, allowing the remainder to decay faster than the critical rate $t^{-3}$, which provides the essential foundation for our derivation of the asymptotic ODE system. In the work of the first author and Lawrie on threshold solutions for wave maps \cite{JL2018}, because of energy loss, only a lower bound on the derivative of the modified parameter $p_k(t)$ could be obtained for its evolution equation, with no upper bound available. Consequently, they could only characterize the order of the modulation parameters but could not study their convergence after appropriate normalization. 
In contrast, in our work, since the energy loss is compensated and the remainder decays faster than the critical rate $t^{-3}$, substituting this into the original modulation equations \eqref{lamkbk} and \eqref{bk'} yields the following asymptotic ODE system:
\begin{equation*}
	\left\{
	\begin{aligned}
		& \lambda'_{k}(s)+b_{k}(s)=o(s^{-3}),\\
		& b'_{k}(s)+\kappa \sum_{\substack{1\leq j\leq K \\ j\neq k}}|z_{j}-z_{k}|^{-3}\lambda_{k}^{\frac{1}{2}}(s)\lambda^{\frac{3}{2}}_{j}(s)=o(s^{-4}),
	\end{aligned}
	\right.
	\quad \text{as } s\rightarrow+\infty.
\end{equation*}
 \mbox{} \par
 \textbf{Step 4: Convergence of modulation parameters after rescaling.} Once the asymptotic orders of the modulation parameters have been established, we study their convergence after an appropriate rescaling, which is equivalent to investigating the long-time dynamical behavior of the asymptotic ODE system derived above. Building on the estimates $\lambda_{k}(s)\sim s^{-2}$ and $b_{k}(s)\sim s^{-3}$ proven in Step 3, we introduce the following change of variables to study the convergence of the parameters:
\[
\alpha_{k}=s^{2}\lambda_{k}(s),\quad \beta_{k}=s^{3}b_{k}(s),\quad s=e^{t},
\]
which transforms the original asymptotic ODE system into the following asymptotically autonomous ODE system:
\begin{equation*}
	\left\{
	\begin{aligned}
		& \alpha'_{k}(t)=2\alpha_{k}(t)-\beta_{k}(t)+o(1),\\
		& \beta'_{k}(t)=3\beta_{k}(t)-\kappa\sum_{\substack{1\leq j\leq K \\ j\neq k}}|z_{j}-z_{k}|^{-3}\alpha_{k}^{\frac{1}{2}}(t)\alpha^{\frac{3}{2}}_{j}(t)+o(1),
	\end{aligned}
	\right.
	\quad \text{as } t\rightarrow+\infty,
\end{equation*}
where the parameters satisfy $\alpha_{k}(t)\sim 1$ and $\beta_{k}(t)\sim 1$. Directly studying the convergence of solutions to this system as $t\rightarrow+\infty$ is challenging due to the insufficient decay of the remainder terms. To this end, we adapt ideas from the theory of asymptotically autonomous semiflows in dynamical systems (see for instance \cite{BHS2005,HSZ,MST}). However, we do not directly invoke any abstract theorem from this theory. Due to the special structure of the modulated PDE problem, several differences from the standard asymptotically autonomous setting arise in the proof.

First, we denote the $\omega$-limit set of a solution $X(t)$ to the above asymptotically autonomous system by
\[
\omega(X):=\{p\in\mathbb{R}^{2K}: \exists\ t_{n}\rightarrow+\infty \text{ such that } X(t_{n})\rightarrow p\}.
\]
We prove that solutions to the autonomous equation with initial data at any point of $\omega(X)$ are globally defined and generate a flow $\varphi(t,p)$ on $\omega(X)$.

Second, we establish that $\omega(X)$ is internally chain transitive with respect to the flow $\varphi(t,p)$. This step differs from existing results in the literature: most previous works assume the existence of a global semiflow, whereas our flow is not defined for all points in the full space or for all positive times. Consequently, we need to perform a more refined analysis of points that do not belong to $\omega(X)$ in the proof (see Lemma \ref{ICT}).

Finally, we use the internal chain transitivity property to establish a connection between $\omega(X)$ and the set of equilibrium points of the autonomous system. The key ingredient in this step is the construction of a Lyapunov function (see \eqref{lyapounv}). Using the fact that this Lyapunov function is strictly increasing along any non-equilibrium trajectory of the autonomous system, we conclude that $\omega(X)$ must be contained in the set of equilibrium points of the autonomous equation, as otherwise we would obtain a contradiction with the internal chain transitivity of $\omega(X)$. This completes the proof of the theorem.

\subsection{Notation}
 \mbox{} \par 
 \begin{itemize}
 \item We denote by $B(z,r)$ the ball in $\RR^{5}$ with center $z$ and radius $r\geq 0$. We also denote by $B_{H}(z,r)$ the ball in a Hilbert space $H$ with center $z$ and radius $r\geq 0$.
     \item The bracket $\langle\cdot,\cdot\rangle$ denotes the distributional pairing and the inner product in $L^2$ and $L^2\times L^2$.
\item We fix a smooth radially symmetric function $\chi\in C^{\infty}_{c}(\RR^5)$ on $\RR^5$ such that $0\leq \chi\leq 1$, $\chi(x)\equiv 1$ if $|x|\leq 1$, and $\chi(x)\equiv 0$ if $|x|\geq 2$.
\item We denote by
\[
\kappa=-15^{\frac{3}{2}}\frac{\langle\Lambda W,f'(W)\rangle}{\left\lVert\Lambda W\right\rVert^{2}_{L^2}}=\frac{3}{2}15^{\frac{3}{2}}\frac{\int W^{\frac{7}{3}}\text{d}x}{\left\lVert\Lambda W\right\rVert^{2}_{L^{2}}}=\frac{128\sqrt{5}}{7\pi}.
\]
the constant appearing in the modulation equations.
\item Let $K\geq 2$ denote the number of bubbles. We denote by
\[\boldsymbol{\lambda}=\vec{\lambda}=(\lambda_k)_{1\leq k\leq K}=(\lambda_{1},\cdots,\lambda_{K})\in\RR^K,\]
the vectors $\vec{a},\vec{b},\vec{c},\boldsymbol{\mu},\boldsymbol{b},\vec{\alpha},\vec{\beta}$ are defined in the same manner.\\
For $x_{1},\cdots,x_{K}\in\RR^{5}$, set
\[
\boldsymbol{x}=(x_{1},\cdots,x_{K})\in\RR^{5K},
\]
the vector $\boldsymbol{y}$ is defined in the same manner.
\item For $A\in\RR$ and $B\geq 0$, we write $A\lesssim B$ or $A=O(B)$ if $|A|\leq CB$ for some universal constant $C>0$. If $A,B\geq 0$, we denote $A\sim B$ if $A\lesssim B$ and $B\lesssim A$, and $A\simeq B$ if $\displaystyle\lim_{t\rightarrow+\infty}\frac{A}{B}=1$. We write $A=o(B)$ for $\displaystyle\lim_{t\rightarrow+\infty}\frac{A}{B}=0$.
 \end{itemize}

	\subsection{Acknowledgments}
    		J. Jendrej was supported by the ERC project INSOLIT (No. 101117126). L. Zhao was supported by National Natural Science Foundation of China (No. 12271497 and No. 12341102).
            
\section{Preliminaries}\label{preliminaries}
\subsection{Properties of the linearized operator}
\vspace{0.5cm}
Linearizing the system (\ref{NLW1}) around $\vec{W}=(W,0)$, one obtains
\begin{equation*}
\partial_{t}\vec{g}=J\circ{{\rm{D}}}^{2}E(\vec{W})\vec{g}=\begin{pmatrix}
        0\ &{\rm{Id}}\\
        -L\ &0
    \end{pmatrix}\vec{g}
\end{equation*}
  where $L$ is given by
    \begin{equation*}
        L g:=-\Delta g-f'(W)g=-\Delta g-\frac{7}{3}W^{\frac{4}{3}}g.
    \end{equation*}
     It is well-known  that ${\rm{ker}}L={\rm{span}}\{\Lambda W,\nabla W\}$ in $\dot{H}^{1}$ and $L$  has exactly one strictly negative simple eigenvalue, which we denote $-\nu^{2}(\nu>0)$ (see for instance \cite{DM}). We denote the corresponding positive eigenfunction  by $\mathcal{Y}$, normalized so that $\lVert\mathcal{Y}\rVert_{L^{2}}=1$. By elliptic regularity $\mathcal{Y}$ is smooth, and by Agmon estimates, it decays exponentially. Self-adjointness of $L$ implies that
     \begin{equation}
         \langle\mathcal{Y},\Lambda W\rangle=\langle\mathcal{Y},\nabla W\rangle=0.
     \end{equation}
     We define 
     \begin{equation*}
         \mathcal{Y}^{-}:=\left(\frac{1}{\nu}\mathcal{Y},-\mathcal{Y}\right),\ \mathcal{Y}^{+}:=\left(\frac{1}{\nu}\mathcal{Y},\mathcal{Y}\right),\ \alpha^{-}:=\frac{1}{2}\left(\nu \mathcal{Y},-\mathcal{Y}\right),\ \alpha^{+}:=\frac{1}{2}(\nu \mathcal{Y},\mathcal{Y}).
     \end{equation*}
     A short computation shows that
     \begin{equation*}
         J\circ{\rm{D}}^{2}E(\vec{W})\mathcal{Y}^{-}=-\nu \mathcal{Y}^{-},\quad J\circ{\rm{D}}^{2}E(\vec{W})\mathcal{Y}^{+}=\nu \mathcal{Y}^{+}
     \end{equation*}
     and for any $\vec{g}\in \dot{H}^{1}\times L^{2}$
     \begin{equation}
         \langle \alpha^{-},J\circ{\rm{D}}E(\vec{W})\vec{g}\rangle=-\nu\langle\alpha^{-},\vec{g}\rangle,\quad \langle\alpha^{+},J\circ{\rm{D}}E(\vec{W})\vec{g}\rangle=\nu\langle\alpha^{+},\vec{g}\rangle.
     \end{equation}
     Note that $\langle \alpha^{-},\mathcal{Y}^{-}\rangle=\langle \alpha^{+},\mathcal{Y}^{+}\rangle=1$ and $\langle \alpha^{-},\mathcal{Y}^{+}\rangle=\langle\alpha^{+},\mathcal{Y^{-}}\rangle=0$.\\
    For $\lambda>0,z\in\RR^{5}$, the rescaled versions of these objects are
     \[
     \mathcal{Y}^{-}_{\lambda,z}:=\left(\frac{1}{\nu}\mathcal{Y}_{\lambda,z},-\mathcal{Y}_{\underline{\lambda,z}}\right),\   \mathcal{Y}^{+}_{\lambda,z}:=\left(\frac{1}{\nu}\mathcal{Y}_{\lambda,z},\mathcal{Y}_{\underline{\lambda,z}}\right),\ 
     \]
     and
     \[\alpha^{-}_{\lambda,z}:=\frac{\nu}{2\lambda}J\mathcal{Y}^{+}_{\lambda,z}=\frac{1}{2}\left(\frac{\nu}{\lambda}\mathcal{Y}_{\underline{\lambda,z}},-\mathcal{Y}_{\underline{\lambda,z}}\right),\ \alpha^{+}_{\lambda,z}:=-\frac{\nu}{2\lambda}J\mathcal{Y}^{-}_{\lambda,z}=\frac{1}{2}\left(\frac{\nu}{\lambda}\mathcal{Y}_{\underline{\lambda,z}},\mathcal{Y}_{\underline{\lambda,z}}\right).\]
     These choices of scalings ensure that $\langle \alpha^{-}_{\lambda,z},\mathcal{Y}^{-}_{\lambda,z}\rangle=\langle \alpha^{+}_{\lambda,z},\mathcal{Y}^{+}_{\lambda,z}\rangle=1$. We have
       \begin{equation*}
       J\circ{\rm{D}}^{2}E(\vec{W}_{\lambda,z})\mathcal{Y}^{-}_{\lambda,z}=-\frac{\nu}{\lambda} \mathcal{Y}^{-}_{\lambda,z},\quad J\circ{\rm{D}}^{2}E(\vec{W}_{\lambda,z})\mathcal{Y}^{+}_{\lambda,z}=\frac{\nu}{\lambda}\mathcal{Y}^{+}_{\lambda,z}
     \end{equation*}
     and for any $\vec{h}\in \dot{H}^{1}\times L^{2}$,
     \begin{equation}\label{negative inner}
         \langle \alpha^{-}_{\lambda,z},J\circ{\rm{D}}^{2}E(\vec{W}_{\lambda,z})\vec{h}\rangle=-\frac{\nu}{\lambda}\langle\alpha^{-}_{\lambda,z},\vec{h}\rangle,\quad \langle\alpha^{+}_{\lambda,z},J\circ{\rm{D}}^{2}E(\vec{W}_{\lambda,z})\vec{h}\rangle=\frac{\nu}{\lambda}\langle\alpha^{+}_{\lambda,z},\vec{h}\rangle.
     \end{equation}
     \subsection{Coercivity estimates}
     For $g\in\dot{H}^{1}(\RR^{5})$ we have the associated quadratic form
     \[
     \langle g,Lg\rangle:=\int\limits_{\RR^{5}}(|\nabla g|^{2}-f'(W)g^{2})dx.
     \]
     We record the following  coercivity lemma from \cite{JM}.
     \begin{lemma}{\rm{(\cite[Lemma 8]{JM})}}
         There exists $\eta>0$ such that, for any $g\in\dot{H}^{1}(\RR^{5})$,
         \[
         \int\limits_{\RR^{5}}(|\nabla g|^{2}-f'(W)g^{2}){\rm{d}}x\geq \eta\lVert\nabla g\rVert^{2}_{L^{2}}-\left((\nu^{2}+1)\langle Y,g\rangle^{2}+\langle \Delta \Lambda W,g\rangle^{2}+|\langle\nabla W,g\rangle|^{2}\right).
         \]
     \end{lemma}
     \begin{lemma}{\rm{(\cite[Lemma 9]{JM})}}
         For any $\eta>0$ there exists $R=R(\eta)>0$ such that for all $g\in \dot{H}^{1}(\RR^{5})$,
         \begin{equation}\label{cor local}
             \int\limits_{|x|\leq R}|\nabla g|^{2}dx-\int\limits_{\RR^{5}}f'(W)g^{2}dx\geq-\eta \lVert\nabla g\rVert^{2}_{L^{2}}-\nu^{2}\langle Y,g\rangle^{2}.
         \end{equation}
     \end{lemma}
     For $\lambda,\mu\in(0,+\infty)$ and $x,y\in\RR^{5}$, we define
     \begin{equation}\label{distance between the parameters}
         \delta((\lambda,x),(\mu,y)):=\left|\log\left(\frac{\lambda}{\mu}\right)\right|+\frac{|x-y|}{\lambda}.
     \end{equation}
     Now we state the coercivity estimate for multiple potentials.
     \begin{lemma}{\rm{(\cite[Lemma 10]{JM})}}
         There exists $\eta>0$ such that the following holds. Let $(\lambda_{k},x_{k})\in (0,+\infty)\times \RR^{5}$ for $k=1,\cdots, K$ satisfy $\delta((\lambda_{j},x_{j}),(\lambda_{k},x_{k}))\geq \eta^{-1}$ for all $j\neq k$. Let $U\in \dot{H}^{1}(\RR^{5})$ satisfy
         \begin{equation*}
             \left\lVert U-\sum_{k=1}^{K}W_{\lambda_{k},x_{k}}\right\rVert_{\dot{H}^{1}}\leq \eta.
         \end{equation*}
         Then for any $g\in \dot{H}^{1}(\RR^{5})$,
         \begin{equation}\label{cor mul}
             \begin{aligned}
             \int\limits_{\RR^{5}}\left(|\nabla g|^{2}-f'(U)g^{2}\right){\rm{d}}x\geq &\eta\lVert\nabla g\rVert^{2}_{L^{2}}-\sum_{k=1}^{K}\{(\nu^{2}+1)\langle\lambda_{k}^{-2}Y_{\lambda_{k},x_{k}},g\rangle^{2}\\
             &+\langle \lambda_{k}^{-2}(\Delta\Lambda W)_{\lambda_{k},x_{k}},g\rangle^{2}+|\langle \lambda_{k}^{-2}(\nabla W)_{\lambda_{k},x_{k}},g\rangle|^{2} \}.
         \end{aligned}
         \end{equation}  
     \end{lemma}

\subsection{Estimates for spatial integrals}

Let $K\geq 2$ and $z_{1},\cdots,z_{K}$ be $K$ points of $\RR^{5}$ distinct two by two. Set
\begin{equation}\label{definition of d}
    d:=\frac{1}{8}\min_{j\neq k}|z_{j}-z_{k}|>0\quad {\rm{and}}\quad \boldsymbol{z}=(z_{1},\cdots,z_{K}).
\end{equation}
Throughout the paper, the  constants in various estimates are allowed to depend on $d$.
We first gather several integral estimates in the following technical lemma.
\begin{lemma}\label{WW}
 Let $W(x)$ be the ground state, and let $z,y\in\RR^{5}$ and $\lambda,\mu>0$ satisfy
    \begin{equation*}
        |z-y|\geq 2d,\quad \lambda\sim \mu,\quad \lambda,\mu\ll 1,
    \end{equation*}
  Then the following estimates hold:
\begin{equation*}
    \left|\left\langle(\Lambda W)_{\lambda,z},(\Lambda W)_{\mu,y}\right\rangle\right|+\left|\left\langle(\underline{\Lambda}\Lambda W)_{\lambda,z},(\Lambda W)_{\mu,y}\right\rangle\right|\lesssim \lambda^{\frac{3}{2}}\mu^{\frac{3}{2}},
\end{equation*}
\begin{equation*}
    \left|\left\langle(\nabla W)_{\lambda,z},(\Lambda W)_{\mu,y}\right\rangle\right|+\left|\left\langle(\nabla \Lambda W)_{\lambda,z},(\Lambda W)_{\mu,y}\right\rangle\right|+\left|\left\langle(\nabla W)_{\lambda,z},(\underline{\Lambda}\Lambda W)_{\mu,y}\right\rangle\right|\lesssim \lambda^{\frac{5}{2}}\mu^{\frac{3}{2}},
\end{equation*}
\begin{equation*}
    \left|\left\langle(\nabla W)_{\lambda,z},(\nabla W)_{\mu,y}\right\rangle\right|\lesssim\lambda^{\frac{5}{2}}\mu^{\frac{5}{2}}\quad,\left|\left\langle(\Delta \Lambda W)_{\lambda,z},(\Lambda W)_{\mu,y}\right\rangle\right|\lesssim\lambda^{\frac{7}{2}}\mu^{\frac{3}{2}},
\end{equation*}
\begin{equation*}
    \left|\left\langle(\Delta \Lambda W)_{\lambda,z},(\nabla W)_{\mu,y}\right\rangle\right|\lesssim\lambda^{\frac{7}{2}}\mu^{\frac{5}{2}},\quad\left\lVert W_{\lambda,z} W^{\frac{4}{3}}_{\mu,y}\right\rVert_{L^{2}}\lesssim\lambda^{\frac{3}{2}}\mu^{\frac{1}{2}},
\end{equation*}
\begin{equation*}
  \left|\left\langle W_{\lambda,z}^{\frac{5
  }{3}},W_{\mu,y}^{\frac{5}{3}}\right\rangle\right|\lesssim\lambda^{\frac{5}{2}}\mu^{\frac{5}{2}}|\log\lambda|,\  \left\lVert W^{2}_{\lambda,z}W^{\frac{4}{3}}_{\mu,y}\right\rVert_{L^{1}}\lesssim \lambda^{2}\mu^{2},\ \left\lVert W_{\lambda,z} W^{\frac{4}{3}}_{\mu,y}\right\rVert_{L^{\frac{10}{7}}}\lesssim\lambda^{\frac{3}{2}}\mu^{\frac{3}{2}}.
\end{equation*}
\end{lemma}
\begin{proof}
    By the definition of $W(x)$, we have the following bounds:
    \begin{align}\label{estimate on W}
       | W(x)|+|\Lambda W(x)|+|\underline{\Lambda}\Lambda W(x)|\lesssim W(x);\notag\\
       |\nabla W(x)|+|\nabla\Lambda W(x)|\lesssim W^{\frac{4}{3}}(x);\\
       |\Delta\Lambda W(x)|=|f'(W(x))\Lambda W(x)|\lesssim W^{\frac{7}{3}}(x).\notag
    \end{align}
    Here we have used the relation $L\Lambda W=0$ in the last line. It then  suffices to show \[\left|\left\langle(\Lambda W)_{\lambda,z},(\Lambda W)_{\mu,y}\right\rangle\right|\lesssim\lambda^{\frac{3}{2}}\mu^{\frac{3}{2}}.
    \] The other estimates can be obtained in the same manner. From \eqref{estimate on W}, we have
    \begin{equation*}
        \left|\left\langle(\Lambda W)_{\lambda,z},(\Lambda W)_{\mu,y}\right\rangle\right|\lesssim\int_{\RR^{5}}\frac{1}{\lambda^{\frac{3}{2}}}W\left(\frac{x-z}{\lambda}\right)\frac{1}{\mu^\frac{3}{2}} W\left(\frac{x-y}{\mu}\right)\text{d}x.
    \end{equation*}
    Now, we estimate the above integral
    in three regions respectively:
    In the region $|x-z|<d$, since $|y-z|\geq 2d$,  we have $|x-y|\geq d$, hence 
    \[\frac{1}{\mu^\frac{3}{2}} W\left(\frac{x-y}{\mu}\right)\lesssim\mu^\frac{3}{2}\left(\frac{\mu}{d}\right)^{3}\lesssim\mu^{\frac{3}{2}},\]
    and then
    \begin{equation*}
    \int_{|x-z|<d}\frac{1}{\lambda^{\frac{3}{2}}}W\left(\frac{x-z}{\lambda}\right)\frac{1}{\mu^\frac{3}{2}} W\left(\frac{x-y}{\mu}\right)\text{d}x\lesssim\mu^{\frac{3}{2}}  \int_{|x|<\frac{d}{\lambda}} \lambda^{\frac{7}{2}}W(x)\text{d} x\lesssim\mu^{\frac{3}{2}}\lambda^{\frac{3}{2}}.
    \end{equation*}
    In the region $|x-y|<d$, by the same argument as before, we have
    \begin{equation*}
        \int_{|x-y|<d}\frac{1}{\lambda^{\frac{3}{2}}}W\left(\frac{x-z}{\lambda}\right)\frac{1}{\mu^\frac{3}{2}} W\left(\frac{x-y}{\mu}\right)\text{d}x\lesssim\mu^{\frac{3}{2}}\lambda^{\frac{3}{2}}.
    \end{equation*}
    In the region $\{|x-y|\geq d\}\cup\{|x-z|\geq d\}$, by the Cauchy-Schwarz inequality, we obtain
    \begin{align*}
        &\int_{\{|x-y|\geq d\}\cup\{|x-z|\geq d\}}\frac{1}{\lambda^{\frac{3}{2}}}W\left(\frac{x-z}{\lambda}\right)\frac{1}{\mu^\frac{3}{2}} W\left(\frac{x-y}{\mu}\right)\text{d}x\\
        \lesssim&\left(\int_{|x-z|\geq d}\frac{1}{\lambda^{3}}W^{2}\left(\frac{x-z}{\lambda}\right)\text{d}x\right)^{\frac{1}{2}}\left(\int_{|x-y|\geq d}\frac{1}{\mu^{3}}W^{2}\left(\frac{x-y}{\mu}\right)\text{d}x\right)^{\frac{1}{2}}\lesssim\mu^{\frac{3}{2}}\lambda^{\frac{3}{2}}.
    \end{align*}
Combining the three estimates gives the desired bound.
\end{proof}
\subsection{Pointwise estimates} We present some pointwise inequalities in the next elementary lemma, which will be used in various places in the rest of the manuscript. 
 \begin{lemma}
     For $f(u)=|u|^{\frac{4}{3}}u$ and $F(u)=\frac{3}{10}|u|^{\frac{10}{3}}$, Let $u\in\RR$ and $\ v_{j}\geq 0$. Then
     \begin{equation}\label{nonlinear 1}
         \left|f\left(\sum_{j=1}^{K}v_{j}\right)-\sum_{j=1}^{K}f\left(v_{j}\right)\right|\lesssim\sum_{j\neq l}|v_{j}||v_{l}|^{\frac{4}{3}},
     \end{equation}
     \begin{equation}\label{nonlinear 2}
         \left|f\left(u+\sum_{j=1}^{K}v_{j}\right)-f(u)-\sum_{j=1}^{K}f(v_{j})-f'(u)\sum_{j=1}^{K}v_{j}\right|\lesssim|u|^{\frac{2}{3}}\sum_{j=1}^{K}|v_{j}|^{\frac{5}{3}}+\sum_{j\neq l}|v_{j}||v_{l}|^{\frac{4}{3}},
     \end{equation}
     \begin{equation}\label{nonlinear 3}
         \left|f\left(u+\sum_{j=1}^{K}v_{j}\right)-f\left(\sum_{j=1}^{K}v_{j}\right)-f'\left(\sum_{j=1}^{K}v_{j}\right)u\right|\lesssim|u|^{2}\sum_{j=1}^{K}|v_{k}|^{\frac{1}{3}}+|u|^{\frac{7}{3}},
     \end{equation}
     \begin{equation}\label{nonlinear 4}
         \left|f'\left(\sum_{j=1}^{K}v_{j}\right)-f'(v_{k})\right|\lesssim\left(\sum_{j\neq k}|v_{j}|\right)\sum_{j=1}^{K}|v_{j}|^{\frac{1}{3}}
     \end{equation}
     \begin{equation}\label{nonlinear 5}
         \left|F\left(\sum_{j=1}^{K}v_{j}\right)-\sum_{j=1}^{K}F\left(v_{j}\right)-\sum_{k=1}^{K}f(v_{k})\sum_{j\neq k}v_{j}\right|\lesssim\sum_{j\neq k}|v_{j}|^{2}|v_{k}|^{\frac{4}{3}}.
     \end{equation}
 \end{lemma}
 \begin{proof}
     (\ref{nonlinear 1}) and (\ref{nonlinear 2}) have already been shown in \cite{JM}. (\ref{nonlinear 3}) and (\ref{nonlinear 4}) follow directly from the standard Taylor expansion. It remains to prove (\ref{nonlinear 5}).
     We argue by induction. First, for $K=2$, we claim that for any $a,b\geq 0$
    \begin{equation}\label{K=2 in}
        \left|F\left(a+b\right)-F(a)-F(b)-f(a)b-f(b)a\right|\lesssim a^{2}b^{\frac{4}{3}}+b^{2}a^{\frac{4}{3}}.
    \end{equation}
    We may assume that both $a,b>0$, for otherwise, the inequality holds naturally. Set $\displaystyle x=\frac{b}{a}>0$; then it suffices to prove that
    \begin{equation*}
        \left|(1+x)^{\frac{10}{3}}-x^{\frac{10}{3}}-1-\frac{10}{3}x^{\frac{7}{3}}-\frac{10}{3}x\right|\lesssim x^{2}+x^{\frac{4}{3}}.
    \end{equation*}
    By continuity, it suffices to check the inequality as $x\to0^+$ and $x\to+\infty$. As $x\to0^+$, Taylor expansion gives
    \begin{align*}
        &\left|(1+x)^{\frac{10}{3}}-x^{\frac{10}{3}}-1-\frac{10}{3}x^{\frac{7}{3}}-\frac{10}{3}x\right|\\
        =&\left|1+\frac{10}{3}x+O(x^{2})-x^{\frac{10}{3}}-1-\frac{10}{3}x^{\frac{7}{3}}-\frac{10}{3}x\right|\lesssim x^{2}.
    \end{align*}
    As $x\rightarrow+\infty$, again by Taylor expansion,
    \begin{align*}
          &\left|(1+x)^{\frac{10}{3}}-x^{\frac{10}{3}}-1-\frac{10}{3}x^{\frac{7}{3}}-\frac{10}{3}x\right|\\
          =&\left|x^{\frac{10}{3}}\left(1+\frac{10}{3}\frac{1}{x}+O\left(\frac{1}{x^{2}}\right)\right)-1-\frac{10}{3}x^{\frac{7}{3}}-\frac{10}{3}x\right|\lesssim x^{\frac{4}{3}}.
    \end{align*}
    Hence, (\ref{K=2 in}) holds. For general $K\geq 3$, suppose that the conclusion holds for $K-1$. For simplicity, denote
    \begin{equation*}
        u=\sum_{j=1}^{K-1}v_{j},\  S_{K-1}=F\left(\sum_{j=1}^{K-1}v_{j}\right)-\sum_{j=1}^{K-1}F\left(v_{j}\right)-\sum_{k=1}^{K-1}f(v_{k})\sum_{j\neq k}v_{j}.
    \end{equation*}
    Then the left side of (\ref{nonlinear 5}) which we denote by $S_{K}$,  can be rewritten as 
    \begin{align*}
            S_{K}&=S_{K-1}+\left(F\left(u+v_{K}\right)-F(u)-F(v_{K})-f(u) v_{K}-f(v_{K})u\right)\\
            &+v_{K}\left(f\left(\sum_{k=1}^{K-1}v_{k}\right)-\sum_{k=1}^{K-1}f(v_{k})\right).
    \end{align*}
For the second term on the right-hand side, from (\ref{K=2 in}),
    \begin{align*}
        \left|F\left(u+v_{K}\right)-F(u)-F(v_{K})-f(u) v_{K}-f(v_{K})u\right|\lesssim v_{K}^{2}u^{\frac{4}{3}}+u^{2}v_{K}^{\frac{4}{3}}\lesssim\sum_{j\neq k}v_{j}^{2}v_{k}^{\frac{4}{3}}.
    \end{align*}
    For the third term, from (\ref{nonlinear 1}), we have
    \begin{align*}
        \left|v_{K}\left(f\left(\sum_{k=1}^{K-1}v_{k}\right)-\sum_{k=1}^{K-1}f(v_{k})\right)\right|\lesssim v_{K}\sum_{1\leq j\neq l\leq K-1}v_{j}v_{l}^{\frac{4}{3}}\lesssim\sum_{j\neq k}v_{j}^{2}v_{k}^{\frac{4}{3}}.
    \end{align*}
    Substituting the above two estimates into $S_K$ and applying the induction hypothesis, we obtain the conclusion.
 \end{proof}
 \subsection{Distance between different bubbles}
 In this subsection, we derive estimates for the distances between different bubbles, which will play a key role in the subsequent modulation argument.
\begin{lemma}[Distance between different bubbles]\label{distance between different bubbles}
Let $\lambda,\mu>0$ and $x,y\in\mathbb{R}^{5}$, and let
$\delta((\lambda,x),(\mu,y))$ be defined in \eqref{distance between the parameters}.
Then
\begin{enumerate}
\item[\textup{(1).}]
$\| W_{\lambda,x}-W_{\mu,y} \|_{\dot{H}^{1}}
\lesssim \delta((\lambda,x),(\mu,y)).$
\item[\textup{(2).}]
For any $\varepsilon>0$, there exists  $\eta(\varepsilon)>0$ such that
\[
\delta((\lambda,x),(\mu,y))\geq \varepsilon\Longrightarrow \left\lVert W_{\lambda,x}-W_{\mu,y}\right\rVert_{\dot{H}^{1}}\geq\eta(\varepsilon).
\]
\item[\textup{(3).}]
There exist $\varepsilon_{0}, C>0$ such that
\[
\| W_{\lambda,x}-W_{\mu,y} \|_{\dot{H}^{1}} \leq \varepsilon_{0}
\;\Rightarrow\;
\delta((\lambda,x),(\mu,y))
\leq C \| W_{\lambda,x}-W_{\mu,y} \|_{\dot{H}^{1}}.
\]
\end{enumerate}
\begin{proof}
    \textbf{Proof of (1).} Let
    \[
    r:=\frac{\lambda}{\mu},\qquad a:=\frac{x-y}{\mu},\qquad s:=\log r.
    \]
  It remains to prove that
    \begin{equation}\label{simple  form for the distance}
        \left\lVert W-W_{e^{s},a}\right\rVert_{\dot{H}^{1}}\lesssim |s|+|a|.
    \end{equation}
    In the region $|s|+|a|> 1$, \eqref{simple  form for the distance} holds since
    \[
      \left\lVert W-W_{e^{s},a}\right\rVert_{\dot{H}^{1}}\leq 2\left\lVert W\right\rVert_{\dot{H}^{1}}\leq2\left\lVert W\right\rVert_{\dot{H}^{1}}(|s|+|a|).
    \]
    In the region $|s|+|a|\leq 1$, observing that
    \begin{align*}
        W-W_{e^{s},a}&=-\int_{0}^{1}\frac{d}{dt}W_{e^{ts},ta} \text{d}t\\
        &=\int_{0}^{1}s(\Lambda W)_{e^{ts},ta}+a\cdot(\nabla W)_{\underline{e^{ts},ta}}\text{d}t.
    \end{align*}
    Hence,
    \begin{align*}
        \left\lVert W-W_{e^{s},a}\right\rVert_{\dot{H}^{1}}\leq \int_{0}^{1} |s|\left\lVert(\Lambda W)_{e^{ts},ta}\right\rVert_{\dot{H}^{1}}+|a|\left\lVert(\nabla W)_{\underline{e^{ts},ta}}\right\rVert_{\dot{H}^{1}}\text{d}t\lesssim|a|+|s|.
    \end{align*}
    \textbf{Proof of (2).} By rescaling, we only need to prove that
    \[|\log r|+|a|\geq \varepsilon\Longrightarrow \left\lVert W-W_{r,a}\right\rVert_{\dot{H}^{1}}\geq\eta(\varepsilon).\]
    We argue by contradiction and suppose that there exist $\varepsilon>0$ and a sequence $(s_{n},a_{n})$ such that
    \begin{equation}\label{contradiction assumption}
        |\log r_{n}|+|a_{n}|\geq \varepsilon,\qquad \left\lVert W-W_{r_{n},a_{n}}\right\rVert_{\dot{H}^{1}}\rightarrow0.
    \end{equation}
    Passing to a subsequence if necessary, we claim that $(r_{n},a_{n})\rightarrow(r^{*},a^{*})$ as $n\rightarrow+\infty$ for some $r^{*}>0$ and $a^*\in\RR^{5}$. Otherwise, one of the following three possibilities must hold
    \[r_{n}\rightarrow0,\quad \text{or}\quad r_{n}\rightarrow+\infty,\quad \text{or}\quad|a_{n}|\rightarrow+\infty,\quad \text{as}\quad n\rightarrow+\infty.\]
    Now, we show that each of the above cases implies
    \[W_{r_{n},a_{n}}\rightharpoonup0\quad \text{in}\ \dot{H}^{1}\quad \text{as}\quad n\rightarrow+\infty,\]
    which contradicts the fact that $W_{r_{n},a_{n}}\rightarrow W$ in $\dot{H}^{1}$.\\
    For any $\phi\in C_c^\infty(\RR^5)$ and integrating by parts, we have
    \[\left\langle W_{r_{n},a_{n}},\phi\right\rangle_{\dot{H}^{1}}=\int \nabla W_{r_{n},a_{n}}\cdot\nabla \phi=-\int\Delta W_{r_{n},a_{n}}\phi=\int W_{r_{n},a_{n}}^{\frac{7}{3}}\phi.\]
   Suppose $\text{supp}\ \phi\subset B(0,R)$, then 
   \begin{equation}\label{preliminary estimate}
       \left|\left\langle W_{r_{n},a_{n}},\phi\right\rangle_{\dot{H}^{1}}\right|\leq \left\lVert\phi\right\rVert_{\infty}\int_{|x|\leq R}W_{r_{n},a_{n}}^{\frac{7}{3}}(x)\text{d}x.
   \end{equation}
   Next, we discuss each of the three aforementioned cases separately.\\
   If $r_{n}\rightarrow0$, by a change of variables, we have
   \[\left|\left\langle W_{r_{n},a_{n}},\phi\right\rangle_{\dot{H}^{1}}\right|\leq \left\lVert\phi\right\rVert_{\infty}r_{n}^{\frac{3}{2}}\int_{\RR^{5}}W^{\frac{7}{3}}(x)\text{d}x\rightarrow 0.\]
   If $r_{n}\rightarrow+\infty$, by the definition of $W$, we have $W(x)\leq 1$ for all $x\in\RR^{5}$, and thus, by \eqref{preliminary estimate}
   \[\left|\left\langle W_{r_{n},a_{n}},\phi\right\rangle_{\dot{H}^{1}}\right|\leq \left\lVert\phi\right\rVert_{\infty}\left|B(0,R)\right|r_{n}^{-\frac{7}{2}}\rightarrow 0.\]
   If $|a_{n}|\rightarrow+\infty$, then for $n$ large enough and any $|x|\leq R$,  $|x-a_{n}|\geq |a_{n}|-R$. Using the rough bound
   \[W^{\frac{7}{3}}(x)\lesssim (1+|x|)^{-7},\]
    we have
    \begin{equation*}
        W_{r_{n},a_{n}}^{\frac{7}{3}}(x)\lesssim r_{n}^{-\frac{7}{2}}\left(1+\frac{|a_{n}|-R}{r_{n}}\right)^{-7}\lesssim\left(|a_{n}|-R\right)^{-\frac{7}{2}}.
    \end{equation*}
    Hence,
    \[ \left|\left\langle W_{r_{n},a_{n}},\phi\right\rangle_{\dot{H}^{1}}\right|\lesssim\left\lVert\phi\right\rVert_{\infty}|B(0,R)|(|a_{n}|-R)^{-\frac{7}{2}}\rightarrow0.\]
   Now, since $(r_{n},a_{n})\rightarrow(r^{*},a^{*})$ as $n\rightarrow+\infty$, we have
   \begin{equation*}
       W_{r_{n},a_{n}}\rightarrow W_{r^*,a^*}\quad \text{in}\ \dot{H}^{1},
   \end{equation*}
   which together with \eqref{contradiction assumption} gives that
   \begin{equation*}
       W(x)=W_{r^*,a^*}(x)\quad \text{in}\ \dot{H}^{1}.
   \end{equation*}
   Therefore, $r^*=1$, $a^*=0$ and hence
   \[|\log r_{n}|+|a_{n}|\rightarrow 0\]
   which contradicts \eqref{contradiction assumption}.\\
   \textbf{Proof of (3).} Using the notations above, we only need to prove that
   \[
   \left\lVert W-W_{r,a}\right\rVert_{\dot{H}^{1}}\leq\varepsilon_{0}\Longrightarrow |\log r|+|a|\leq C \left\lVert W-W_{r,a}\right\rVert_{\dot{H}^{1}}
   \]
   holds  for some $\varepsilon_{0}>0$ and $C>0$. Set
   \[U_{\rho}:=\{(s,a):|s|+|a|<\rho\},\qquad s=\log r,\]
   where $\rho>0$ will be chosen sufficiently small later. By the previous proof, there exists $\varepsilon_{0}(\rho)>0$ such that
   \[
   \left\lVert W-W_{e^{s},a}\right\rVert_{\dot{H}^{1}}\leq \varepsilon_{0}(\rho)\Longrightarrow(s,a)\in U_{\rho}.
   \]
   Consider the smooth function
   \[
   \mathcal{G}(s,a):=\left(\left\langle\Delta\Lambda W,W-W_{e^{s},a}\right\rangle,\left\langle\nabla W,W-W_{e^{s},a}\right\rangle\right)\in\RR^{6}
   \]
   A direct computation gives $\mathcal{G}(0,0)=0$ and
   \[
   D_{s,a}\mathcal{G}\upharpoonright_{(0,0)}=\text{diag}\left(-\lVert\nabla\Lambda W\rVert^{2}_{L^{2}},\frac{1}{5}\lVert\nabla W\rVert^{2}_{L^{2}}I_{5}\right).
   \]
By the local Inverse Function Theorem, there exist $\rho>0$ and $C>0$ such that for all $(s,a)\in U_{\rho}$, we have
   \begin{equation*}
      |s|+|a|\leq C\left|\mathcal{G}(s,a)\right|.
   \end{equation*}
   On the other hand, by the definition of $\mathcal{G}$ and H\"{o}lder's inequality,
   \[\left|\mathcal{G}(s,a)\right|\lesssim \left\lVert W-W_{e^{s},a}\right\rVert_{\dot{H}^{1}}.\]
   Hence, for any $(s,a)\in U_{\rho}$, we obtain
   \begin{equation}\label{bounds controlled by the distance of bubble}
          |s|+|a|\lesssim\left\lVert W-W_{e^s,a}\right\rVert_{\dot{H}^{1}}
   \end{equation}
Finally, for $\rho$ as above,  take $\varepsilon_{0} = \varepsilon_{0}(\rho)$,  which ensures that
\[
\left\lVert W - W_{e^{s},a} \right\rVert_{\dot{H}^{1}} \leq \varepsilon_{0} \implies (s,a) \in U_{\rho},
\]
then, using \eqref{bounds controlled by the distance of bubble}, we complete the proof.
\end{proof}
\end{lemma}
     \section{Modulation around the multi-solitons}\label{modulation around the multi-solitons}
     In this section, we give a careful analysis of the modulation equations that govern the evolution of $K$-bubble configurations.
     \begin{lemma}\label{Mod}
     There exist $\delta_0, C_0 > 0$ such that the following is true. Let $J\subset\RR $ be a time interval,  and let $\vec{u}(t)$ be a solution to (\ref{NLW1}) on $J$. Assume that there exist  $\mu_{k}(t)>0$, $y_{k}(t)\in\RR^{5}$ $(1\leq k\leq K)$, and $0<\delta\leq\delta_{0}$ such that
     \begin{equation}\label{inital assumption}
         \bigg\lVert u(t)-\sum_{k=1}^{K}W_{\mu_{k},y_{k}}\bigg\rVert_{\dot{H}^{1}}+\lVert \partial_{t}u\rVert_{L^{2}}+\sum_{k=1}^{K}\mu_{k}+\sum_{k=1}^{K}|y_{k}-z_{k}|<\delta
     \end{equation}
     and 
     \begin{equation}
         C_{1}\leq \frac{\mu_{k}}{\mu_{j}}\leq C_{2}
     \end{equation}
     for all $t\in J$ and $1\leq j,k\leq K$, where $C_{1},C_{2}>0$ are some fixed constants.
Then there exist unique $C^{1}(J)$ functions $\lambda_{k}>0, x_{k}\in\RR^{5},b_{k}\in\RR$ $(1\leq k\leq K)$ such that, defining $\vec{g}(t)=(g(t),\dot{g}(t))$ by
     \begin{align*}
        & g(t):=u(t)-\sum_{k=1}^{K}W_{\lambda_{k},x_{k}},\\
         & \dot{g}(t):=\partial_{t}u(t)-\sum_{k=1}^{K}b_{k}\left(\Lambda W\right)_{\underline{\lambda_{k},x_{k}}}.
     \end{align*}
     we have, for each $t\in J$, 
     \begin{equation}\label{ort1}
          \langle (\Delta \Lambda W)_{\lambda_{k},x_{k}},g\rangle=0,
     \end{equation}
     \begin{equation}\label{ort2}
         \langle (\nabla W)_{\lambda_{k},x_{k}},g\rangle=0,
     \end{equation}
     \begin{equation}\label{ort3}
         \langle (\Lambda W)_{\underline{\lambda_{k},x_{k}}},\dot{g}\rangle=0,
     \end{equation}
     \begin{equation}\label{lamklamj}
         \frac{C_{1}}{2}\leq \frac{\lambda_{k}}{\lambda_{j}}\leq 2C_{2},\quad \forall\ 1\leq j,k\leq K,
     \end{equation}
     and 
     \begin{equation}\label{rough estimate}
         \lVert g\rVert_{\dot{H}^{1}}+\lVert \dot{g}\rVert_{L^{2}}+\sum_{k=1}^{K}\lambda_{k}+\sum_{k=1}^{K}|x_{k}-z_{k}|+\sum_{k=1}^{K}|b_{k}|\leq C_{0}\delta.
     \end{equation}
     \end{lemma}
     \begin{remark}\label{implicit function theorem}
         We will use the following, less-standard version of the implicit function theorem in the proof of Lemma \ref{Mod}.\\
         Let $X, Y, Z$ be Banach spaces and let $(x_0, y_0) \in X \times Y$, and $\delta_1, \delta_2 > 0$. Consider a mapping $G : B(x_0, \delta_1) \times B(y_0, \delta_2) \to Z$, continuous in $x$ and $C^1$ in $y$. Suppose that $G(x_0, y_0) = 0$ and $(D_y G)(x_0, y_0) =: L_0$ has bounded inverse $L_0^{-1}$. Suppose in addition that
\begin{align}\label{implicit function}
\|L_0 - D_y G(x, y)\|_{\mathcal{L}(Y,Z)} &\leq \frac{1}{3\|L_0^{-1}\|_{\mathcal{L}(Z,Y)}}, \\
\|G(x, y_0)\|_Z &\leq \frac{\delta_2}{3\|L_0^{-1}\|_{\mathcal{L}(Z,Y)}} \notag
\end{align}
for all $\|x - x_0\|_X \leq \delta_1$ and $\|y - y_0\|_Y \leq \delta_2$. Then there exists a continuous function $\varsigma : B(x_0, \delta_1) \to B(y_0, \delta_2)$ such that for all $x \in B(x_0, \delta_1)$, $y = \varsigma(x)$ is the unique solution of $G(x, \varsigma(x)) = 0$ in $B(y_0, \delta_2)$.\\
This statement is proved in the same fashion as the usual implicit function theorem; see, e.g., \cite{CH}. The key point is that the bounds \eqref{implicit function} give uniform control on the size of the open set on which the Banach contraction mapping theorem can be applied.
     \end{remark}
     \begin{proof}
         We sketch the proof.\\
         \textbf{Step 1: Existence and local uniqueness of the parameters $\boldsymbol{\lambda_{k}}$, $\boldsymbol{x_{k}}$, $\boldsymbol{b_{k}}$.} We suppress the dependence of all functions on $t\in J$ below as no constants appearing in the proof will depend on $t\in J$. We denote
         \begin{align}\label{g0,dotg0}
              &g_{0}:=u-\sum_{k=1}^{K}W_{\mu_{k},y_{k}},\notag\\
         &\dot{g}_{0}:=\partial_{t}u.
         \end{align}
         Using the assumption \eqref{inital assumption}, we have
         \begin{equation}\label{initial estimate}
             \left\lVert \left(g_{0},\dot{g}_{0}\right)\right\rVert_{\dot{H}^{1}\times L^2}+\sum_{k=1}^{K}\mu_{k}+\sum_{k=1}^{K}|y_{k}-z_{k}|<\delta.
         \end{equation}
         Define the function $\boldsymbol{F}=(F,\dot{F})$ by
         \begin{align*} &F(h,\dot{h},\boldsymbol{\lambda},\boldsymbol{x},\boldsymbol{b}):=h+\sum_{k=1}^{K}W_{\mu_{k},y_{k}}-\sum_{k=1}^{K}W_{\lambda_{k},x_{k}},\\
        &\dot{F}(h,\dot{h},\boldsymbol{\lambda},\boldsymbol{x},\boldsymbol{b}):=\dot{h}-\sum_{k=1}^{K}b_{k}\left(\Lambda W\right)_{\underline{\lambda_{k},x_{k}}}.
         \end{align*}
         Note that $\boldsymbol{F}(0,0,\boldsymbol{\mu},\boldsymbol{y},\boldsymbol{0})=(0,0)$. Next, define the function $G$ by
         \begin{align*}
&G(h,\dot{h},\boldsymbol{\lambda},\boldsymbol{x},\boldsymbol{b})=\left(
\left(\frac{1}{\lambda_k}\left\langle\left(\Delta \Lambda
W\right)_{\underline{\lambda_k,x_k}},F(h,\dot{h},\boldsymbol{\lambda},\boldsymbol{x},\boldsymbol{b})\right\rangle\right)_{k=1}^K,
\right.\\
&\left.
\left(\frac{1}{\mu_k}\left\langle\left(\nabla
W\right)_{\underline{\lambda_k,x_k}},
F(h,\dot{h},\boldsymbol{\lambda},\boldsymbol{x},\boldsymbol{b})
\right\rangle\right)_{k=1}^K,
\left(\left\langle\left(\Lambda
W\right)_{\underline{\lambda_k,x_k}},\dot{F}
(h,\dot{h},\boldsymbol{\lambda},\boldsymbol{x},\boldsymbol{b})\right\rangle
\right)_{k=1}^K
\right),
\end{align*}
We note that $G(0,0,\boldsymbol{\mu},\boldsymbol{y},\boldsymbol{0})=(\boldsymbol{0},\boldsymbol{0},\boldsymbol{0})$. We now introduce new variables
\[
\ell_{k}:=\log \lambda_{k},\quad \ell_{k}^{in}:=\log \mu_{k},\quad \boldsymbol{\ell}:=(\ell_{1},\cdots,\ell_{K}),\quad \boldsymbol{\ell^{in}}:=(\ell_{1}^{in},\cdots,\ell_{K}^{in}),
\]
\[
\tilde{x}_{k}:=\frac{1}{\mu_{k}}x_{k},\quad \tilde{y}_{k}:=\frac{1}{\mu_{k}}y_{k},\quad \boldsymbol{\tilde{x}}:=(\tilde{x}_{1},\cdots,\tilde{x}_{K}),\quad \boldsymbol{\tilde{y}}:=(\tilde{y}_{1},\cdots,\tilde{y}_{K}).
\]
Set
\[
\tilde{G}(h,\dot{h},\boldsymbol{\ell},\boldsymbol{\tilde{x}},\boldsymbol{b})=G(h,\dot{h},\boldsymbol{\lambda},\boldsymbol{x},\boldsymbol{b}),\quad \boldsymbol{\tilde{F}}(h,\dot{h},\boldsymbol{\ell},\boldsymbol{\tilde{x}},\boldsymbol{b})=\boldsymbol{F}(h,\dot{h},\boldsymbol{\lambda},\boldsymbol{x},\boldsymbol{b}).
\]
We now check that the conditions \eqref{implicit function} are satisfied for $x_{0}=(0,0)\in\dot{H}^{1}\times L^{2}$, $y_{0}=(\boldsymbol{\ell^{in}},\boldsymbol{\tilde{y}},\boldsymbol{0})\in\RR^{7K}$ and $\tilde{G}:B_{\dot{H}^{1}\times L^{2}}(0,2\delta)\times B_{\RR^{7K}}((\boldsymbol{\ell^{in}},\boldsymbol{\tilde{y}},\boldsymbol{0}),C_{1}\delta)\rightarrow\RR^{7K}$, for $0<\delta_{1}=2\delta<2\delta_{0}$ small enough and $C_{1}$ a uniform constant. 
Note that in the new variables
\[
\lambda_{k}\partial_{\lambda_{k}}=\partial_{\ell_{k}},\qquad \partial_{i}^{\tilde{x}_{k}}=\mu_{k}\partial_{i}^{x_{k}}.
\]
We define $L_{0}:=D_{\boldsymbol{\ell,\tilde{x},b}}\tilde{G}\upharpoonright_{(0,0,\boldsymbol{\ell^{in}},\boldsymbol{\tilde{y}},\boldsymbol{0})}$. 
A straightforward computation using \eqref{initial estimate} and Lemma \eqref{WW} gives
\begin{equation*}
    L_{0}=\text{diag}\left(-\left\lVert\nabla\Lambda W\right\rVert^{2}_{L^{2}}I_{K},\frac{1}{5}\left\lVert\nabla W\right\rVert^{2}_{L^{2}}I_{5K},-\lVert\Lambda W\rVert^{2}_{L^{2}}I_{K}\right)+O(\delta).
\end{equation*}
The matrix is diagonally dominant for $0<\delta<\delta_{0}$ small enough and hence invertible.\\
The second condition in \eqref{implicit function} is clear since $\boldsymbol{\tilde{F}}(h,\dot{h},\boldsymbol{\ell^{in}},\boldsymbol{\tilde{y}},\boldsymbol{0})=(h,\dot{h})$ and hence
\begin{align*}
    \left|\tilde{G}(h, \dot{h}, \boldsymbol{\ell^{in}}, \boldsymbol{\tilde{y}}, \boldsymbol{0})\right|
    &= \left| \Biggl(
        \left( \frac{1}{\mu_k} \left\langle \left( \Delta \Lambda W \right)_{\underline{\mu_k, y_k}}, h \right\rangle \right)_{k=1}^K,
        \left( \frac{1}{\mu_k} \left\langle \left( \nabla W \right)_{\underline{\mu_k, y_k}}, h \right\rangle \right)_{k=1}^K,
    \right. \\
    &\phantom{= \left| \Biggl( \right.}
    \left. \left( \left\langle \left( \Lambda W \right)_{\underline{\mu_k, y_k}}, \dot{h} \right\rangle \right)_{k=1}^K
    \Biggr) \right|
    \lesssim \lVert h \rVert_{\dot{H}^1} + \lVert \dot{h} \rVert_{L^2}.
\end{align*}
For the first estimate in \eqref{implicit function}, from \eqref{initial estimate}, Lemma \ref{WW} and Lemma \ref{distance between different bubbles}, we have
\begin{align*}
    \left|(L_{0})_{ij}-(D_{\boldsymbol{\ell,\tilde{x},b}}\tilde{G})_{ij}\right|\lesssim \lVert h \rVert_{\dot{H}^{1}}&+\lVert \dot{h}\rVert_{L^{2}}+\sum_{j\neq k}\lambda_{k}^{\frac{1}{2}}\lambda_{j}^{\frac{1}{2}}\\
    &+\left|\boldsymbol{b}\right|+\left|\boldsymbol{\ell}-\boldsymbol{\ell^{in}}\right|+\left|\boldsymbol{\tilde{y}}-\boldsymbol{\tilde{x}}\right|\ll 1.
\end{align*}
provided that $\delta_{0}>0$ is chosen sufficiently small. Applying Remark \ref{implicit function theorem}, we obtain a continuous mapping $\varsigma:B_{\dot{H}^{1}\times L^{2}}(0,2\delta)\rightarrow B_{\RR^{7K}}((\boldsymbol{\ell^{in}},\boldsymbol{\tilde{y}},\boldsymbol{0}),C_{1}\delta)$ so that for all $(h,\dot{h},\boldsymbol{\ell},\boldsymbol{\tilde{x}},\boldsymbol{b})\in B_{\dot{H}^{1}\times L^{2}}(0,2\delta)\times B_{\RR^{7K}}((\boldsymbol{\ell^{in}},\boldsymbol{\tilde{y}},\boldsymbol{0}),C_{1}\delta) $ we have
\begin{equation*}
G(h,\dot{h},\boldsymbol{\lambda},\boldsymbol{x},\boldsymbol{b})=0\Longleftrightarrow(\boldsymbol{\ell},\boldsymbol{\tilde{x}},\boldsymbol{b})=\varsigma (h,\dot{h}).
\end{equation*}
We observe that if we let $(g_{0},\dot{g}_{0})$ be as in \eqref{g0,dotg0}, and define $\lambda_{k}=e^{\ell_{k}},\ x_{k}=\mu_{k}\tilde{x}_{k}$ and $(g,\dot{g})\in\dot{H}^{1}\times L^{2}$ by
\[(\boldsymbol{\ell},\boldsymbol{\tilde{x}},\boldsymbol{b}):=\varsigma(g_{0},\dot{g}_{0}),\]
\[g:=F(g_{0},\dot{g}_{0},\boldsymbol{\lambda},\boldsymbol{x},\boldsymbol{b})=u-\sum_{k=1}^{K}W_{\lambda_{k},x_{k}},\]
\[\dot{g}:=\dot{F}(g_{0},\dot{g}_{0},\boldsymbol{\lambda},\boldsymbol{x},\boldsymbol{b})=\partial_{t}u-\sum_{k=1}^{K}b_{k}\left(\Lambda W\right)_{\underline{\lambda_{k},x_{k}}}.\]
Then 
\begin{equation*}
      \langle (\Delta \Lambda W)_{\lambda_{k},x_{k}},g\rangle=0,\
         \langle (\nabla W)_{\lambda_{k},x_{k}},g\rangle=0,\
         \langle (\Lambda W)_{\underline{\lambda_{k},x_{k}}},\dot{g}\rangle=0.
\end{equation*}
  Moreover, since $|\boldsymbol{\ell}-\boldsymbol{\ell^{in}}|+|\boldsymbol{\tilde{x}}-\boldsymbol{\tilde{y}}|+|\boldsymbol{b}|\lesssim \delta$, we have
    \begin{equation*}
        \left|\frac{\lambda_{k}}{\mu_{k}}-1\right|+\frac{\left|x_{k}-y_{k}\right|}{\mu_{k}}\lesssim\delta\ll 1,
    \end{equation*}
    \begin{equation*}
        \lVert(g,\dot{g})\rVert_{\dot{H}^{1}\times L^{2}}\lesssim \lVert(g_{0},\dot{g}_{0})\rVert_{\dot{H}^{1}\times L^{2}}+|\boldsymbol{\ell}-\boldsymbol{\ell^{in}}|+|\boldsymbol{\tilde{x}}-\boldsymbol{\tilde{y}}|+|\boldsymbol{b}|\lesssim \delta.
    \end{equation*}
    In particular, we have
    \begin{equation*}
        \frac{C_{1}}{2}\leq \frac{\lambda_{k}}{\lambda_{j}}\leq 2C_{2},\quad \forall\ 1\leq j,k\leq K,
    \end{equation*}
    and
    \begin{equation*}
         \lVert g\rVert_{\dot{H}^{1}}+\lVert \dot{g}\rVert_{L^{2}}+\sum_{k=1}^{K}\lambda_{k}+\sum_{k=1}^{K}|x_{k}-z_{k}|+\sum_{k=1}^{K}|b_{k}|\lesssim\delta,
    \end{equation*}
    which is \eqref{lamklamj} and \eqref{rough estimate}.\\
    \textbf{Step 2: Uniqueness of the modulation parameters.} Let $(\boldsymbol{\lambda^1},\boldsymbol{x^1},\boldsymbol{b^1})$ be another parameters such that \eqref{ort1},\eqref{ort2},\eqref{ort3},\eqref{lamklamj} and \eqref{rough estimate} hold. We show that $(\boldsymbol{\lambda},\boldsymbol{x},\boldsymbol{b})=(\boldsymbol{\lambda^1},\boldsymbol{x^1},\boldsymbol{b^1})$ for $\delta_{0}>0$ small enough. By the local uniqueness of the parameters in Step 1, it suffices to show that $(\boldsymbol{\ell^{1}},\boldsymbol{\tilde{x}^1},\boldsymbol{b^1})\in B_{\RR^{7K}}((\boldsymbol{\ell^{in}},\boldsymbol{\tilde{y}},\boldsymbol{0}),C_{1}\delta)  $. From  \eqref{rough estimate}, we have
    \begin{equation}\label{anthor parameters}
       \lambda_{k}^1\lesssim \delta,\qquad |x_{k}^1-z_{k}|\lesssim\delta,\qquad |b_{k}^1|\lesssim\delta,
    \end{equation}
    and
 \begin{align}\label{distance between multi-bubble}
  \left\lVert   \sum_{k=1}^{K}W_{\mu_{k},y_{k}}-\sum_{k=1}^{K}W_{\lambda_{k}^1,x_{k}^1}\right\rVert_{\dot{H}^{1}}&\leq \left\lVert u-\sum_{k=1}^KW_{\mu_{k},y_{k}}\right\rVert_{\dot{H}^{1}}+\left\lVert u-\sum_{k=1}^{K}W_{\lambda_{k}^1,x_{k}^1}\right\rVert_{\dot{H}^{1}}\notag\\
  &=\left\lVert g_{0}\right\rVert_{\dot{H}^{1}}+\left\lVert g^1\right\rVert_{\dot{H}^{1}}\lesssim \delta.
 \end{align}
 We first claim that, for $\delta_{0}>0$ sufficiently small,
 \begin{equation}\label{distance between single bubble}
     \left\lVert W_{\mu_{k},y_{k}}-W_{\lambda_{k}^1,x_{k}^1}\right\rVert\lesssim \delta,\quad \text{for all}\ 1\leq k\leq K.
 \end{equation}
 Take $\displaystyle\chi_{k}(x):=\chi\left(\frac{x-z_{k}}{d}\right)$, 
 where $\chi\in C^{\infty}_{c}(\RR^{5})$ satisfies $\chi\equiv 1$ on $|x|\leq 1$, ${\rm{supp}}\chi\subset \{|x|< 2\}$, and $d$ is defined by \eqref{definition of d}. Then we have
\begin{align*}
    \left\lVert W_{\mu_{k},y_{k}} - W_{\lambda_{k}^1,x_{k}^1} \right\rVert_{\dot{H}^{1}}
    &\leq \left\lVert (1-\chi_{k}) \left(W_{\mu_{k},y_{k}} - W_{\lambda_{k}^1,x_{k}^1} \right) \right\rVert_{\dot{H}^{1}} \\
    &\quad + \left\lVert \chi_{k} \left(W_{\mu_{k},y_{k}} - W_{\lambda_{k}^1,x_{k}^1} \right) \right\rVert_{\dot{H}^1}
\end{align*}
For the first term on the right side, using the fact that $\chi_{k}\equiv 1$ on $|x-z_{k}|\leq d$ and 
\[|x-z_{k}|\leq |x-y_{k}|+|y_{k}-z_{k}|\leq \frac{1}{2}d+C\delta\leq d,\]
for all $\displaystyle |x-y_{k}|\leq \frac{d}{2}$ and $\delta>0$ small enough, we have
\begin{align*}
    \left\lVert(1-\chi_{k})W_{\mu_{k},y_{k}}\right\rVert_{\dot{H}^{1}}\lesssim\left\lVert W_{\mu_{k},y_{k}}\right\rVert_{L^{2}(|x-y_{k}|\geq\frac{d}{2})}+\left\lVert\left(\nabla W\right)_{\underline{\mu_{k},y_{k}}}\right\rVert_{L^{2}(|x-y_{k}|\geq \frac{d}{2})}\lesssim\mu_{k}^{\frac{3}{2}}\ll \delta.
\end{align*}
By similar computations for the term $\left\lVert(1-\chi_{k})W_{\lambda_{k}^1,x_{k}^1}\right\rVert_{\dot{H}^{1}}$, we have
\begin{equation*}
    \left\lVert (1-\chi_{k}) \left(W_{\mu_{k},y_{k}} - W_{\lambda_{k}^1,x_{k}^1} \right) \right\rVert_{\dot{H}^{1}}\ll \delta.
\end{equation*}
For the second term on the right side, we first note that
\begin{align*}
     \left\lVert \chi_{k} \left(W_{\mu_{k},y_{k}} - W_{\lambda_{k}^1,x_{k}^1} \right) \right\rVert_{\dot{H}^1}\leq&\left\lVert\chi_{k}\left( \sum_{k=1}^{K}W_{\mu_{k},y_{k}}-\sum_{k=1}^{K}W_{\lambda_{k}^1,x_{k}^1}\right)\right\rVert_{\dot{H}^{1}}\\
     &+\left\lVert\chi_{k}\left(\sum_{j\neq k}W_{\mu_{j},y_{j}}-\sum_{j\neq k} W_{\lambda_{j}^1,x_{j}^1}\right)\right\rVert_{\dot{H}^{1}}
\end{align*}
For the first line, by the definition of $\chi_{k}$ and \eqref{distance between multi-bubble},
\begin{equation*}
    \left\lVert\chi_{k}\left( \sum_{k=1}^{K}W_{\mu_{k},y_{k}}-\sum_{k=1}^{K}W_{\lambda_{k}^1,x_{k}^1}\right)\right\rVert_{\dot{H}^{1}}\lesssim   \left\lVert   \sum_{k=1}^{K}W_{\mu_{k},y_{k}}-\sum_{k=1}^{K}W_{\lambda_{k}^1,x_{k}^1}\right\rVert_{\dot{H}^{1}}\lesssim\delta.
\end{equation*}
For the second line, we observe that for any $j\neq k$, if $x\in\text{supp}\ \chi_{k} \subset B(z_{k},2d)$, then
\[
|x-y_{j}|\geq |z_{k}-z_{j}|-|x-z_{k}|-|y_{j}-z_{j}|\geq 4d,
\]
for $\delta>0$ small enough. Hence,
\begin{equation*}
\left\lVert\chi_{k}W_{\mu_{j},y_{j}}\right\rVert_{\dot{H}^{1}}\lesssim\left\lVert W_{\mu_{j},y_{j}}\right\rVert_{L^{2}(|x-y_{j}|\geq4d)}+\left\lVert\left(\nabla W\right)_{\underline{\mu_{j},y_{j}}}\right\rVert_{L^{2}(|x-y_{j}|\geq 4d)}\lesssim \mu_{j}^{\frac{3}{2}}\ll \delta,
\end{equation*}
The same computation for the other terms on the second line gives
\begin{equation*}
    \left\lVert\chi_{k}\left(\sum_{j\neq k}W_{\mu_{j},y_{j}}-\sum_{j\neq k} W_{\lambda_{j}^1,x_{j}^1}\right)\right\rVert_{\dot{H}^{1}}\ll \delta.
\end{equation*}
Combining the above estimates, we obtain \eqref{distance between single bubble}. Now, by Lemma \ref{distance between different bubbles}, for $\delta>0$ small enough, we have
\begin{equation*}
    \left|\log \left(\frac{\lambda_{k}^{1}}{\mu_{k}}\right)\right|+\frac{|x_k^1-y_k|}{\mu_{k}}\lesssim\left\lVert W_{\mu_{k},y_{k}}-W_{\lambda_{k}^1,x_{k}^1}\right\rVert\lesssim\delta\quad \text{for all }\ 1\leq k\leq K,
\end{equation*}
which together with \eqref{anthor parameters} gives that
\begin{equation*}
    \left|\boldsymbol{\ell^1}-\boldsymbol{\ell^{in}}\right|+\left|\boldsymbol{\tilde{x}^{1}}-\boldsymbol{\tilde{y}}\right|+\left|\boldsymbol{b}^1\right|\lesssim\delta.
\end{equation*}
Shrinking $\delta>0$ if necessary, we obtain $(\boldsymbol{\ell^{1}},\boldsymbol{\tilde{x}^1},\boldsymbol{b^1})\in B_{\RR^{7K}}((\boldsymbol{\ell^{in}},\boldsymbol{\tilde{y}},\boldsymbol{0)},C_{1}\delta)$ as desired.\\
\textbf{Step 3: $\boldsymbol{C^{1}}$ regularity of the parameters.}
The $C^{1}$ regularity of $\lambda_{k}(t)$, $x_{k}(t)$ and $b_{k}(t)$ can be obtained by the same method as in \cite[Remark 3.13]{JL2018}, and we therefore omit the details here.
     \end{proof}
     Inserting  
     \begin{align*}
        & g(t):=u(t)-\sum_{k=1}^{K}W_{\lambda_{k},x_{k}},\\
         & \dot{g}(t):=\partial_{t}u(t)-\sum_{k=1}^{K}b_{k}\left(\Lambda W\right)_{\underline{\lambda_{k},x_{k}}}.
     \end{align*}
   into  equation (\ref{NLW1}) and using the relation $\Delta W_{\lambda_{k},x_{k}}+f\left(W_{\lambda_{k},x_{k}}\right)=0$, we obtain the equation for $\vec{g}(t)=(g(t),\dot{g}(t))$,
   \begin{equation}\label{g}
   \partial_{t}g= \sum\limits_{k=1}^{K}(b_{k}+\lambda'_{k})\left(\Lambda W\right)_{\underline{\lambda_{k},x_{k}}}+\sum\limits_{k=1}^{K}x'_{k}\cdot\left(\nabla W\right)_{\underline{\lambda_{k},x_{k}}}+\dot{g}
   \end{equation}
   \begin{equation}\label{dotg}
   \begin{aligned}
       \partial_{t}\dot{g}=\Delta g+& f\left(\sum\limits_{k=1}^{K} W_{\lambda_{k},x_{k}}+g\right)-\sum\limits_{k=1}^{K}f\left(W_{\lambda_{k},x_{k}}\right)-\sum\limits_{k=1}^{K}b'_{k}\left(\Lambda W\right)_{\underline{\lambda_{k},x_{k}}}+\\
       &\sum\limits_{k=1}^{K}\frac{b_{k}\lambda'_{k}}{\lambda_{k}}\left({\underline{\Lambda}}\Lambda W\right)_{\underline{\lambda_{k},x_{k}}}+\frac{b_{k}x'_{k}}{\lambda_{k}}\cdot \left(\nabla \Lambda W\right)_{\underline{\lambda_{k},x_{k}}}.
   \end{aligned}  
   \end{equation}
 
     \begin{proposition}[Modulation Control]\label{mod con}
         Let $\delta>0$ be as in Lemma \ref{Mod}, let $J\subset\RR$ be a time interval, and let $\vec{u}(t)$ be a solution to (\ref{NLW1}) on $J$.  Assume that there exist  $\mu_{k}(t)>0$ and $y_{k}(t)\in\RR^{5}$ $(1\leq k\leq K)$ such that
     \begin{equation}
         \bigg\lVert u(t)-\sum_{k=1}^{K}W_{\mu_{k},y_{k}}\bigg\rVert_{\dot{H}^{1}}+\lVert\partial_{t}u\rVert_{L^{2}}+\sum_{k=1}^{K}\mu_{k}+\sum_{k=1}^{K}|y_{k}-z_{k}|<\delta
     \end{equation}
     and 
     \begin{equation}
         C_{1}\leq \frac{\mu_{k}}{\mu_{j}}\leq C_{2}
     \end{equation}
     for all $t\in J$, $1\leq j,k\leq K$. Let  $\lambda_{k}(t), x_{k}(t),b_{k}(t)\in C^{1}(J)$ be given by lemma \ref{Mod}. Then the following estimates hold for $t\in J$:
     \begin{equation}\label{lamkbk}
         |\lambda'_{k}+b_{k}|+|x'_{k}|\lesssim \lVert\dot{g}\rVert_{L^{2}}+\sum_{k=1}^{K}|b_{k}|\lVert g\rVert_{\dot{H}^{1}},
     \end{equation}
     \begin{align}\label{bk'}
         \left|b'_{k}+\kappa \sum_{j\neq k}|z_{j}-z_{k}|^{-3}\lambda_{k}^{\frac{1}{2}}\lambda_{j}^{\frac{3}{2}}\right|\lesssim &\frac{1}{\lambda_{k}}\left(\lVert \dot{g}\rVert^{2}_{L^{2}}+ \lVert g\rVert^{2}_{\dot{H}^{1}}\right)+\frac{1}{\lambda_{k}}\sum_{k=1}^{K}|b_{k}|\lVert\dot{g}\rVert_{L^{2}}\notag\\
        & +\lambda^{2}_{k}\left(\sum_{j=1}^{K}|x_{j}-z_{j}|\right)+\sum_{k=1}^{K}b^{2}_{k},
     \end{align}
     \begin{align}\label{refined bk'}
             \Biggl|b'_{k}&+ \kappa\sum_{j\neq k}|z_{j}-z_{k}|^{-3}\lambda^{\frac{1}{2}}_{k}\lambda_{j}^{\frac{3}{2}}+\frac{\lambda'_{k}}{\lambda_{k}}\frac{\langle(\underline{\Lambda}\Lambda W)_{\underline{\lambda_{k},x_{k}}},\dot{g}\rangle}{\lVert\Lambda W\rVert^{2}_{L^{2}}}+\frac{x'_{k}}{\lambda_{k}}\cdot\frac{\langle(\nabla \Lambda W)_{\underline{\lambda_{k},x_{k}}},\dot{g}\rangle}{\lVert \Lambda W\rVert^{2}_{L^{2}}}\notag\\
             & -\frac{1}{\lVert \Lambda W\rVert^{2}_{L^{2}}}\bigg\langle f\left(\sum_{k=1}^{K}W_{\lambda_{k},x_{k}}+g\right)-f\left(\sum_{k=1}^{K}W_{\lambda_{k},x_{k}}\right)-f'\left(\sum_{k=1}^{K}W_{\lambda_{k},x_{k}}\right)g,\notag\\
             &\left(\Lambda W\right)_{\underline{{\lambda_{k},x_{k}}}}\bigg\rangle
             \Biggl|
             \lesssim \lambda_{k}^{3}+\lambda_{k}^{2}\left(\sum_{j=1}^{K}|x_{j}-z_{j}|+\lVert g\rVert_{\dot{H}^{1}}\right)+\sum_{k=1}^{K}\left(b_{k}^{2}+|b_{k}|\lVert\dot{g}\rVert_{L^{2}}\right).
     \end{align}
     \end{proposition}
     \begin{proof}
         Differentiating (\ref{ort1}) yields
         \begin{align}
             0=\frac{d}{dt}\big\langle\left(\Delta \Lambda W\right)_{\underline{\lambda_{k},x_{k}}},g\big\rangle&=-\frac{\lambda'_{k}}{\lambda_{k}}\big\langle\left({\underline{\Lambda}}\Delta \Lambda W\right)_{\underline{\lambda_{k},x_{k}}},g\big\rangle-\frac{x'_{k}}{\lambda_{k}}\cdot\big\langle\left(\nabla \Delta \Lambda W\right)_{\underline{\lambda_{k},x_{k}}},g\big\rangle\notag\\
             &+\big\langle\left(\Delta\Lambda W\right)_{\underline{\lambda_{k},x_{k}}},\partial_{t}g\big\rangle.
                   \end{align}
Plugging \eqref{g} into the above identity and rearranging, we obtain 
   \begin{align}
           &\left(-\lVert\nabla\Lambda W\rVert^{2}_{L^{2}}-\frac{1}{\lambda_{k}}\big\langle\left({\underline{\Lambda}}\Delta \Lambda W\right)_{\underline{\lambda_{k},x_{k}}},g\big\rangle\right)\left(\lambda'_{k}+b_{k}\right)\notag\\
           &\qquad+\sum_{j\neq k}\left(\lambda'_{j}+b_{j}\right)\big\langle\left(\Delta \Lambda W\right)_{\underline{\lambda_{k},x_{k}}},\left(\Lambda W\right)_{\underline{\lambda_{j},x_{j}}}\big\rangle-\frac{x'_{k}}{\lambda_{k}}\cdot\big\langle\left(\nabla \Delta\Lambda W\right)_{\underline{\lambda_{k},x_{k}}},g\big\rangle\notag\\
           &\qquad+\sum_{j\neq k}x'_{j}\cdot\big\langle \left(\Delta \Lambda W\right)_{\underline{\lambda_{k},x_{k}}},\left(\nabla W\right)_{\underline{\lambda_{j},x_{j}}}\big\rangle\notag\\
           &=-\big\langle\left(\Delta \Lambda W\right)_{\underline{\lambda_{k},x_{k}}},\dot{g}\big\rangle-\frac{b_{k}}{\lambda_{k}}\big\langle\left({\underline{\Lambda}}\Delta \Lambda W\right)_{\underline{\lambda_{k},x_{k}}},g\big\rangle.
       \end{align}  
   Differentiating (\ref{ort2}) yields
   \begin{equation}
   \begin{aligned}
       0=\frac{d}{dt}\big\langle\left(\partial_{i}W\right)_{\underline{\lambda_{k},x_{k}}},g\big\rangle=&-\frac{\lambda'_{k}}{\lambda_{k}}\big\langle\left({\underline{\Lambda}}\partial_{i}W\right)_{\underline{\lambda_{k},x_{k}}},g\big\rangle-\frac{x'_{k}}{\lambda_{k}}\cdot\big\langle\left[\nabla\left(\partial_{i}W\right)\right]_{\underline{\lambda_{k},x_{k}}},g\big\rangle\\
       +&\big\langle\left(\partial_{i}W\right)_{\underline{\lambda_{k}x_{k}}},\partial_{t}g\big\rangle.
       \end{aligned}
   \end{equation}
Plugging \eqref{g} into the above identity and rearranging, we obtain
   \begin{equation}
       \begin{aligned}
           &-\frac{1}{\lambda_{k}}\big\langle\left(\underline{\Lambda}\partial_{i} W\right)_{\underline{\lambda_{k},x_{k}}},g\big\rangle(\lambda'_{k}+b_{k})+\sum_{j\neq k}(\lambda'_{j}+b_{j})\big\langle\left(\partial_{i}W\right)_{\underline{\lambda_{k},x_{k}}},\left(\Lambda W\right)_{\underline{\lambda_{j},x_{j}}}\big\rangle\\
           &+\left(\big\langle\partial_{i}W,\nabla W\big\rangle-\frac{1}{\lambda_{k}}\big\langle\left[\nabla\left(\partial_{i}W\right)\right]_{\underline{\lambda_{k},x_{k}}},g\big\rangle\right)\cdot x'_{k}\notag\\
           &+\sum_{j\neq k}\big\langle\left(\partial_{i}W\right)_{\underline{\lambda_{k},x_{k}}},\left(\nabla W\right)_{\underline{\lambda_{j},x_{j}}}\big\rangle\cdot{x'_{j}}\\
           =&-\big\langle\left(\partial_{i}W\right)_{\underline{\lambda_{k},x_{k}}},\dot{g}\big\rangle-\frac{b_{k}}{\lambda_{k}}\big\langle\left(\underline{\Lambda}\partial_{i} W\right)_{\underline{\lambda_{k},x_{k}}},g\big\rangle
       \end{aligned}
   \end{equation}
We then arrive at the following linear system for $\left(\boldsymbol{\lambda},\boldsymbol{x}\right)$
\begin{equation}
    M\begin{pmatrix}
        \boldsymbol{b}+\boldsymbol{\lambda}'\\
        \boldsymbol{x}'
    \end{pmatrix}=\begin{pmatrix}
        \boldsymbol{G}_{1}\\
        \boldsymbol{G}_{2}
    \end{pmatrix},
\end{equation}
where 
\begin{equation*}
    M=\begin{pmatrix}
        A & B\\ 
        C & D
    \end{pmatrix},\ A\in\RR^{K\times K},D\in\RR^{5K\times 5K}, B\in\RR^{K\times 5K}, C\in\RR^{5K\times K},
\end{equation*}
\begin{align*}
    &A_{kk}:=-\lVert\nabla\Lambda W\rVert^{2}_{L^{2}}-\frac{1}{\lambda_{k}}\big\langle({\underline{\Lambda}}\Delta\Lambda W)_{\underline{\lambda_{k},x_{k}}},g\big\rangle\ (1\leq k\leq K),\\
    &A_{ij}:=\big\langle\left(\Delta \Lambda W\right)_{\underline{\lambda_{i},x_{i}}},\left(\Lambda W\right)_{\underline{\lambda_{j},x_{j}}}\big\rangle\ (i\neq j, 1\leq i,j\leq K),\\
    &D_{kk}:=\frac{1}{5}\lVert\nabla W\rVert^{2}_{L^{2}}I_{5}+\left(-\frac{1}{\lambda_{k}}\big\langle\left(\partial_{ij}W\right)_{\underline{\lambda_{k},x_{k}}},g\big\rangle\right)_{1\leq i,j\leq 5}\ (1\leq k\leq K),\\ &D_{kl}:=\left(\big\langle(\partial_{i}W)_{\underline{\lambda_{k},x_{k}}},(\partial_{j}W)_{\underline{\lambda_{l},x_{l}}}\big\rangle\right)_{1\leq i,j\leq 5}\ (1\leq k, l\leq K,\ k\neq l),\\
    &B_{kk}:=\left(-\frac{1}{\lambda_{k}}\big\langle\left[\partial_{i}(\Delta\Lambda W)\right]_{\underline{\lambda_{k},x_{k}}},g\big\rangle\right)_{1\times 5},\ B_{kl}:=\left(\big\langle (\Delta\Lambda W)_{\underline{\lambda_{k},x_{k}}},(\partial_{i} W)_{\underline{\lambda_{l},x_{l}}}\big\rangle\right)_{1\times 5},\\
    &C_{kk}:=\left(-\frac{1}{\lambda_{k}}\big\langle(\underline{\Lambda}\partial_{i} W)_{\underline{\lambda_{k},x_{k}}},g\big\rangle\right)_{5\times 1},\ C_{kl}:=\left(\big\langle(\partial_{i}W)_{\underline{\lambda_{k},x_{k}}},(\Lambda W)_{\underline{\lambda_{l},x_{l}}}\big\rangle\right)_{5\times 1},\\
    &\boldsymbol{G}_{1}=\left(-\big\langle\left(\Delta \Lambda W\right)_{\underline{\lambda_{k},x_{k}}},\dot{g}\big\rangle
           -\frac{b_{k}}{\lambda_{k}}\big\langle\left({\underline{\Lambda}}\Delta \Lambda W\right)_{\underline{\lambda_{k},x_{k}}},g\big\rangle\right)_{K\times 1},\\
           &\boldsymbol{G}_{2}=\left(-\big\langle\left(\nabla W\right)_{\underline{\lambda_{k},x_{k}}},\dot{g}\big\rangle-\frac{b_{k}}{\lambda_{k}}\big\langle\left(\underline{\Lambda}\nabla W\right)_{\underline{\lambda_{k},x_{k}}},g\big\rangle\right)_{5K\times 1}.
\end{align*}
Note that 
\begin{align*}
    &\left|\frac{1}{\lambda_{k}}\big\langle\left(\underline{\Lambda}\Delta\Lambda W\right)_{\underline{\lambda_{k},x_{k}}},g\big\rangle\right|+\left|\frac{1}{\lambda_{k}}\big\langle(\partial_{ij} W)_{\underline{\lambda_{k},x_{k}}},g\big\rangle\right|+\left|\frac{1}{\lambda_{k}}\big\langle\left[\partial_{i}(\Delta\Lambda W)\right]_{\underline{\lambda_{k},x_{k}}},g\big\rangle\right|\\
    &+\left|\frac{1}{\lambda_{k}}\big\langle(\underline{\Lambda}\partial_{i}W)_{\underline{\lambda_{k},x_{k}}},g\big\rangle\right|\lesssim \lVert g\rVert_{\dot{H}^{1}},
\end{align*}
and, by Lemma \ref{WW},
\begin{equation*}
    |A_{k l}|\lesssim  \lambda_{k}^\frac{5}{2}\lambda_{l}^{\frac{1}{2}},\ |D_{kl}|\lesssim \lambda^{\frac{3}{2}}_{k}\lambda^{\frac{3}{2}}_{l},\ |B_{kl}|\lesssim \lambda^{\frac{3}{2}}_{l}\lambda_{k}^{\frac{5}{2}} ,\ |C_{kl}|\lesssim \lambda^{\frac{3}{2}}_{k}\lambda_{l}^{\frac{1}{2}},\ \forall\ k\neq l.
\end{equation*}
Hence, if we choose $\delta_{0}>0$ sufficiently small in Lemma \ref{Mod}, then the matrix $M$ is diagonally dominant and
\begin{equation}
    M=\begin{pmatrix}
        -\lVert\nabla \Lambda W\rVert^{2}_{L^{2}}I_{K}&\\
        &\frac{1}{5}\lVert\nabla W\rVert^{2}_{L^{2}}I_{5K}
    \end{pmatrix}+O(\delta),
\end{equation}
thus
\begin{equation*}
    \begin{pmatrix}
        \boldsymbol{b}+\boldsymbol{\lambda}'\\
        \boldsymbol{x}'
    \end{pmatrix}= M^{-1}\begin{pmatrix}
        \boldsymbol{G}_{1}\\
        \boldsymbol{G}_{2}
    \end{pmatrix},
\end{equation*}
which, combined with the fact that
\begin{equation*}
    |\boldsymbol{G}_{1}|+|\boldsymbol{G}_{2}|\lesssim \lVert \dot{g}\rVert_{L^{2}}+\left(\sum_{k=1}^{K}|b_{k}|\right)\lVert g\rVert_{\dot{H}^{1}},
\end{equation*}
yields
\begin{equation}
    \left|\lambda'_{k}+b_{k}\right|+\left|x'_{k}\right|\lesssim \lVert \dot{g}\rVert_{L^{2}}+\left(\sum_{k=1}^{K}|b_{k}|\right)\lVert g\rVert_{\dot{H}^{1}},\  \forall 1\leq k\leq K.
\end{equation}
Finally, differentiating (\ref{ort3}) gives
\begin{align}
    0=\frac{d}{dt}\big\langle \dot{g},(\Lambda W)_{\underline{ \lambda_{k},x_{k}}}\big\rangle=&-\frac{\lambda'_{k}}{\lambda_{k}}\big\langle(\underline{\Lambda}\Lambda W)_{\underline{\lambda_{k},x_{k}}},\dot{g}\big\rangle-\frac{x'_{k}}{\lambda_{k}}\cdot\big\langle (\nabla \Lambda W)_{\underline{\lambda_{k},x_{k}}},\dot{g}\big\rangle\notag\\
    +&\big\langle\partial_{t}\dot{g},(\Lambda W)_{\underline{\lambda_{k},x_{k}}}\big\rangle.
\end{align}
Plugging in \eqref{dotg}, using the cancellation $L\Lambda W=0$, and rearranging, we obtain
\begin{align}
        &\sum_{j=1}^{K}\left(\big\langle(\Lambda W)_{\underline{\lambda_{j},x_{j}}},(\Lambda W)_{\underline{\lambda_{k},x_{k}}}\big\rangle\right)b_{j}'\notag\\
        =&\bigg\langle f'\left(W_{\lambda_{k},x_{k}}\right)\sum_{j\neq k}W_{\lambda_{j},x_{j}},\left(\Lambda W\right)_{\underline{\lambda_{k},x_{k}}}\bigg\rangle\ ({\rm{I}})\notag\\
        &-\frac{\lambda'_{k}}{\lambda_{k}}\big\langle(\underline{\Lambda}\Lambda W)_{\underline{\lambda_{k},x_{k}}},\dot{g}\big\rangle-\frac{x'_{k}}{\lambda_{k}}\cdot\big\langle (\nabla \Lambda W)_{\underline{\lambda_{k},x_{k}}},\dot{g}\big\rangle\ ({\rm{II}})\notag\\
        &+\Bigg\langle f\left(\sum_{k=1}^{K}W_{\lambda_{k},x_{k}}+g\right)-f\left(\sum_{k=1}^{K} W_{\lambda_{k},x_{k}}\right)-f'\left(\sum_{k=1}^{K}W_{\lambda_{k},x_{k}}\right)g,(\Lambda W)_{\underline{\lambda_{k},x_{k}}}\Bigg\rangle({\rm{III}})\notag\\
        &+\Bigg\langle\left( f'\left(\sum_{j=1}^{K}W_{\lambda_{j},x_{j}}\right)-f'\left(W_{\lambda_{k},x_{k}}\right)\right)g,(\Lambda W)_{\underline{\lambda_{k},x_{k}}}\Bigg\rangle\ ({\rm{IV}})\notag\\
        &+\Bigg\langle f\left(\sum_{k=1}^{K}W_{\lambda_{k},x_{k}}\right)-\sum_{k=1}^{K}f(W_{\lambda_{k},x_{k}})-f'(W_{\lambda_{k},x_{k}})\sum_{j\neq k}W_{\lambda_{j},x_{j}}, (\Lambda W)_{\underline{\lambda_{k},x_{k}}}\Bigg\rangle\ ({\rm{V}})\notag\\
        &+\sum_{j=1}^{K}\frac{b_{j}\lambda'_{j}}{\lambda_{j}}\big\langle (\underline{\Lambda}\Lambda W)_{\underline{\lambda_{j},x_{j}}},(\Lambda W)_{\underline{\lambda_{k},x_{k}}}\big\rangle+\sum_{j=1}^{K}\frac{b_{j}x'_{j}}{\lambda_{j}}\cdot\big\langle (\nabla \Lambda W)_{\underline{\lambda_{j},x_{j}}},(\Lambda W)_{\underline{\lambda_{k},x_{k}}}\big\rangle\ ({\rm{VI}}).
    \end{align}
We now estimate the terms on the right-hand side separately:\\
For (I), we first claim that for any $j\neq k$
\begin{equation}\label{interaction}
\begin{aligned}
    &\big\langle (\Lambda W)_{\underline{\lambda_{k},x_{k}}},f'\left(W_{\lambda_{k},x_{k}}\right)W_{\lambda_{j},x_{j}}\big\rangle\\
    =&15^{\frac{3}{2}}|z_{j}-z_{k}|^{-3}\langle \Lambda W,f'(W)\rangle\lambda_{k}^{\frac{1}{2}}\lambda_{j}^{\frac{3}{2}}
    +O\left(\lambda_{k}^{3}+\lambda_{k}^{2}\left(|x_{k}-z_{k}|+|x_{j}-z_{j}|\right)\right).
\end{aligned}
\end{equation}
Note that 
\begin{equation*}
    \begin{aligned}
        &\left|\big\langle(\Lambda W)_{\underline{\lambda_{k},x_{k}}},f'\left(W_{\lambda_{k},x_{k}}\right)W_{\lambda_{j},x_{j}}\big\rangle-15^{\frac{3}{2}}|z_{j}-z_{k}|^{-3}\big\langle\Lambda W,f'(W)\big\rangle\lambda^{\frac{1}{2}}_{k}\lambda_{j}^{\frac{3}{2}}\right|\\
        =&\frac{7}{3}\lambda_{j}^{-\frac{3}{2}}\lambda_{k}^{\frac{1}{2}}\left|\int \Lambda W(x) W^{\frac{4}{3}}(x)\left[W\left(\frac{\lambda_{j}}{\lambda_{k}}x-\frac{x_{j}-x_{k}}{\lambda_{j}}\right)-15^{\frac{3}{2}}\lambda_{j}^{3}|z_{k}-z_{j}|^{-3}\right]{\rm{d}}x\right|,
    \end{aligned}
\end{equation*}
Then, in the region $\displaystyle |x|\geq\frac{d}{\lambda_{k}}$, by the Cauchy--Schwarz inequality,
\begin{equation*}
    \begin{aligned}
        &\left|\int_{|x|\geq \frac{d}{\lambda_{k}}}\Lambda W(x) W^{\frac{4}{3}}(x) W\left(\frac{\lambda_{k}}{\lambda_{j}}x-\frac{x_{j}-x_{k}}{\lambda_{j}}\right) {\rm{d}} x\right|\\
        \lesssim&\lambda^{4}_{k}\int_{|x|\geq\frac{d}{\lambda_{k}}}\Lambda W(x)W\left(\frac{\lambda_{k}}{\lambda_{j}}x-\frac{x_{j}-x_{k}}{\lambda_{j}}\right){\rm{d}}x\lesssim\lambda^{4}_{k},
    \end{aligned}
\end{equation*}
and 
\begin{equation*}
    \int_{|x|\geq \frac{d}{\lambda_{k}}}\Lambda W(x) W^{\frac{4}{3}}(x){\rm{d}}x\lesssim\int_{\frac{d}{\lambda_{k}}}^{+\infty}r^{-7}r^{4}{\rm{d}}r\lesssim\lambda^{2}_{k}.
\end{equation*}
In the region $\displaystyle|x|\leq \frac{d}{\lambda_{k}}$, 
\begin{equation*}
\begin{aligned}
    &\left|W\left(\frac{\lambda_{k}}{\lambda_{j}}x-\frac{x_{j}-x_{k}}{\lambda_{j}}\right)-15^{\frac{3}{2}}\frac{\lambda_{j}^{3}}{|z_{j}-z_{k}|^{3}}\right|\\
    \leq &\left|W\left(\frac{\lambda_{k}}{\lambda_{j}}x-\frac{x_{j}-x_{k}}{\lambda_{j}}\right)-W\left(\frac{x_{j}-x_{k}}{\lambda_{j}}\right)\right|\\
    &+\left|W\left(\frac{x_{j}-x_{k}}{\lambda_{j}}\right)-W\left(\frac{z_{j}-z_{k}}{\lambda_{j}}\right)\right|\\
    &+\left|W\left(\frac{z_{j}-z_{k}}{\lambda_{j}}\right)-15^{\frac{3}{2}}\lambda_{j}^{3}\frac{1}{|z_{j}-z_{k}|^{3}}\right|.
\end{aligned}
\end{equation*}
By the mean value theorem,
\begin{equation*}
    \left|W\left(\frac{\lambda_{k}}{\lambda_{j}}x-\frac{x_{j}-x_{k}}{\lambda_{j}}\right)-W\left(\frac{x_{j}-x_{k}}{\lambda_{j}}\right)\right|\lesssim\lambda^{4}_{k}|x|,
\end{equation*}
\begin{equation*}
    \left|W\left(\frac{z_{j}-z_{k}}{\lambda_{j}}\right)-W\left(\frac{x_{j}-x_{k}}{\lambda_{j}}\right)\right|\lesssim\lambda^{3}_{k}\left(|x_{k}-z_{k}|+|x_{j}-z_{j}|\right).
\end{equation*}
Using the fact that
\begin{equation*}
    \left|W(y)-15^{\frac{3}{2}}|y|^{-3}\right|\leq |y|^{-5},\quad \forall \ |y|\geq 1,
\end{equation*}
and for $\delta>0$ small enough, we obtain
\begin{equation*}
    \left|W\left(\frac{z_{j}-z_{k}}{\lambda_{j}}\right)-15^{\frac{3}{2}}\lambda^{3}_{j}|z_{j}-z_{k}|^{-3}\right|\lesssim \lambda_{k}^{5}.
\end{equation*}
Hence
\begin{equation*}
   \left|W\left(\frac{\lambda_{k}}{\lambda_{j}}x-\frac{x_{j}-x_{k}}{\lambda_{j}}\right)-15^{\frac{3}{2}}\frac{\lambda_{j}^{3}}{|z_{j}-z_{k}|^{3}}\right|\lesssim\lambda_{k}^{4}|x|+\lambda^{3}_{k}\left(|x_{k}-z_{k}|+|x_{j}-z_{j}|\right)+\lambda^{5}_{k},
\end{equation*}
and then
\begin{align*}
    &\left|\int_{|x|\leq\frac{d}{\lambda_{k}}}\Lambda W(x)W^{\frac{4}{3}}(x)\left[W\left(\frac{\lambda_{k}}{\lambda_{j}}x-\frac{x_{j}-x_{k}}{\lambda_{j}}\right)-15^{\frac{3}{2}}\frac{\lambda^{3}_{j}}{|z_{j}-z_{k}|^{3}}\right]{\rm{d}}x\right|\\
    \lesssim&\int_{|x|\leq\frac{d}{\lambda_{k}}}\left|\Lambda W(x)\right|W^{\frac{4}{3}}(x)\left(\lambda_{k}^{4}|x|+\lambda_{k}^{3}\left(|x_{j}-z_{j}|+|x_{j}-z_{k}|\right)+\lambda^{5}_{k}\right){\rm{d}}x\\
    \lesssim&\lambda_{k}^{4}+\lambda^{3}_{k}\left(|x_{k}-z_{k}|+|x_{j}-z_{j}|\right).
\end{align*}
Combining the above estimates gives (\ref{interaction}).\\
For term (II), 
\begin{equation*}
    |{\rm{II}}|\lesssim \frac{|\lambda'_{k}+b_{k}|+|b_{k}|}{\lambda_{k}}\lVert \dot{g}\rVert_{L^{2}}+\frac{|x'_{k}|}{\lambda_{k}}\lVert\dot{g}\rVert_{L^{2}}\lesssim\frac{1}{\lambda_{k}}\lVert\dot{g}\rVert^{2}_{L^{2}}+\frac{1}{\lambda_{k}}\sum_{k=1}^{K}|b_{k}|\lVert\dot{g}\rVert_{L^{2}}.
\end{equation*}
For term (III), from (\ref{nonlinear 3}) and H\"{o}lder inequality
\begin{equation*}
|{\rm{III}}|\lesssim \frac{1}{\lambda_{k}}\int\left(|g|^{\frac{1}{3}}+\sum_{k=1}^{K}\left|W_{\lambda_{k},x_{k}}\right|^{\frac{1}{3}}\right)|g|^{2}\left|\left(\Lambda W\right)_{\lambda_{k},x_{k}}\right|{\rm{d}}x\lesssim\frac{1}{\lambda_{k}}\lVert g\rVert^{2}_{\dot{H}^{1}}.
\end{equation*}
For term (IV), from (\ref{nonlinear 4})
\begin{equation*}
    \begin{aligned}
        |{\rm{IV}}|&\lesssim\frac{1}{\lambda_{k}}\int\left|f'\left(\sum_{j=1}^{K}W_{\lambda_{j},x_{j}}\right)-f'\left(W_{\lambda_{k},x_{k}}\right)\right|\left|W_{\lambda_{k},x_{k}}\right||g|{\rm{d}}x\\
        &\lesssim\frac{1}{\lambda_{k}}\int\sum_{j\neq k}W_{\lambda_{j},x_{j}}W_{\lambda_{k},x_{k}}\left(\sum_{j=1}^{K}\left|W_{\lambda_{j},x_{j}}\right|^{\frac{1}{3}}\right)|g|{\rm{d}}x\\
        &\lesssim\frac{1}{\lambda_{k}}\sum_{j\neq k}\int W^{\frac{4}{3}}_{\lambda_{j},x_{j}}W_{\lambda_{k},x_{k}}|g|{\rm{d}}x\lesssim\frac{1}{\lambda_{k}}\lambda^{3}_{k}\lVert g\rVert_{\dot{H}^{1}}\lesssim \lambda_{k}^{2}\lVert g\rVert_{\dot{H}^{1}}.
    \end{aligned}
\end{equation*}
For term (V), we first note from (\ref{nonlinear 2}) that
\begin{equation*}
\begin{aligned}
    &\left|f\left(\sum_{k=1}^{K}W_{\lambda_{k},x_{k}}\right)-\sum_{k=1}^{K}f(W_{\lambda_{k},x_{k}})-f'(W_{\lambda_{k},x_{k}})\sum_{j\neq k}W_{\lambda_{j},x_{j}}\right|\\
    \lesssim& \big|W_{\lambda_{k},x_{k}}\big|^{\frac{2}{3}}\sum_{j\neq k}\big|W_{\lambda_{j},x_{j}}\big|^{\frac{5}{3}}+\sum_{j\neq l;\ j,l\neq k}\big|W_{\lambda_{j},x_{j}}W_{\lambda_{l},x_{l}}^{\frac{4}{3}}\big|.
\end{aligned}
\end{equation*}
Combining this with Lemma~\ref{WW} and H\"{o}lder's inequality gives
\begin{equation*}
    \begin{aligned}
        &|{\rm{V}}|\lesssim \frac{1}{\lambda_{k}}\int|W_{\lambda_{k},x_{k}}|^{\frac{5}{3}}\sum_{j\neq k}|W_{\lambda_{j},x_{j}}|^{\frac{5}{3}}+\sum_{j\neq l;\ j,l\neq k}|W_{\lambda_{j},x_{j}}W_{\lambda_{l},x_{l}}^{\frac{4}{3}}||W_{\lambda_{k},x_{k}}|\\
        \lesssim&\frac{1}{\lambda_{k}}\sum_{j\neq l}\int |W_{\lambda_{l},x_{l}}|^{\frac{5}{3}}|W_{\lambda_{j},x_{j}}|^{\frac{5}{3}}{\rm{d}}x\lesssim \lambda_{k}^{4}|\log \lambda_{k}|.
    \end{aligned}
\end{equation*}
For term (VI), using Lemma \ref{WW} and the fact that $\langle \underline{\Lambda}\Lambda W,\Lambda W\rangle=0$, $\langle \nabla \Lambda W,\Lambda W\rangle=0$, we have
\begin{equation*}
    \begin{aligned}
        |{\rm{VI}}|&\lesssim\sum_{j\ne k}\frac{|b_{j}\lambda'_{j}|}{\lambda_{j}}\left|\big\langle(\underline{\Lambda}\Lambda W)_{\underline{\lambda_{j},x_{j}}},\left(\Lambda W\right)_{\underline{\lambda_{k},x_{k}}}\big\rangle\right|+\frac{|b_{j}x'_{j}|}{\lambda_{j}}\left|\big\langle\left(\nabla\Lambda W\right)_{\underline{\lambda_{j},x_{j}}},\left(\Lambda W\right)_{\underline{\lambda_{k},x_{k}}}\big\rangle\right|\\
        &\lesssim \sum_{j\neq k}\frac{\left(\lVert\dot{g}\rVert_{L^{2}}+\left(\sum\limits_{k=1}^{K}|b_{k}|\right)\lVert g\rVert_{\dot{H}^{1}}+|b_{j}|\right)|b_{j}|}{\lambda_{j}}\lambda_{j}^\frac{1}{2}\lambda_{k}^{\frac{1}{2}}\lesssim\lVert\dot{g}\rVert_{L^{2}}\sum_{k=1}^{K}|b_{k}|+\sum_{k=1}^{K}|b_{k}|^{2}.
    \end{aligned}
\end{equation*}
Combining the above estimates, we obtain
\begin{align*}
    &\sum_{j=1}^{K}\left(\big\langle\left(\Lambda W\right)_{\underline{\lambda_{j},x_{j}}},\left(\Lambda W\right)_{\underline{\lambda_{k},x_{k}}}\big\rangle\right)b'_{j}\\
    =&15^{\frac{3}{2}}|z_{j}-z_{k}|^{-3}\big\langle \Lambda W, f'(W)\big\rangle\lambda_{k}^{\frac{1}{2}}\sum_{j\neq k}\lambda_{j}^{\frac{3}{2}}\\
    &-\frac{\lambda'_{k}}{\lambda_{k}}\big\langle\left(\underline{\Lambda}\Lambda W\right)_{\underline{\lambda_{k},x_{k}}},\dot{g}\big\rangle-\frac{x'_{k}}{\lambda_{k}}\cdot\big\langle\left(\nabla \Lambda W\right)_{\underline{\lambda_{k},x_{k}}}, \dot{g}\big\rangle\\
    &+\bigg\langle f\left(\sum_{k=1}^{K}W_{\lambda_{k},x_{k}}+g\right)-f\left(\sum_{k=1}^{K}W_{\lambda_{k},x_{k}}\right)-f'\left(\sum_{k=1}^{K}W_{\lambda_{k},x_{k}}\right)g,
             \left(\Lambda W\right)_{\underline{{\lambda_{k},x_{k}}}}\bigg\rangle\\
             &+O\left(\lambda_{k}^{3}+\lambda_{k}^{2}\sum_{j=1}^{K}|x_{j}-z_{j}|+\lambda_{k}^{2}\lVert\dot{g}\rVert_{\dot{H}^{1}}+\lVert\dot{g}\rVert_{L^{2}}\sum_{k=1}^{K}|b_{k}|+\sum_{k=1}^{K}|b_{k}|^{2}\right).
\end{align*}
Since the matrix
\begin{equation*}
    \begin{aligned}
        S&=\left(\big\langle(\Lambda W)_{\underline{\lambda_{j},x_{j}}},(\Lambda W)_{\underline{\lambda_{k},x_{k}}}\big\rangle\right)_{1\leq i,j\leq K}\\
        &=\lVert\Lambda W\rVert^{2}_{L^{2}}I_{K}+O(\lambda_{k})
    \end{aligned}
\end{equation*}
is diagonally dominant for $0<\lambda_{k}\lesssim \delta$ sufficiently small, inverting $S$ yields \eqref{bk'} and \eqref{refined bk'}.
     \end{proof}
     \section{Refined modulation}\label{refined modulation}

     In this section, we define a truncated virial functional and state some estimates related to it. 
     The same functional was used in an essential way in the construction of multi-bubble solutions to \eqref{NLW1} by the first author and Martel in \cite{JM}.
     \begin{lemma}\cite[Lemma 18]{JM}\label{q(x)}
         Let any $\epsilon>0$ and $R>0$. There exists a radially symmetric function $q=q_{\epsilon,R}\in C^{3,1}(\RR)$ with the following properties:\\
         ${\rm{(i)}}$ $\displaystyle q(x)=\frac{1}{2}|x|^{2}$ for $|x|\leq R$.\\
         ${\rm{(ii)}}$ There exists $\tilde{R}$ (depending on $\epsilon$ and $R$) such that $q$ is constant for $|x|\geq \tilde{R}$.\\
         ${\rm{(iii)}}$ $\left|\nabla q(x)\right|\lesssim |x| $ and $|\Delta q(x)|\lesssim 1$ for all $x\in\RR^{5}$, with constants independent of $\epsilon$ and $R$.\\
         ${\rm{(iv)}}$ $\displaystyle\sum_{1\leq j,l\leq 5}\left(\partial_{x_{j}x_{l}}q(x)\right)v_{j}v_{l}\geq -\epsilon\sum_{j=1}^{5}|v_{j}|^{2}$, for all $x\in\RR^{5}$, $v\in\RR^{5}$.\\
         ${\rm{(v)}}$ $\Delta^{2}q(x)\leq \epsilon|x|^{-2}$, for all $x\in\RR^{5}$.
     \end{lemma}
    We fix a function $q$ as in Lemma \ref{q(x)} and define the operators
     \begin{align*}
         &[A_{k}h](x)=\frac{3}{10}\frac{1}{\lambda_{k}}\Delta q\left(\frac{x-x_{k}}{\lambda_{k}}\right)h(x)+\nabla q\left(\frac{x-x_{k}}{\lambda_{k}}\right)\cdot \nabla h(x),\\
         &[{\underline{A}_{k}}h](x)=\frac{1}{2}\frac{1}{\lambda_{k}}\Delta q\left(\frac{x-x_{k}}{\lambda_{k}}\right)h(x)+\nabla q\left(\frac{x-x_{k}}{\lambda_{k}}\right)\cdot \nabla h(x).
     \end{align*}
     \begin{lemma}\cite[Lemma 18]{JM}\label{AK}
     For any $\eta>0$, there exist $\epsilon_{1}, R_{1}>0$ such that, for all $\epsilon<\epsilon_{1}$ and $R>R_{1}$ in Lemma \ref{q(x)}, the operators $A_{k}$ and ${\underline{A}}_{k}$ $(1\leq k\leq K)$ defined above have the following properties.\\
         ${\rm{(i)}}$ The families
         \begin{align*}
             &\{A_{k}; \lambda_{k}>0,x_{k}\in\RR^{5}\},\ \{{\underline{A}}_{k}; \lambda_{k}>0,x_{k}\in \RR^{5}\},\\
             \{\lambda_{k}\partial_{\lambda_{k}}&A_{k}; \lambda_{k}>0, x_{k}\in\RR^{5}\},\ \{\lambda_{k}\partial_{\lambda_{k}}{\underline{A}}_{k}; \lambda_{k}>0, x_{k}\in\RR^{5}\},\\
             \{\lambda_{k}\partial_{x_{k}}A_{k}&;\lambda_{k}>0, x_{k}\in\RR^{5}\}\ {\rm{and}}\ \{\lambda_{k}\partial_{x_{k}}{\underline{A}}_{k}; \lambda_{k}>0, x_{k}\in\RR^{5}\}
         \end{align*}
         are bounded in $\mathcal{L}(\dot{H}^{1},L^{2})$, with norms depending on $q$.\\
         ${\rm{(ii)}}$ For any $g,h\in\dot{H}^{1}\cap \dot{H}^{2}$,
         \begin{equation}\label{sym}
             \langle A_{k}h,f(h+g)-f(h)-f'(h)g\rangle=-\langle A_{k}g,f(h+g)-f(h)\rangle.
         \end{equation}
         ${\rm{(iii)}}$ For the modulated parameters $\lambda_{k}>0$ and $x_{k}\in\RR^{5}$ as in Lemma \ref{Mod}, we have
         \begin{equation}\label{AK 2}
             \left\lVert\underline{A}_{k} \left(\Lambda W\right)_{\lambda_{k},x_{k}}-\left(\underline{\Lambda}\Lambda W\right)_{\underline{\lambda_{k},x_{k}}}\right\rVert_{L^{2}}+\left\lVert\underline{A}_{k}\left(\nabla W\right)_{\lambda_{k},x_{k}}-\left(\nabla \Lambda W\right)_{\underline{\lambda_{k},x_{k}}}\right\rVert_{L^{2}}\leq \eta,
         \end{equation}
    \begin{equation}\label{AK 3}
        \left\lVert A_{k}W_{\lambda_{k},x_{k}}-\left(\Lambda W\right)_{\underline{\lambda_{k},x_{k}}}\right\rVert_{L^{\frac{10}{3}}}\leq\frac{\eta}{\lambda_{k}},
    \end{equation}
     \begin{equation}\label{AK 4}
             \left\lVert\underline{A}_{k}\left(\Lambda W\right)_{\lambda_{j},x_{j}}\right\rVert_{L^{2}}+\left\lVert\underline{A}_{k}\left(\nabla W\right)_{\lambda_{j},x_{j}}\right\rVert\lesssim_{q}\lambda_{j}^{\frac{1}{2}},\quad \forall\ j\neq k,
         \end{equation}
         \begin{equation}\label{AK 5}
            \left\lVert A_{k} W_{\lambda_{j},x_{j}}\right\rVert_{L^{\frac{10}{3}}}\lesssim_{q} \lambda_{j}^{\frac{1}{2}},\quad \forall\ j\neq k.
         \end{equation}
         ${\rm{(iv)}}$ Finally, for all $g\in \dot{H}^{1}\cap\dot{H}^{2}$, we have
         \begin{equation}\label{cor loc}
             \langle {\underline{A}}_{k}g,\Delta g\rangle\leq \frac{\eta}{\lambda_{k}}\lVert g\rVert^{2}_{\dot{H}^{1}}-\frac{1}{\lambda_{k}}\int\limits_{|x-x_{k}|<R\lambda_{k}}|\nabla g(x)|^{2}{\rm{d}}x.
         \end{equation}
     \end{lemma}
    We introduce the following modifications of $\lambda_k(t)$ and $b_k(t)$:
     \begin{equation}\label{refine lambda}
         \zeta_{k}(t):=\lambda_{k}(t)-\frac{1}{\lVert \Lambda W\rVert^{2}_{L^{2}}}\bigg\langle\chi\left(\frac{\cdot-x_{k}(t)}{\lambda_{k}(t)M}\right)\left(\Lambda W\right)_{\underline{\lambda_{k},x_{k}}},g(t)\bigg\rangle,
     \end{equation}
     \begin{equation}\label{pk}
     \begin{aligned}
         p_{k}(t):=b_{k}(t)
         &-\frac{1}{\lVert\Lambda W\rVert^{2}_{L^{2}}}\frac{b_{k}(t)}{\lambda_{k}(t)}\bigg\langle \chi\left(\frac{\cdot-x_{k}(t)}{\lambda_{k}(t)M}\right)({\underline{\Lambda}}\Lambda W)_{\underline{\lambda_{k},x_{k}}},g\bigg\rangle \\
         &+\frac{1}{\lVert \Lambda W\rVert^{2}_{L^{2}}}\langle \dot{g},{\underline{A}}_{k}g\rangle,
         \end{aligned}
     \end{equation}
     where $\chi\in C^{\infty}_{c}(\RR^{5})$ satisfies $\chi\equiv 1$ on $|x|\leq 1$, ${\rm{supp}}\chi\subset \{|x|\leq 2\}$, and $M>0$ is a sufficiently large constant to be determined below.\\
     
We are now ready to state the main modulation estimates.
 \begin{lemma}\label{some refined estimates}
 Let $\lambda_{k}(t)$, $x_{k}(t)$, $b_{k}(t)$ be given by lemma \ref{Mod}, let $\zeta_{k}(t)$ be given by (\ref{refine lambda}), and  let $p_{k}(t)$ be given by (\ref{pk}). Then 
         \begin{equation}\label{zetalam}
             \left|\frac{\zeta_{k}(t)}{\lambda_{k}(t)}-1\right|\lesssim \sqrt{M}\lVert g\rVert_{\dot{H}^{1}},
         \end{equation}
         \begin{equation}\label{zetabk}
             \left|\zeta'_{k}(t)+b_{k}(t)\right|\lesssim \frac{1}{\sqrt{M}}\lVert\dot{g}\rVert_{L^{2}}+\sqrt{M}\left(\sum_{k=1}^{K}|b_{k}|\right)\lVert g\rVert_{\dot{H}^{1}},
         \end{equation}
         \begin{equation}\label{pkbk}
             \left|p_{k}(t)-b_{k}(t)\right|\lesssim \sqrt{M}|b_{k}|\lVert g\rVert_{\dot{H}^{1}}+\lVert \dot{g}\rVert_{L^{2}}\lVert g\rVert_{\dot{H}^{1}},
         \end{equation}
         and
         \begin{equation}\label{dervative of pk 1}
             \begin{aligned}
                 p'_{k}(t)\leq &-\kappa \sum_{j\neq k}|z_{j}-z_{k}|^{-3}\lambda^{\frac{3}{2}}_{j}\lambda^{\frac{1}{2}}_{k}\\
                 &-\frac{1}{\lambda_{k}\lVert\Lambda W\rVert^{2}_{L^{2}}}\int\limits_{|x-x_{k}|<R\lambda_{k}}\left(\left|\nabla g(x)\right|^{2}-f'\left(W_{\lambda_{k},x_{k}}\right)g^{2}\right){\rm{d}}x\\
                 & +O\Bigg(\frac{\eta}{\lambda_{k}}\left(\lVert \dot{g}\rVert^{2}_{L^{2}}+\lVert g\rVert^{2}_{\dot{H}^{1}}\right)+\sqrt{M}\frac{\lVert g\rVert_{\dot{H}^{1}}}{\lambda_{k}}\Bigg(\lambda^{3}_{k}+\sum_{j=1}^{K}b^{2}_{j}\Bigg)
                 \\
                 &\quad+\frac{1}{\sqrt{M}\lambda_{k}}\left(\sum_{j=1}^{K}|b_{j}|\right)\lVert \dot{g}\rVert_{L^{2}}+\lambda^{3}_{k}+\lambda^{2}_{k}\left(\sum_{j=1}^{K}|x_{j}-z_{j}|\right)+\sum_{j=1}^{K}b^{2}_{j}\Bigg).
             \end{aligned}
         \end{equation}
         where $\eta>0$ is sufficiently small and $R>0$ is sufficiently large, as specified in Lemma \ref{q(x)} and Lemma \ref{AK}.
     \end{lemma}
     \begin{proof} 
         \textit{Proof of (\ref{zetalam}).} By the definition of $\zeta_{k}$ in (\ref{refine lambda}),
         \begin{equation*}
             \begin{aligned}
                 \left|\zeta_{k}(t)-\lambda_{k}(t)\right|&=\frac{1}{\lVert \Lambda W\rVert^{2}_{L^{2}}}\left|\bigg\langle\chi\left(\frac{\cdot-x_{k}(t)}{\lambda_{k}(t) M}\right)(\Lambda W)_{\underline{\lambda_{k},x_{k}}},g\bigg\rangle\right|\\
                 &\lesssim \Bigg\lVert\chi\left(\frac{\cdot-x_{k}}{\lambda_{k}M}\right)(\Lambda W)_{\underline{\lambda_{k},x_{k}}}\Bigg\rVert_{L^{\frac{10}{7}}}\lVert g\rVert_{\dot{H}^{1}}
             \end{aligned}
         \end{equation*}
         and 
         \begin{equation*}
         \begin{aligned}
             \Bigg\lVert\chi\left(\frac{\cdot-x_{k}}{\lambda_{k}M}\right)(\Lambda W)_{\underline{\lambda_{k},x_{k}}}\Bigg\rVert_{L^{\frac{10}{7}}}&\lesssim\left(\int\limits_{|x-x_{k}|\leq 2\lambda_{k} M}\left|\frac{1}{\lambda_{k}^{\frac{5}{2}}}W\left(\frac{x-x_{k}}{\lambda_{k}}\right)\right|^{\frac{10}{7}}{\rm{d}} x\right)^{\frac{7}{10}}\\
             &=\left(\int\limits_{|y|\leq 2M} W^\frac{10}{7}(y)\lambda_{k}^{\frac{10}{7}}{\rm{d}}y\right)^{\frac{7}{10}}\lesssim \sqrt{M}\lambda_{k}.
             \end{aligned}
         \end{equation*}
         Thus, (\ref{zetalam}) follows.\\
         \textit{Proof of (\ref{zetabk}).} From (\ref{g}), we have
         \begin{align*}
             \zeta'_{k}(t)=&\lambda'_{k}(t)-\frac{1}{\lVert\Lambda W\rVert^{2}_{L^{2}}}\bigg(-\frac{\lambda'_{k}}{\lambda_{k}}\bigg\langle\chi\left(\frac{\cdot-x_{k}}{\lambda_{k}M}\right)(\underline{\Lambda}\Lambda W)_{\underline{\lambda_{k},x_{k}}},g\bigg\rangle\\
             &-\frac{x'_{k}}{\lambda_{k}}\cdot\bigg\langle\chi\left(\frac{\cdot-x_{k}}{\lambda_{k}M}\right)(\nabla \Lambda W)_{\underline{\lambda_{k},x_{k}}},g\bigg\rangle\\
             &-\frac{x'_{k}}{\lambda_{k}M}\cdot\bigg\langle\left(\nabla \chi\right)\left(\frac{\cdot-x_{k}}{\lambda_{k}M}\right)(\Lambda W)_{\underline{\lambda_{k},x_{k}}},g\bigg\rangle\\
             &-\frac{\lambda'_{k}}{\lambda_{k}}\bigg\langle(x\cdot\nabla \chi)\left(\frac{\cdot-x_{k}}{\lambda_{k}M}\right)(\Lambda W)_{\underline{\lambda_{k},x_{k}}},g\bigg\rangle+\bigg\langle\chi\left(\frac{\cdot-x_{k}}{\lambda_{k}M}\right)(\Lambda W)_{\underline{\lambda_{k},x_{k}}},\dot{g}\bigg\rangle\\
    &+\sum_{j=1}^{K}\left(b_{j}+\lambda'_{j}\right)\bigg\langle(\Lambda W)_{\underline{\lambda_{j},x_{j}}},\chi\left(\frac{\cdot-x_{k}}{\lambda_{k}M}\right)(\Lambda W)_{\underline{\lambda_{k},x_{k}}}\bigg\rangle\\
             &+\sum_{j=1}^{K} x'_{j}\cdot\bigg\langle(\nabla W)_{\underline{\lambda_{j},x_{j}}},\chi\left(\frac{\cdot-x_{k}}{\lambda_{k}M}\right)(\Lambda W)_{\underline{\lambda_{k},x_{k}}}\bigg\rangle\bigg).
         \end{align*}
         We now estimate the terms on the right-hand side separately. We keep the first term. For the second term, from (\ref{lamkbk}),
         \begin{equation*}
             \begin{aligned}
                 \left|\frac{\lambda'_{k}}{\lambda_{k}}\bigg\langle\chi\left(\frac{\cdot-x_{k}}{\lambda_{k}M}\right)(\underline{\Lambda} \Lambda W)_{\underline{\lambda_{k},x_{k}}},g\bigg\rangle\right|&\lesssim\left|\lambda'_{k}\right|\sqrt{M}\lVert g\rVert_{\dot{H}^{1}}\\
                 &\lesssim\sqrt{M}\left(|b_{k}|\lVert g\rVert_{\dot{H}^{1}}+\lVert \dot{g}\rVert_{L^{2}}\lVert g\rVert_{\dot{H}^{1}}\right).
             \end{aligned}
         \end{equation*}
         For the third through fifth terms, arguing in the same manner, we obtain
         \begin{equation*}
             \left|\frac{x'_{k}}{\lambda_{k}}\cdot\bigg\langle\chi\left(\frac{\cdot-x_{k}}{\lambda_{k}M}\right)(\nabla \Lambda W)_{\underline{\lambda_{k},x_{k}}}, g\bigg\rangle\right|\lesssim \lVert\dot{g}\rVert_{L^{2}}\lVert g\rVert_{\dot{H}^{1}}+\sum_{k=1}^{K}|b_{k}|\lVert g\rVert^{2}_{\dot{H}^{1}},
         \end{equation*}
         \begin{equation*}
             \left|\frac{x'_{k}}{\lambda_{k}M}\bigg\langle\left(\nabla \chi\right)\left(\frac{\cdot-x_{k}}{\lambda_{k}M}\right)(\Lambda W)_{\underline{\lambda_{k},x_{k}}},g\bigg\rangle\right|\lesssim\frac{1}{\sqrt{M}}\left(\lVert \dot{g}\rVert_{L^{2}}\lVert g\rVert_{\dot{H}^{1}}+\sum_{k=1}^{K}|b_{k}|\lVert g\rVert_{\dot{H}^{1}}\right),
         \end{equation*}
         \begin{equation*}
             \left|\frac{\lambda'_{k}}{\lambda_{k}}\bigg\langle(x\cdot\nabla \chi)\left(\frac{\cdot-x_{k}}{\lambda_{k}M}\right)(\Lambda W)_{\underline{\lambda_{k},x_{k}}},g\bigg\rangle\right|\lesssim\sqrt{M}\left(|b_{k}|\lVert g\rVert_{\dot{H}^{1}}+\lVert\dot{g}\rVert_{L^{2}}\lVert g\rVert_{\dot{H}^{1}}\right).
         \end{equation*}
         For the sixth term, from the orthogonality condition $\big\langle(\Lambda W)_{\underline{\lambda_{k},x_{k}}},\dot{g}\big\rangle=0$, 
         \begin{equation*}
             \begin{aligned}
                 &\left|\bigg\langle\chi\left(\frac{\cdot-x_{k}}{\lambda_{k}M}\right)(\Lambda W)_{\underline{\lambda_{k},x_{k}}},\dot{g}\bigg\rangle\right|\\
                 =&\left|\bigg\langle(1-\chi)\left(\frac{\cdot-x_{k}}{\lambda_{k}M}\right)(\Lambda W)_{\underline{\lambda_{k},x_{k}}},\dot{g}\bigg\rangle\right|\\
                 \lesssim&\left\lVert(1-\chi)\left(\frac{\cdot-x_{k}}{\lambda_{k}M}\right)(\Lambda W)_{\underline{\lambda_{k},x_{k}}}\right\rVert_{L^{2}}\lVert \dot{g}\rVert_{L^{2}}\\
                 \lesssim&\left(\int_{|y|\geq M}W^{2}(y){\rm{d}}y\right)^{\frac{1}{2}}\lVert\dot{g}\rVert_{L^{2}}\lesssim\frac{1}{\sqrt{M}}\lVert \dot{g}\rVert_{L^{2}}.
             \end{aligned}
         \end{equation*}
         For the last term, using \eqref{lamkbk} and the fact that $\big\langle\nabla W, \chi \Lambda W\big\rangle=0$,
         \begin{equation*}
             \begin{aligned}
                &\left| \sum_{j=1}^{K} x'_{j}\cdot\bigg\langle(\nabla W)_{\underline{\lambda_{j},x_{j}}},\chi\left(\frac{\cdot-x_{k}}{\lambda_{k}M}\right)(\Lambda W)_{\underline{\lambda_{k},x_{k}}}\bigg\rangle\right|\\
                \lesssim&\sum_{j\neq k}\left|x'_{j}\right|\big\langle\left|(\Lambda W)_{\underline{\lambda_{k},x_{k}}}\right|,\left|\nabla W\right|_{\underline{\lambda_{j},x_{j}}}\big\rangle\lesssim\lambda_{k}^{2}\left(\lVert \dot{g}\rVert_{L^{2}}+\sum_{k=1}^{K}|b_{k}|\lVert g\rVert_{\dot{H}^{1}}\right).
             \end{aligned}
         \end{equation*}
         Finally, we estimate
         \begin{align*}  &\left|\sum_{j=1}^{K}\left(b_{j}+\lambda'_{j}\right)\bigg\langle(\Lambda W)_{\underline{\lambda_{j},x_{j}}},\chi\left(\frac{\cdot-x_{k}}{\lambda_{k}M}\right)(\Lambda W)_{\underline{\lambda_{k},x_{k}}}\bigg\rangle-\lVert\Lambda W\rVert^{2}_{L^{2}}\left(b_{k}+\lambda'_{k}\right)\right|\\
        \leq&\left|b_{k}+\lambda'_{k}\right|\left|\bigg\langle\left(1-\chi\right)\left(\frac{\cdot-x_{k}}{\lambda_{k}M}\right)(\Lambda W)_{\underline{\lambda_{k},x_{k}}},(\Lambda W)_{\underline{\lambda_{k},x_{k}}}\bigg\rangle\right|\\
             &\quad+\sum_{j\neq k}\left|b_{j}+\lambda'_{j}\right|\left|\bigg\langle(\Lambda W)_{\underline{\lambda_{j},x_{j}}},\chi\left(\frac{\cdot-x_{k}}{\lambda_{k}M}\right)(\Lambda W)_{\underline{\lambda_{k},x_{k}}}\bigg\rangle\right|\\
             \lesssim&\sum_{j\neq k}\left|b_{j}+\lambda'_{j}\right|\lambda_{j}^{\frac{1}{2}}\lambda_{k}^{\frac{1}{2}}+\left|b_{k}+\lambda'_{k}\right|\frac{1}{M}\lesssim\frac{1}{M}\left(\lVert \dot{g}\rVert_{L^{2}}+\sum_{k=1}^{K}|b_{k}|\lVert g\rVert_{\dot{H}^{1}}\right).
         \end{align*}
         Combining the above estimates gives (\ref{zetabk}).\\
         \textit{Proof of (\ref{pkbk}).} By the definition of $p_{k}$ in (\ref{pk}) and Lemma \ref{AK},
         \begin{equation*}
             \left|p_{k}(t)-b_{k}(t)\right|\lesssim \sqrt{M}|b_{k}|\lVert g\rVert_{\dot{H}^{1}}+\lVert \dot{g}\rVert_{L^{2}}\lVert g\rVert_{\dot{H}^{1}}
         \end{equation*}
         which is (\ref{pkbk}).\\
         \textit{Proof of (\ref{dervative of pk 1}).} Differentiating \eqref{pk} yields
\begin{align*}
    	p'_{k}(t)=&b'_{k}(t)-\frac{1}{\lVert\Lambda W\rVert^{2}_{L^{2}}}\bigg(\frac{b_{k}}{\lambda_{k}}\left\langle\partial_{t}g,\chi\left(\frac{\cdot-x_{k}}{\lambda_{k}M}\right)(\underline{\Lambda}\Lambda W)_{\underline{\lambda_{k},x_{k}}}\right\rangle\notag\\
		&+\frac{b'_{k}}{\lambda_{k}}\left\langle g,\chi\left(\frac{\cdot-x_{k}}{\lambda_{k}M}\right)(\underline{\Lambda} \Lambda W)_{\underline{\lambda_{k},x_{k}}}\right\rangle\\
        &-\frac{b_{k}\lambda'_{k}}{\lambda_{k}^{2}}\left\langle g,\chi\left(\frac{\cdot-x_{k}}{\lambda_{k}M}\right)(\underline{\Lambda}\Lambda W)_{\underline{\lambda_{k},x_{k}}}\right\rangle\notag\\
		&-\frac{b_{k}\lambda'_{k}}{\lambda^{2}_{k}}\left\langle g, \chi\left(\frac{\cdot-x_{k}}{\lambda_{k}M}\right)(\underline{\Lambda}\underline{\Lambda}\Lambda W)_{\underline{\lambda_{k},x_{k}}}\right\rangle\\
		&-\frac{b_{k}x'_{k}}{\lambda_{k}^{2}}\cdot\left\langle g,\chi\left(\frac{\cdot-x_{k}}{\lambda_{k}M}\right)(\nabla \underline{\Lambda}\Lambda W)_{\underline{\lambda_{k},x_{k}}}\right\rangle\notag\\
		&-\frac{b_{k}x'_{k}}{\lambda^{2}_{k}M}\cdot\left\langle g,(\nabla \chi)\left(\frac{\cdot-x_{k}}{\lambda_{k}M}\right)(\underline{\Lambda}\Lambda W)_{\underline{\lambda_{k},x_{k}}}\right\rangle\notag\\
		&-\frac{b_{k}\lambda'_{k}}{\lambda_{k}^{2}}\left\langle(x\cdot \nabla \chi)\left(\frac{\cdot-x_{k}}{\lambda_{k}M}\right)(\underline{\Lambda}\Lambda W)_{\underline{\lambda_{k},x_{k}}}\right\rangle\bigg)+\frac{1}{\lVert\Lambda W\rVert^{2}_{L^{2}}}\bigg(\left\langle\partial_{t}\dot{g},\underline{A}_{k}g\right\rangle\notag\\
		&+\left\langle\dot{g},\underline{A}_{k}(\partial_{t}g)\right\rangle+\frac{\lambda'_{k}}{\lambda_{k}}\left\langle\dot{g},\lambda_{k}\partial_{\lambda_{k}}\underline{A}_{k} g\right\rangle+\frac{x'_{k}}{\lambda_{k}}\cdot\left\langle\dot{g},\lambda_{k}\partial_{x_{k}}\underline{A}_{k} g\right\rangle\bigg)\notag.
\end{align*}
             Substituting \eqref{g}, \eqref{dotg}, and (\ref{refined bk'}), we obtain
             \begin{align}\label{expand for pk'}
            	p_{k}'(t)=&-\kappa\sum_{j\neq k}|z_{j}-z_{k}|^{-3}\lambda_{k}^{\frac{1}{2}}\lambda_{j}^{\frac{3}{2}}-\frac{\lambda'_{k}}{\lambda_{k}}\frac{\langle(\underline{\Lambda}\Lambda W)_{\underline{\lambda_{k},x_{k}}},\dot{g}\rangle}{\lVert \Lambda W\rVert^{2}_{L^{2}}}-\frac{x'_{k}}{\lambda_{k}}\cdot\frac{\langle (\nabla \Lambda W)_{\underline{\lambda_{k},x_{k}}},\dot{g}\rangle}{\lVert\Lambda W\rVert^{2}_{L^{2}}}\notag\\
	&+\frac{1}{\lVert \Lambda W\rVert^{2}_{L^{2}}}\bigg\langle f\left(\sum_{k=1}^{K}W_{\lambda_{k},x_{k}}+g\right)-f\left(\sum_{k=1}^{K}W_{\lambda_{k},x_{k}}\right)-f'\left(\sum_{k=1}^{K}W_{\lambda_{k},x_{k}}\right)g,\notag\\
	&\left(\Lambda W\right)_{\underline{{\lambda_{k},x_{k}}}}\bigg\rangle+O\left(\lambda_{k}^{3}+\lambda_{k}^{2}\left(\sum_{j=1}^{K}|x_{j}-z_{j}|+\lVert g\rVert_{\dot{H}^{1}}\right)+\sum_{k=1}^{K}\left(b_{k}^{2}+|b_{k}|\lVert\dot{g}\rVert_{L^{2}}\right)\right)\notag\\
	&-\frac{1}{\lVert \Lambda W\rVert^{2}_{L^{2}}}
	\bigg(\frac{b_{k}}{\lambda_{k}}
	\bigg\langle\dot{g}
	+\sum_{j=1}^{K}(b_{j}+\lambda'_{j})(\Lambda W)_{\underline{\lambda_{j},x_{j}}}
	+\sum_{j=1}^{K}x'_{j}\cdot(\nabla W)_{\underline{\lambda_{j},x_{j}}},\notag\\
	&\chi\left(\frac{\cdot-x_{k}}{\lambda_{k}M}\right)(\underline{\Lambda}\Lambda W)_{\underline{\lambda_{k},x_{k}}}\bigg\rangle+\frac{b'_{k}}{\lambda_{k}}\left\langle g,\chi\left(\frac{\cdot-x_{k}}{\lambda_{k}M}\right)(\underline{\Lambda} \Lambda W)_{\underline{\lambda_{k},x_{k}}}\right\rangle\notag\\
	&-\frac{b_{k}\lambda'_{k}}{\lambda_{k}^{2}}\left\langle g,\chi\left(\frac{\cdot-x_{k}}{\lambda_{k}M}\right)(\underline{\Lambda}\Lambda W)_{\underline{\lambda_{k},x_{k}}}\right\rangle-\frac{b_{k}\lambda'_{k}}{\lambda^{2}_{k}}\left\langle g, \chi\left(\frac{\cdot-x_{k}}{\lambda_{k}M}\right)(\underline{\Lambda}\underline{\Lambda}\Lambda W)_{\underline{\lambda_{k},x_{k}}}\right\rangle\notag\\
	&-\frac{b_{k}x'_{k}}{\lambda_{k}^{2}}\cdot\left\langle g,\chi\left(\frac{\cdot-x_{k}}{\lambda_{k}M}\right)(\nabla \underline{\Lambda}\Lambda W)_{\underline{\lambda_{k},x_{k}}}\right\rangle\notag\\
	&-\frac{b_{k}x'_{k}}{\lambda^{2}_{k}M}\cdot\left\langle g,(\nabla \chi)\left(\frac{\cdot-x_{k}}{\lambda_{k}M}\right)(\underline{\Lambda}\Lambda W)_{\underline{\lambda_{k},x_{k}}}\right\rangle\notag\\
	&-\frac{b_{k}\lambda'_{k}}{\lambda_{k}^{2}}\left\langle g, ( x\cdot \nabla \chi)\left(\frac{\cdot-x_{k}}{\lambda_{k}M}\right)(\underline{\Lambda}\Lambda W)_{\underline{\lambda_{k},x_{k}}}\right\rangle\bigg)\notag\\
	&+\frac{1}{\lVert\Lambda W\rVert^{2}_{L^{2}}}\bigg(\bigg\langle\Delta g
	+f\left(\sum_{k=1}^{K} W_{\lambda_{k},x_{k}}+g\right)-\sum\limits_{k=1}^{K}f\left(W_{\lambda_{k},x_{k}}\right)\notag\\
	&-\sum\limits_{k=1}^{K}b'_{k}\left(\Lambda W\right)_{\underline{\lambda_{k},x_{k}}}+\sum\limits_{k=1}^{K}\frac{b_{k}\lambda'_{k}}{\lambda_{k}}\left({\underline{\Lambda}}\Lambda W\right)_{\underline{\lambda_{k},x_{k}}}+\sum_{k=1}^{K}\frac{b_{k}x'_{k}}{\lambda_{k}}\cdot \left(\nabla \Lambda W\right)_{\underline{\lambda_{k},x_{k}}},\underline{A}_{k}g\bigg\rangle\notag\\
	&+\left\langle\dot{g}, \underline{A}_{k}\left(\sum\limits_{k=1}^{K}(b_{k}+\lambda'_{k})\left(\Lambda W\right)_{\underline{\lambda_{k},x_{k}}}+\sum\limits_{k=1}^{K}x'_{k}\cdot\left(\nabla W\right)_{\underline{\lambda_{k},x_{k}}}+\dot{g}\right)\right\rangle\notag\\
	&+\frac{\lambda'_{k}}{\lambda_{k}}\left\langle\dot{g},\lambda_{k}\partial_{\lambda_{k}}\underline{A}_{k} g\right\rangle+\frac{x'_{k}}{\lambda_{k}}\cdot\left\langle\dot{g},\lambda_{k}\partial_{x_{k}}\underline{A}_{k} g\right\rangle\bigg).
             \end{align}
      We now estimate the terms on the right-hand side separately. First, using the orthogonality conditions $\left\langle\Lambda W,\underline{\Lambda}\Lambda W\right\rangle=0$ and $\left\langle\chi\underline{\Lambda}\Lambda W,\nabla W\right\rangle=0$,  we have
        \begin{align*}
		&\left|\frac{b_{k}}{\lambda_{k}}\left\langle\sum_{j=1}^{K}(b_{j}+\lambda'_{j})(\Lambda W)_{\underline{\lambda_{j},x_{j}}}
		+\sum_{j=1}^{K}x'_{j}\cdot(\nabla W)_{\underline{\lambda_{j},x_{j}}},
		\chi\left(\frac{\cdot-x_{k}}{\lambda_{k}M}\right)\left(\underline{\Lambda}\Lambda W\right)_{\underline{\lambda_{k},x_{k}}}
		\right\rangle\right|\\
		\lesssim&\frac{|b_{k}|}{\lambda_{k}}\bigg(\sum_{j\neq k}\left|b_{j}+\lambda'_{j}\right|\left|\left\langle(\Lambda W)_{\underline{\lambda_{j},x_{j}}},\chi\left(\frac{\cdot-x_{k}}{\lambda_{k}M}\right)\left(\underline{\Lambda}\Lambda W\right)_{\underline{\lambda_{k},x_{k}}}\right\rangle\right|\\
		&+\left|b_{k}+\lambda'_{k}\right|\left|\left\langle(\Lambda W)_{\underline{\lambda_{k},x_{k}}},(1-\chi)\left(\frac{\cdot-x_{k}}{\lambda_{k}M}\right)\left(\underline{\Lambda}\Lambda W\right)_{\underline{\lambda_{k},x_{k}}}\right\rangle\right|\\
		&+\sum_{j\neq k}\left|x'_{j}\right|\left|\left\langle(\nabla W)_{\underline{\lambda_{j},x_{j}}},\chi\left(\frac{\cdot-x_{k}}{\lambda_{k}M}\right)\left(\underline{\Lambda}\Lambda W\right)_{\underline{\lambda_{k},x_{k}}}\right\rangle\right|\bigg)\\
		\lesssim&\frac{|b_{k}|}{\lambda_{k}}\left(\left(\lVert \dot{g}\rVert_{L^{2}}+\sum_{k=1}^{K}|b_{k}|\lVert g\rVert_{\dot{H}^{1}}\right)\left(\lambda_{j}^{\frac{1}{2}}\lambda_{k}^{\frac{1}{2}}+\frac{1}{M}\right)\right)\\
		\lesssim&\frac{|b_{k}|}{M\lambda_{k}}\left(\lVert \dot{g}\rVert_{L^{2}}+\sum_{k=1}^{K}|b_{k}|\lVert g\rVert_{\dot{H}^{1}}\right).
	\end{align*}
    Then, by computations similar to those in the proof of \eqref{zetabk}, combined with \eqref{lamkbk} and \eqref{bk'}, we obtain
  	\begin{align*}
		&\left|\frac{b'_{k}}{\lambda_{k}}\left\langle g,\chi\left(\frac{\cdot-x_{k}}{\lambda_{k}M}\right)(\underline{\Lambda} \Lambda W)_{\underline{\lambda_{k},x_{k}}}\right\rangle\right|
		+\left|\frac{b_{k}\lambda'_{k}}{\lambda_{k}^{2}}\left\langle g,\chi\left(\frac{\cdot-x_{k}}{\lambda_{k}M}\right)(\underline{\Lambda}\Lambda W)_{\underline{\lambda_{k},x_{k}}}\right\rangle\right|+\\
		&\left|\frac{b_{k}\lambda'_{k}}{\lambda^{2}_{k}}\left\langle g, \chi\left(\frac{\cdot-x_{k}}{\lambda_{k}M}\right)(\underline{\Lambda}\underline{\Lambda}\Lambda W)_{\underline{\lambda_{k},x_{k}}}\right\rangle\right|
		+\left|\frac{b_{k}x'_{k}}{\lambda_{k}^{2}}\cdot\left\langle g,\chi\left(\frac{\cdot-x_{k}}{\lambda_{k}M}\right)(\nabla \underline{\Lambda}\Lambda W)_{\underline{\lambda_{k},x_{k}}}\right\rangle\right|\\
		&+\left|\frac{b_{k}x'_{k}}{\lambda^{2}_{k}M}\cdot\left\langle g,(\nabla \chi)\left(\frac{\cdot-x_{k}}{\lambda_{k}M}\right)(\underline{\Lambda}\Lambda W)_{\underline{\lambda_{k},x_{k}}}\right\rangle\right|\\
		&+\left|\frac{b_{k}\lambda'_{k}}{\lambda_{k}^{2}}\left\langle g,(x\cdot \nabla \chi)\left(\frac{\cdot-x_{k}}{\lambda_{k}M}\right)(\underline{\Lambda}\Lambda W)_{\underline{\lambda_{k},x_{k}}}\right\rangle\right|\\
		\lesssim&\sqrt{M}\Bigg[\left(\lambda_{k}^{2}+\frac{1}{\lambda_{k}}\left(\lVert \dot{g}\rVert^{2}_{L^{2}}+ \lVert g\rVert^{2}_{\dot{H}^{1}}\right)+\frac{1}{\lambda_{k}}\sum_{k=1}^{K}|b_{k}|\lVert\dot{g}\rVert_{L^{2}}
		+\sum_{k=1}^{K}b^{2}_{k}\right)\lVert g\rVert_{\dot{H}^{1}}\\
		&+\frac{b^{2}_{k}+|b_{k}|\left(\lVert\dot{g}\rVert_{L^{2}}+\sum\limits_{k=1}^{K}|b_{k}|\lVert g\rVert_{\dot{H}^{1}}\right)}{\lambda_{k}}\lVert g\rVert_{\dot{H}^{1}}\Bigg]\\
		\lesssim&\sqrt{M}\lVert g\rVert_{\dot{H}^{1}}\left(\lambda_{k}^{2}+\frac{b^{2}_{k}}{\lambda_{k}}+\sum_{k=1}^{K}b_{k}^{2}\right.\\
        &\left.+\frac{1}{\lambda_{k}}\left(\lVert \dot{g}\rVert^{2}_{L^{2}}+\lVert g\rVert^{2}_{\dot{H}^{1}}+\sum_{k=1}^{K}|b_{k}|\lVert\dot{g}\rVert_{L^{2}}+\sum_{k=1}^{K}b^{2}_{k}\lVert g\rVert_{\dot{H}^{1}}\right)\right).
	\end{align*}
From Lemma \ref{AK}, (\ref{lamkbk}), and  (\ref{bk'}), 
	\begin{align*}
		&\left|\frac{\lambda'_{k}}{\lambda_{k}}\left\langle\dot{g},\lambda_{k}\partial_{\lambda_{k}}\underline{A}_{k} g\right\rangle\right|+\left|\frac{x'_{k}}{\lambda_{k}}\cdot\left\langle\dot{g},\lambda_{k}\partial_{x_{k}}\underline{A}_{k} g\right\rangle\right|+\sum_{j=1}^{K}\left|b'_{j}\left\langle(\Lambda W)_{\underline{\lambda_{j},x_{j}}},\underline{A}_{k} g\right\rangle\right|\\
		&\sum_{j=1}^{K}\left|\frac{b_{j}\lambda'_{j}}{\lambda_{j}}\left\langle(\underline{\Lambda}\Lambda W)_{\underline{\lambda_{j},x_{j}}},\underline{A}_{k} g\right\rangle\right|+\sum_{j=1}^{K}\left|\frac{b_{j}x'_{j}}{\lambda_{j}}\cdot\left\langle(\nabla \Lambda W)_{\underline{\lambda_{j},x_{j}}},\underline{A}_{k}g\right\rangle\right|\\
\lesssim&\frac{|b_{k}|+\lVert\dot{g}\rVert_{L^{2}}+\sum\limits_{j=1}^{K}|b_{j}|\lVert g\rVert_{\dot{H}^{1}}}{\lambda_{k}}\lVert\dot{g}\rVert_{L^{2}}\lVert g\rVert_{\dot{H}^{1}}+\lambda_{k}^{2}\lVert g\rVert_{\dot{H}^{1}}+\frac{1}{\lambda_{k}}\left(\lVert\dot{g}\rVert^{2}_{L^{2}}+\lVert g\rVert^{2}_{\dot{H}^{1}}\right)\lVert g\rVert_{\dot{H}^{1}}\\
		&+\frac{1}{\lambda_{k}}\sum_{j=1}^{K}b^{2}_{j}\lVert g\rVert_{\dot{H}^{1}}+\frac{1}{\lambda_{k}}\sum_{j=1}^{K}|b_{j}|\lVert\dot{g}\rVert_{L^{2}}\lVert g\rVert_{\dot{H}^{1}}\\
		\lesssim &\lambda_{k}^{2}\lVert g\rVert_{\dot{H}^{1}}+\frac{1}{\lambda_{k}}\left(\lVert\dot{g}\rVert^{2}_{L^{2}}+\lVert g\rVert^{2}_{\dot{H}^{1}}\right)\lVert g\rVert_{\dot{H}^{1}}
		+\frac{1}{\lambda_{k}}\sum_{j=1}^{K}b^{2}_{j}\lVert g\rVert_{\dot{H}^{1}}.
		%+\frac{1}{\lambda_{k}}\sum_{j=1}^{K}|b_{j}|\lVert\dot{g}\rVert_{L^{2}}\lVert g\rVert_{\dot{H}^{1}}
	\end{align*}
    Substituting the above estimates into \eqref{expand for pk'} and using $\langle\dot{g},\underline{A}_{k}\dot{g}\rangle=0$ (via integration by parts), we have
	\begin{align}\label{der for pk}
		p'_{k}(t)=&-\kappa\sum_{j\neq k}|z_{j}-z_{k}|^{-3}\lambda_{j}^{\frac{3}{2}}\lambda_{k}^{\frac{1}{2}}+\frac{1}{\lVert \Lambda W\rVert^{2}_{L^{2}}}\left((I)+(II)+(III)+(IV)\right)\notag\\
		&+O\Biggl(\sqrt{M}\frac{\lVert g\rVert_{\dot{H}^{1}}}{\lambda_{k}}\Biggl(\lambda^{3}_{k}+\sum_{j=1}^{K}b^{2}_{j}+\lVert \dot{g}\rVert^{2}_{L^{2}}+\lVert g\rVert^{2}_{\dot{H}^{1}}\Biggl)\notag\\
		&+\frac{1}{M\lambda_{k}}\sum_{j=1}^{K}|b_{j}|\lVert\dot{g}\rVert_{L^{2}}
		+\lambda_{k}^{3}+\lambda_{k}^{2}\Biggl(\sum_{j=1}^{K}|x_{j}-z_{j}|\Biggl)+\sum_{j=1}^{K}b^{2}_{j}\Biggl),
	\end{align}
    where
\begin{align*}
	(I)=&-\frac{\lambda'_{k}}{\lambda_{k}}\langle(\underline{\Lambda}\Lambda W)_{\underline{\lambda_{k},x_{k}}},\dot{g}\rangle-\frac{b_{k}}{\lambda_{k}}\left\langle\chi\left(\frac{\cdot-x_{k}}{\lambda_{k}M}\right)(\underline{\Lambda}\Lambda W)_{\underline{\lambda_{k},x_{k}}},\dot{g}\right\rangle\\
	&+\sum_{j=1}^{K}\left(b_{j}+\lambda'_{j}\right)\left\langle\dot{g},\underline{A}_{k}(\Lambda W)_{\underline{\lambda_{j},x_{j}}}\right\rangle,\\
	(II)=&-\frac{x'_{k}}{\lambda_{k}}\cdot\langle(\nabla \Lambda W)_{\underline{\lambda_{k},x_{k}}},\dot{g}\rangle+\sum_{j=1}^{K}x'_{j}\cdot\left\langle\dot{g},\underline{A}_{k}(\nabla W)_{\underline{\lambda_{j},x_{j}}}\right\rangle,\\
	(III)=&\bigg\langle f\left(\sum_{k=1}^{K}W_{\lambda_{k},x_{k}}+g\right)-f\left(\sum_{k=1}^{K}W_{\lambda_{k},x_{k}}\right)-f'\left(\sum_{k=1}^{K}W_{\lambda_{k},x_{k}}\right)g,
	\left(\Lambda W\right)_{\underline{{\lambda_{k},x_{k}}}}\bigg\rangle,\\
	(IV)=&\left\langle\Delta g+f\left(\sum_{j=1}^{K}W_{\lambda_{j},x_{j}}+g\right)-\sum_{j=1}^{K}f\left( W_{\lambda_{j},x_{j}}\right),\underline{A}_{k}g\right\rangle.
	\end{align*}
 For (I), we have
	\begin{align*}
		|(I)|\lesssim&\frac{|b_{k}+\lambda'_{k}|}{\lambda_{k}}\left|\left\langle\dot{g},\underline{A}_{k}(\Lambda W)_{\lambda_{k},x_{k}}-(\underline{\Lambda}\Lambda W)_{\underline{\lambda_{k},x_{k}}}\right\rangle\right|\\
		& +\sum_{j\neq k}\left|b_{j}+\lambda'_{j}\right|\left|\langle\dot{g},\underline{A}_{k}(\Lambda W)_{\underline{\lambda_{j},x_{j}}}\rangle\right|+\frac{|b_{k}|}{\lambda_{k}}\left|\left\langle\left(1-\chi\right)\left(\frac{\cdot-x_{k}}{\lambda_{k}M}\right)\left(\underline{\Lambda}\Lambda W\right)_{\underline{\lambda_{k},x_{k}}},\dot{g}\right\rangle\right|,
		\end{align*}
        Combining this with \eqref{lamkbk}, \eqref{AK 2}, and \eqref{AK 4} gives
	\begin{align*}
		|(I)|	\lesssim&\frac{\eta}{\lambda_{k}}\lVert\dot{g}\rVert_{L^{2}}\left(\lVert\dot{g}\rVert_{L^{2}}+\sum_{k=1}^{K}|b_{k}|\lVert g\rVert_{\dot{H}^{1}}\right)\\
		&+\left(\lVert\dot{g}\rVert_{L^{2}}+\sum_{k=1}^{K}|b_{k}|\lVert g\rVert_{\dot{H}^{1}}\right)\lVert\dot{g}\rVert_{L^{2}}\lambda_{k}^{-\frac{1}{2}}
		+\frac{|b_{k}|}{\lambda_{k}\sqrt{M}}\lVert\dot{g}\rVert_{L^{2}}\\
		\lesssim& \frac{\eta}{\lambda_{k}}\lVert \dot{g}\rVert^{2}_{L^{2}}+\frac{1}{\sqrt{M}}\sum_{k=1}^{K}\frac{|b_{k}|}{\lambda_{k}}\lVert\dot{g}\rVert_{L^2}.
	\end{align*}
    For  (II), using \eqref{lamkbk}, \eqref{AK 2}, and \eqref{AK 4} again, we have
	\begin{align*}
		|(II)|&\lesssim\frac{|x'_{k}|}{\lambda_{k}}\left|\left\langle\dot{g},(\nabla \Lambda W)_{\underline{\lambda_{k},x_{k}}}-\underline{A}_{k}\left(\nabla W\right)_{\lambda_{k},x_{k}}\right\rangle\right|+\sum_{j\neq k}|x'_{j}|\left|\left\langle A_{k}(\nabla W)_{\underline{\lambda_{j},x_{j}}},\dot{g}\right\rangle\right|\\
		&\lesssim \frac{\eta}{\lambda_{k}}\left(\lVert\dot{g}\rVert^{2}_{L^{2}}+\sum_{j=1}^{K}|b_{j}|\lVert\dot{g}\rVert_{L^2}\lVert g\rVert_{\dot{H}^{1}}\right).
	\end{align*}
    We now estimate term (IV), which can be rewritten as
\begin{align*}
	(IV)=&\left\langle\Delta g,\underline{A}_{k} g\right\rangle+\left\langle f\left(\sum_{j=1}^{K}W_{\lambda_{j},x_{j}}\right)-\sum_{j=1}^{K}f\left(W_{\lambda_{j},x_{j}}\right),\underline{A}_{k} g\right\rangle\\
	&+\left\langle f\left(\sum_{j=1}^{K}W_{\lambda_{j},x_{j}}+g\right)-f\left(\sum_{j=1}^{K}W_{\lambda_{j},x_{j}}\right),\underline{A}_{k}g\right\rangle,
\end{align*}
Note from \eqref{nonlinear 1}, Lemma \ref{WW} and the boundedness of $\underline{A}_{k}:\dot{H}^{1}\rightarrow L^{2}$, the term on the second line can be estimated as follows,
\begin{equation*}
	\left|\left\langle f\left(\sum_{j=1}^{K}W_{\lambda_{j},x_{j}}\right)-\sum_{j=1}^{K}f\left(W_{\lambda_{j},x_{j}}\right),\underline{A}_{k} g\right\rangle\right|\lesssim\sum_{l\neq k}\left\lVert W_{\lambda_{l},x_{l}}W^{\frac{4}{3}}_{\lambda_{k},x_{k}}\right\rVert_{L^{2}}\lVert g\rVert_{\dot{H}^{1}}\lesssim\lambda_{k}^{2}\lVert g\rVert_{\dot{H}^{1}}.
\end{equation*}
Combining this with \eqref{cor loc} yields
\begin{align*}
	(IV)\leq& \frac{\eta}{\lambda_{k}}\lVert g\rVert^{2}_{\dot{H}^{1}}-\frac{1}{\lambda_{k}}\int\limits_{|x-x_{k}|<R\lambda_{k}}\left|\nabla g(x)\right|^{2}{\rm{d}}x +O\left(\lambda_{k}^{2}\lVert g\rVert_{\dot{H}^{1}}\right)\\
&	+\left\langle f\left(\sum_{j=1}^{K}W_{\lambda_{j},x_{j}}+g\right)-f\left(\sum_{j=1}^{K}W_{\lambda_{j},x_{j}}\right),\underline{A}_{k}g\right\rangle.
\end{align*}
Hence, from (\ref{sym}),
	\begin{align*}
		&(IV)\leq \frac{\eta}{\lambda_{k}}\lVert g\rVert^{2}_{\dot{H}^{1}}-\frac{1}{\lambda_{k}}\int\limits_{|x-x_{k}|<R\lambda_{k}}\left|\nabla g(x)\right|^{2}{\rm{d}}x +O\left(\lambda_{k}^{2}\lVert g\rVert_{\dot{H}^{1}}\right)\\
		&-\left\langle f\left(\sum_{k=1}^{K}W_{\lambda_{k},x_{k}}+g\right)-f\left(\sum_{k=1}^{K}W_{\lambda_{k},x_{k}}\right)-f'\left(\sum_{k=1}^{K} W_{\lambda_{k},x_{k}}\right)g,A_{k}\left(\sum_{j=1}^{K}W_{\lambda_{j},x_{j}}\right)\right\rangle\\
		&+\frac{1}{5}\left\langle f\left(\sum_{k=1}^{K}W_{\lambda_{k},x_{k}}+g\right)-f\left(\sum_{k=1}^{K}W_{\lambda_{k},x_{k}}\right),
		\frac{1}{\lambda_{k}}\Delta q\left(\frac{\cdot-x_{k}}{\lambda_{k}}\right)g\right\rangle\\
		&= \frac{\eta}{\lambda_{k}}\lVert g\rVert^{2}_{\dot{H}^{1}}-\frac{1}{\lambda_{k}}\int\limits_{|x-x_{k}|<R\lambda_{k}}\left(\left|\nabla g(x)\right|^{2}-f'\left(W_{\lambda_{k},x_{k}}\right)g^{2}\right){\rm{d}}x +O\left(\lambda_{k}^{2}\lVert g\rVert_{\dot{H}^{1}}\right)\\
		& -\left\langle f\left(\sum_{k=1}^{K}W_{\lambda_{k},x_{k}}+g\right)-f\left(\sum_{k=1}^{K}W_{\lambda_{k},x_{k}}\right)
		-f'\left(\sum_{k=1}^{K}W_{\lambda_{k},x_{k}}\right)g,
		\left(\Lambda W\right)_{\underline{{\lambda_{k},x_{k}}}}\right\rangle\\
		&+(1)+(2)+(3)+(4)+(5).
	\end{align*}
where
\begin{align*}
	(1)=&-\left\langle f\left(\sum_{k=1}^{K}W_{\lambda_{k},x_{k}}+g\right)-f\left(\sum_{k=1}^{K}W_{\lambda_{k},x_{k}}\right)
	-f'\left(\sum_{k=1}^{K}W_{\lambda_{k},x_{k}}\right)g,\right.\\
	&A_{k}W_{\lambda_{k},x_{k}}-(\Lambda W)_{\underline{\lambda_{k},x_{k}}}\Bigg\rangle,\\
	(2)=&-\left\langle f\left(\sum_{k=1}^{K}W_{\lambda_{k},x_{k}}+g\right)-f\left(\sum_{k=1}^{K}W_{\lambda_{k},x_{k}}\right)
	-f'\left(\sum_{k=1}^{K}W_{\lambda_{k},x_{k}}\right)g,\right.\\
	&\left.A_{k}\left(\sum_{j\neq k}W_{\lambda_{j},x_{j}}\right)\right\rangle,\\
	(3)=&\frac{1}{5}\left\langle f\left(\sum_{k=1}^{K}W_{\lambda_{k},x_{k}}+g\right)-f\left(\sum_{k=1}^{K}W_{\lambda_{k},x_{k}}\right)
	-f'\left(\sum_{k=1}^{K}W_{\lambda_{k},x_{k}}\right)g,\right.\\
	&\frac{1}{\lambda_{k}}\Delta q\left(\frac{\cdot-x_{k}}{\lambda_{k}}\right)g\Bigg\rangle,\\
	(4)=&\frac{1}{5}\left\langle\left(f'\left(\sum_{k=1}^{K}W_{\lambda_{k},x_{k}}\right)-f'\left(W_{\lambda_{k},x_{k}}\right)\right)g,\frac{1}{\lambda_{k}}\Delta q\left(\frac{\cdot-x_{k}}{\lambda_{k}}\right)g\right\rangle,\\
	(5)=&\frac{1}{5}\left\langle f'\left(W_{\lambda_{k},x_{k}}\right)g,\frac{1}{\lambda_{k}}\Delta q\left(\frac{\cdot-x_{k}}{\lambda_{k}}\right)g\right\rangle-\frac{1}{\lambda_{k}}\int\limits_{|x-x_{k}|<R\lambda_{k}}f'\left(W_{\lambda_{k},x_{k}}\right)g^{2}{\rm{d}}x.
\end{align*}
For (1), using (\ref{nonlinear 3}) and (\ref{AK 3}),
	\begin{align*}
		\left|(1)\right|&\lesssim\int\left(\sum_{k=1}^{K}\left|W_{\lambda_{k},x_{k}}\right|^{\frac{1}{3}}+|g|^{\frac{1}{3}}\right)|g|^{2}\left|A_{k}W_{\lambda_{k},x_{k}}-(\Lambda W)_{\underline{\lambda_{k},x_{k}}}\right|{\rm{d}}x\\
		&\lesssim\left(\sum_{k=1}^{K}\left\lVert W_{\lambda_{k},x_{k}}\right\rVert^{\frac{1}{3}}_{\dot{H}^{1}}+\lVert g\rVert_{\dot{H}^{1}}^{\frac{1}{3}}\right)\lVert g\rVert^{2}_{\dot{H}^{1}}\left\lVert A_{k}W_{\lambda_{k},x_{k}}-(\Lambda W)_{\underline{\lambda_{k},x_{k}}}\right\rVert_{L^{\frac{10}{3}}}\lesssim \frac{\eta}{\lambda_{k}}\lVert g\rVert^{2}_{\dot{H}^{1}}.
	\end{align*}
For (2),  using (\ref{nonlinear 3}) and (\ref{AK 5}),
	\begin{equation*}
		\left|(2)\right|\lesssim\left(\sum_{k=1}^{K}\left\lVert W_{\lambda_{k},x_{k}}\right\rVert^{\frac{1}{3}}_{\dot{H}^{1}}+\lVert g\rVert_{\dot{H}^{1}}^{\frac{1}{3}}\right)\lVert g\rVert^{2}_{\dot{H}^{1}}\sum_{j\neq k}\left\lVert A_{k} W_{\lambda_{j},x_{j}}\right\rVert_{L^{\frac{10}{3}}}\lesssim \lambda_{k}^{-\frac{1}{2}}\lVert g\rVert ^{2}_{\dot{H}^{1}}.
	\end{equation*}
    For (3), using the fact that $\Delta q(x)$ is bounded and (\ref{nonlinear 3}), we have
	\begin{equation*}
		\left|(3)\right|\lesssim\frac{1}{\lambda_{k}}\lVert g\rVert^{3}_{\dot{H}^{1}}.
	\end{equation*}
    For (4), using \eqref{nonlinear 4} and the fact that $\Delta q(x)$ is supported on a ball of radius $\tilde{R}$,
	\begin{align*}
		\left|(4)\right|&\lesssim \frac{1}{\lambda_{k}}\int_{|x-x_{k}|\leq \tilde{R}\lambda_{k}}\left|f'\left(\sum_{j=1}^{K} W_{\lambda_{j},x_{j}}\right)- f'\left(W_{\lambda_{k},x_{k}}\right)\right|g^{2}{\rm{d}}x\\
		&\lesssim\frac{1}{\lambda_{k}}\lVert g\rVert^{2}_{\dot{H}^{1}}\left\lVert\left(\sum_{j\neq k}W_{\lambda_{j},x_{j}}\right)\sum_{j=1}^{K}\left|W_{\lambda_{j},x_{j}}\right|^{\frac{1}{3}}\right\rVert_{L^{\frac{5}{2}}(|x-x_{k}|\leq \tilde{R}\lambda_{k})}\\
		&\lesssim\frac{1}{\lambda_{k}}\lVert g\rVert^{2}_{\dot{H}^{1}}\sum_{j\neq k}\left\lVert W_{\lambda_{j},x_{j}}\right\rVert_{L^{\frac{10}{3}}(|x-x_{k}|\leq \tilde{R}\lambda_{k})}\lesssim \lambda_{k}^{\frac{1}{2}}\lVert g\rVert^{2}_{\dot{H}^{1}}.
	\end{align*}
    For the last term, by the definition of $q$ in Lemma \ref{q(x)}, we have $\Delta q(x)=5$ for $|x|\leq R$, and $\Delta q(x)$ is bounded for all $x\in\RR^{5}$. Thus,
	\begin{align*}
		|(5)|&\lesssim \int_{|x-x_{k}|\geq R\lambda_{k}}\frac{1}{\lambda_{k}}\left|f'\left(W_{\lambda_{k},x_{k}}\right)\right|g^{2}{\rm{d}}x\\
		&\lesssim\frac{1}{\lambda_{k}}\left\lVert f'\left(W_{\lambda_{k},x_{k}}\right)\right\rVert_{L^{\frac{5}{2}}(|x-x_{k}|\geq R \lambda_{k})}\lVert g\rVert^{2}_{\dot{H}^{1}}\\
		&\lesssim\frac{R^{-2}}{\lambda_{k}}\lVert g\rVert^{2}_{\dot{H}^{1}}\lesssim\frac{\eta}{\lambda_{k}}\lVert g\rVert^{2}_{\dot{H}^{1}},
	\end{align*}
    provided that $R>0$ is chosen sufficiently large in the definition of $q$ in Lemma~\ref{q(x)}. Hence,
	\begin{align*}
		(IV)\leq& -\frac{1}{\lambda_{k}}\int\limits_{|x-x_{k}|<R\lambda_{k}}\left(\left|\nabla g(x)\right|^{2}-f'\left(W_{\lambda_{k},x_{k}}\right)g^{2}\right){\rm{d}}x +O\left(\lambda_{k}^{2}\lVert g\rVert_{\dot{H}^{1}}+\frac{\eta}{\lambda_{k}}\lVert g\rVert^{2}_{\dot{H}^{1}}\right)\\
		&-\left\langle f\left(\sum_{k=1}^{K}W_{\lambda_{k},x_{k}}+g\right)-f\left(\sum_{k=1}^{K}W_{\lambda_{k},x_{k}}\right)
		-f'\left(\sum_{k=1}^{K}W_{\lambda_{k},x_{k}}\right)g,
		\left(\Lambda W\right)_{\underline{{\lambda_{k},x_{k}}}}\right\rangle.
	\end{align*}
    Substituting this into \eqref{der for pk}, the second line cancels with term (III) in \eqref{der for pk}, and we obtain
	\begin{align*}
		p'_{k}(t)\leq&-\kappa\sum_{j\neq k}|z_{j}-z_{k}|^{-3}\lambda_{j}^{\frac{3}{2}}\lambda_{k}^{\frac{1}{2}}\\
		&-\frac{1}{\lambda_{k}\lVert\Lambda W\rVert^{2}_{L^{2}}}\int_{|x-x_{k}|<R\lambda_{k}}\left(\left|\nabla g(x)\right|^{2}-f'\left(W_{\lambda_{k},x_{k}}\right)g^{2}\right){\rm{d}}x \\
		&+O\left(\frac{\eta}{\lambda_{k}}\left(\lVert g\rVert^{2}_{\dot{H}^{1}}+\lVert\dot{g}\rVert^{2}_{L^{2}}\right)+\sqrt{M}\frac{\lVert g\rVert_{\dot{H}^{1}}}{\lambda_{k}}\left(\lambda_{k}^{3}+\sum_{j=1}^{K}b_{j}^{2}\right)+\frac{1}{\sqrt{M}\lambda_{k}}\sum_{j=1}^{K}|b_{j}|\lVert 
		\dot{g}\rVert_{L^{2}}\right.\\
		&\left.+\lambda_{k}^{3}+\sum_{j=1}^{K}b_{j}^{2}+\lambda_{k}^{2}\left(\sum_{j=1}^{K}|x_{j}-z_{j}|\right)\right).
	\end{align*}
    This completes the proof of (\ref{dervative of pk 1}).
    \end{proof}

     \section{Energy expansion}\label{energy expansion}
     In this section, using energy conservation, we derive additional estimates for the modulation parameters and the remainder terms $\vec{g}(t)$.

First, we compute the energy of the approximate solution: 
 \begin{lemma}
     Let $\lambda_{k}$, $x_{k}$, $b_{k}$ be defined as in Lemma \ref{Mod}, then 
     \begin{align}\label{energy expand}
             &E\left(\sum_{k=1}^{K} W_{\lambda_{k},x_{k}},\sum_{k=1}^{K}b_{k}\left(\Lambda W\right)_{\underline{\lambda_{k},x_{k}}}\right)\notag\\
             =&\frac{1}{2}\lVert\Lambda W\rVert^{2}_{L^{2}}\sum_{k=1}^{K}|b_{k}|^{2}+KE(W,0)-
             \left(15^{\frac{3}{2}}\int W^{\frac{7}{3}}\right)\sum_{1\leq j<k\leq K}|z_{j}-z_{k}|^{-3}\lambda^{\frac{3}{2}}_{k}\lambda_{j}^{\frac{3}{2}}\notag\\
             &+O\Bigg(\lambda_{k}^{4}+\lambda_{k}^{3}\left(\sum_{k=1}^{K}|x_{k}-z_{k}|\right)+
             \left(\sum_{k=1}^{K}|b_{k}|^{2}\right)\lambda_{k}\Bigg).
     \end{align}
 \end{lemma}
 \begin{proof}
     By the definition of the energy $E$, we have
     \begin{align*}
         &E\left(\sum_{k=1}^{K} W_{\lambda_{k},x_{k}},\sum_{k=1}^{K}b_{k}\left(\Lambda W\right)_{\underline{\lambda_{k},x_{k}}}\right)\\
         =&\int_{\RR^{5}}\frac{1}{2}\left|\sum_{k=1}^{K}b_{k}\left(\Lambda W\right)_{\underline{\lambda_{k},x_{k}}}\right|^{2}+\frac{1}{2}\left|\nabla\left(\sum_{k=1}^{K}W_{\lambda_{k},x_{k}}\right)\right|^{2}-F\left(\sum_{k=1}^{K}W_{\lambda_{k},x_{k}}\right){\rm{d}}x\\
         =&\frac{1}{2}\lVert \Lambda W\rVert^{2}_{L^{2}}\sum_{k=1}^{K}|b_{k}|^{2}+KE(W,0)+\sum_{1\leq k<j\leq K}b_{k}b_{j}\int_{\RR^{5}}\left(\Lambda W\right)_{\underline{\lambda_{k},x_{k}}}\left(\Lambda W\right)_{\underline{\lambda_{j},x_{j}}}{\rm{d}}x\\
         &+\sum_{1\leq k<j\leq K}\int_{\RR^{5}}\nabla W_{\lambda_{j},x_{j}}\cdot \nabla W_{\lambda_{k},x_{k}}{\rm{d}}x+\int_{\RR^{5}}\sum_{k=1}^{K}F\left(W_{\lambda_{k},x_{k}}\right)-F\left(\sum_{k=1}^{K} W_{\lambda_{k},x_{k}}\right){\rm{d}}x.
     \end{align*}
     Using the fact that $\displaystyle \Delta W_{\lambda_{k},x_{k}}+f\left(W_{\lambda_{k},x_{k}}\right)=0$ and integrating by parts, we obtain
     \begin{align*}
         &E\left(\sum_{k=1}^{K} W_{\lambda_{k},x_{k}},\sum_{k=1}^{K}b_{k}\left(\Lambda W\right)_{\underline{\lambda_{k},x_{k}}}\right)\\
         =&\frac{1}{2}\lVert \Lambda W\rVert^{2}_{L^{2}}\sum_{k=1}^{K}|b_{k}|^{2}+KE(W,0)\\
         &-\int_{\RR^{5}}\left(F\left(\sum_{k=1}^{K}W_{\lambda_{k},x_{k}}\right)-\sum_{k=1}^{K}F\left(W_{\lambda_{k},x_{k}}\right)-\sum_{1\leq k<j\leq K}W_{\lambda_{j},x_{j}}f\left(W_{\lambda_{k},x_{k}}\right)\right){\rm{d}}x\\
         &+\sum_{1\leq k<j\leq K}b_{k}b_{j}\int_{\RR^{5}}\left(\Lambda W\right)_{\underline{\lambda_{k},x_{k}}}\left(\Lambda W\right)_{\underline{\lambda_{j},x_{j}}}{\rm{d}}x.
     \end{align*}
     For the second line, we first note from (\ref{nonlinear 5}) that 
     \begin{equation*}
         \left|F\left(\sum_{k=1}^{K}W_{\lambda_{k},x_{k}}\right)-\sum_{k=1}^{K}F\left(W_{\lambda_{k},x_{k}}\right)-\sum_{k=1}^{K}f\left(W_{\lambda_{k},x_{k}}\right)\sum_{j\neq k}W_{\lambda_{j},x_{j}}\right|\lesssim\sum_{j\neq k}W^{2}_{\lambda_{j},x_{j}}W^{\frac{4}{3}}_{\lambda_{k},x_{k}}.
     \end{equation*}
      Hence, from Lemma \ref{WW},
 \begin{align*}   &\Bigg|\int_{\RR^{5}}\left(F\left(\sum_{k=1}^{K}W_{\lambda_{k},x_{k}}\right)-\sum_{k=1}^{K}F\left(W_{\lambda_{k},x_{k}}\right)-\sum_{1\leq k<j\leq K}W_{\lambda_{j},x_{j}}f\left(W_{\lambda_{k},x_{k}}\right)\right){\rm{d}}x\\
     &-\sum_{1\leq j<k\leq K}\int_{\RR^{5}}W_{\lambda_{j},x_{j}}f\left(W_{\lambda_{k},x_{k}}\right){\rm{d}}x\Bigg|\lesssim\sum_{j\neq k}\int_{\RR^{5}} W^{2}_{\lambda_{j},x_{j}}W_{\lambda_{k},x_{k}}^{\frac{4}{3}}{\rm{d}}x\lesssim \lambda_{k}^{4}.
 \end{align*}
 For the term $\displaystyle\left\langle W_{\lambda_{j},x_{j}}, f\left(W_{\lambda_{k},x_{k}}\right)\right\rangle\ (j\neq k)$, computations similar to those in the proof of \eqref{interaction} give
 \begin{align*}
     \left\langle W_{\lambda_{j},x_{j}},f\left(W_{\lambda_{k},x_{k}}\right)\right\rangle=&15^{\frac{3}{2}}\int_{\RR^{5}}W^{\frac{7}{3}}|z_{j}-z_{k}|^{-3}\lambda_{k}^{\frac{3}{2}}\lambda_{j}^\frac{3}{2}\\
     &+O\left(\lambda_{k}^{4}+\lambda_{j}^{3}\left(|x_{j}-z_{j}|+|x_{k}-z_{k}|\right)\right).
 \end{align*}
 For the third line, by Lemma \ref{WW}, we have
 \begin{equation*}
    \left| \sum_{1\leq k<j\leq K}b_{k}b_{j}\int_{\RR^{5}}\left(\Lambda W\right)_{\underline{\lambda_{k},x_{k}}}\left(\Lambda W\right)_{\underline{\lambda_{j},x_{j}}}{\rm{d}}x\right|\lesssim\lambda_{k}\left(\sum_{k=1}^{K}b_{k}^{2}\right).
 \end{equation*}
 Finally, combining the above estimates gives (\ref{energy expand}).
   \end{proof}
\begin{remark}
It follows from \eqref{Bub1} and \eqref{asp} that there exists $T_{0}>0$ such that, for all $t\geq T_{0}$, $\vec{u}(t)$ and the parameters $\mu_{k}(t)$, $y_{k}(t)$ satisfy the assumptions of Lemma \ref{Mod}, and hence there exist unique $C^{1}([T_{0},+\infty))$ modulated parameters 
$\lambda_{k}(t)>0$, $x_{k}(t)\in\RR^{5}$, $b_{k}(t)\in\RR$ $(1\leq k\leq K)$ satisfying (\ref{ort1}), (\ref{ort2}), (\ref{ort3}), (\ref{lamklamj}). Moreover,
 \begin{equation}\label{asy}
    \lVert g(t)\rVert_{\dot{H}^{1}} +\lVert \dot{g}(t)\rVert_{L^{2}}+\sum_{k=1}^{K}\lambda_{k}(t)+\sum_{k=1}^{K}|x_{k}(t)-z_{k}|+\sum_{k=1}^{K}|b_{k}(t)|\rightarrow 0\ {\rm{as}}\ t\rightarrow+\infty.
 \end{equation}
and there exists a fixed $A>0$ such that
\begin{equation}\label{lamksimlamj}
    A \lambda_{k}(t)\leq \lambda_{1}(t)\leq\frac{1}{A}\lambda_{k}(t),\quad \forall\ 1\leq k\leq K\ {\rm{and}}\ \forall\ t\geq T_{0}.
\end{equation}
\end{remark}
\vspace{0.5cm}

Now, for $1\leq k\leq K$, we denote 
 \begin{equation}\label{negative directions}
   a_{k}^{+}=\langle \alpha^{+}_{\lambda_{k},x_{k}},\vec{g} \rangle\quad {\rm{and}}\quad a_{k}^{-}=\langle \alpha^{-}_{\lambda_{k},x_{k}},\vec{g}\rangle.
 \end{equation}
 Then, by the coercivity estimates (\ref{cor local}), for $\vec{g}$ as in Lemma \ref{Mod}, we have
 \begin{equation}\label{cor loc1}
 \begin{aligned}
     \int\limits_{|x-x_{k}|\leq R\lambda_{k}}\left(\left|\nabla g\right|^{2}-f'(W_{\lambda_{k},x_{k}})g^{2}\right) {\rm{d}}x&\geq -\eta \lVert\nabla g\rVert^{2}_{L^{2}}-\nu^{2}\big\langle\lambda^{-2}_{k}Y_{\lambda_{k},x_{k}},g\big\rangle^{2}\\
    & \geq-\eta\lVert\nabla g\rVert^{2}_{L^{2}}-C\left((a^{+}_{k})^{2}+(a^{-}_{k})^{2}\right),
 \end{aligned}
 \end{equation}
 where $\eta>0$ can be taken arbitrarily small and $R=R(\eta)>0$.
 For $\delta>0$ sufficiently small in Lemma \ref{Mod}, from the orthogonality conditions \eqref{ort1}, \eqref{ort2}, and the coercivity estimate \eqref{cor mul},
 \begin{equation}\label{cor for g}
     \int \left|\nabla g\right|^{2}-f'\left(\sum_{k=1}^{K}W_{\lambda_{k},x_{k}}\right)g^{2}{\rm{d}}x\geq C_{0}\lVert\nabla g\rVert^{2}_{L^{2}}-C\left(\sum_{k=1}^{K}(a^{+}_{k})^{2}+(a^{-}_{k})^{2}\right),
 \end{equation}
 where $C_{0}>0$ is a fixed constant.\\
 Then by  conservation of energy, we have the following estimates:
 \begin{proposition}[Energy expansion]
 For $\lambda_{k}(t)$, $x_{k}(t)$, $b_{k}(t)\in C^{1}([T_{0},+\infty))$ and $\vec{g}(t)$ obtained in Lemma \ref{Mod}, the following estimates hold for all $t\geq T_{0}$:
     \begin{equation}\label{energy estimate}
         \lVert \dot{g}\rVert^{2}_{L^{2}}+\lVert\nabla g\rVert^{2}_{L^{2}}+\sum_{k=1}^{K}|b_{k}|^{2}\lesssim \lambda_{1}^{3}+\sum_{k=1}^{K}(a^{+}_{k})^{2}+(a^{-}_{k})^{2}.
     \end{equation}
     In particular,
     \begin{equation}\label{estimate for bk}
         \sum_{k=1}^{K}|b_{k}|^{2}\lesssim \lambda^{3}_{1}+\lVert\dot{g}\rVert_{L^{2}}^{2}+\lVert g\rVert_{\dot{H}^{1}}^{2}.
     \end{equation}
     Moreover,
     \begin{equation}\label{refined energy estimate}
         \begin{aligned}
             &(1+O(\lambda_{1}))\left[\lVert \dot{g}\rVert^{2}_{L^{2}}+C_{0}\lVert\nabla g\rVert^{2}_{L^{2}}\right]+\lVert\Lambda W\rVert^{2}_{L^{2}}\sum_{k=1}^{K}|b_{k}|^{2}\\
             \leq &\left(2\times15^{\frac{3}{2}}\int W^{\frac{7}{3}}\right)\sum_{1\leq j<k\leq K}|z_{j}-z_{k}|^{-3}\lambda_{j}^{\frac{3}{2}}\lambda_{k}^{\frac{3}{2}}\\
             &\quad+C\Bigg(\sum_{k=1}^{K}\left((a^{+}_{k})^{2}+(a_{k}^{-})^{2}\right)+\lambda_{1}^{4}+\lambda_{1}^{3}\Bigg(\sum_{j=1}^{K}|x_{j}-z_{j}|+\lVert g\rVert_{\dot{H}^{1}}\Bigg)+\lVert g\rVert^{3}_{\dot{H}^{1}}\Bigg).
         \end{aligned}
     \end{equation}
 \end{proposition}
 
 \begin{proof}
    Recall that
    \begin{align*}
        u(t)&=\sum_{k=1}^{K}W_{\lambda_{k},x_{k}}+g(t),\\
        \partial_{t}u(t)&=\sum_{k=1}^{K}b_{k}\left(\Lambda W\right)_{\underline{\lambda_{k},x_{k}}}+\dot{g}(t).
    \end{align*}
    For simplicity, we write
    \begin{equation*}
        \vec{U}=\left(\sum_{k=1}^{K} W_{\lambda_{k},x_{k}},\sum_{k=1}^{K}b_{k}\left(\Lambda W\right)_{\underline{\lambda_{k},x_{k}}}\right),\quad \vec{g}=(g,\dot{g}).
    \end{equation*}
    We then expand the energy as follows:
    \begin{equation*}
        E\left(u,\partial_{t}u\right)=E\left(\vec{U}\right)+\left\langle DE\left(\vec{U}\right),\vec{g}\right\rangle+\frac{1}{2}\left\langle D^{2}E(\vec{U})\vec{g},\vec{g}\right\rangle+S,
    \end{equation*}
    where
    \begin{align*}
        \left|S\right|=&\Bigg|\int F\left(\sum_{k=1}^{K} W_{\lambda_{k},x_{k}}+g\right)-F\left(\sum_{k=1}^{K}W_{\lambda_{k},x_{k}}\right)\\
        &\qquad-f\left(\sum_{k=1}^{K}W_{\lambda_{k},x_{k}}\right)g-\frac{1}{2}f'\left(\sum_{k=1}^{K}W_{\lambda_{k},x_{k}}\right)g^{2}{\rm{d}}x\Bigg|\\        \lesssim&\int\left(\sum_{k=1}^{K}\left|W_{\lambda_{k},x_{k}}\right|^{\frac{1}{3}}+|g|^{\frac{1}{3}}\right)|g|^{3}{\rm{d}}x\lesssim\lVert g\rVert^{3}_{\dot{H}^{1}}
    \end{align*}
    For the second term on the right-hand side, using \eqref{nonlinear 1}, Lemma \ref{WW}, and the orthogonality condition \eqref{ort3}, we have
  \begin{align*}
      \left|\left\langle DE\left(\vec{U}\right),\vec{g}\right\rangle\right|\leq&\int\left|f\left(\sum_{k=1}^{K}W_{\lambda_{k},x_{k}}\right)-\sum_{k=1}^{K}f\left(W_{\lambda_{k},x_{k}}\right)\right||g|{\rm{d}}x\\
      \lesssim&\sum_{j\neq l}\int\ W_{\lambda_{j},x_{j}}\left(W_{\lambda_{l},x_{l}}\right)^{\frac{4}{3}}|g|{\rm{d}}x\\
      \lesssim&\sum_{j\neq l}\left\lVert W_{\lambda_{j},x_{j}}\left(W_{\lambda_{l},x_{l}}\right)^{\frac{4}{3}}\right\rVert_{L^{\frac{10}{7}}}\lVert g\rVert_{\dot{H}^{1}}\lesssim\lambda^{3}_{1}\lVert g\rVert_{\dot{H}^{1}}.
  \end{align*}
  Substituting the above estimates and \eqref{energy expand} into the energy expansion, we obtain
  \begin{align}\label{energy conservation}
      E\left(u,\partial_t u\right)=&\frac{1}{2}\left\langle D^{2}E(\vec{U})\vec{g},\vec{g}\right\rangle+KE(W,0)+\frac{1}{2}\lVert\Lambda W\rVert^{2}_{L^{2}}\sum_{k=1}^{K}|b_{k}|^{2}\notag\\
      &-\left(15^{\frac{3}{2}}\int W^{\frac{7}{3}}\right)\sum_{1\leq j<k\leq K}|z_{j}-z_{k}|^{-3}\lambda^{\frac{3}{2}}_{k}\lambda_{j}^{\frac{3}{2}}\\    &+O\Bigg(\lambda_{1}^{4}+\lambda_{1}^{3}\left(\sum_{k=1}^{K}|x_{k}-z_{k}|\right)+\left(\sum_{k=1}^{K}|b_{k}|^{2}\right)\lambda_{1}+\lVert g\rVert^{3}_{\dot{H}^{1}}+\lambda_{k}^{3}\lVert g\rVert_{\dot{H}^{1}}\Bigg).\notag
  \end{align}
  From (\ref{asy}) and using
  \begin{equation*}
      \left|\left\langle D^{2}E(\vec{U})\vec{g},\vec{g}\right\rangle\right|\lesssim \lVert \dot{g}\rVert_{L^{2}}^{2}+\lVert \nabla g\rVert^{2}_{L^{2}},
  \end{equation*}
  we have
  \begin{equation*}
      \lim_{t\rightarrow+\infty}E\left(u(t),\partial_{t}u(t)\right)=KE(W,0).
  \end{equation*}
  Hence, by conservation of energy,
  \begin{equation*}
      E\left(u(t),\partial_{t}u(t)\right)=KE(W,0),\quad \forall\ t\geq T_{0}.
  \end{equation*}
  Substituting this into \eqref{energy conservation} and combining it with the coercivity estimate \eqref{cor for g}, we obtain
  \begin{align*}
      KE(W,0)\geq& KE(W,0)+\frac{1}{2}\lVert\dot{g}\rVert^{2}_{L^{2}}+\frac{C_{0}}{2}\lVert g\rVert^{2}_{\dot{H}^{1}}-C\left(\sum_{k=1}^{K}\left(a_{k}^{+}\right)^{2}+\left(a_{k}^{-}\right)^{2}\right)\\
      &+\frac{1}{2}\lVert\Lambda W\rVert^{2}_{L^{2}}\sum_{k=1}^{K}|b_{k}|^{2}-
             \left(15^{\frac{3}{2}}\int W^{\frac{7}{3}}\right)\sum_{1\leq j<k\leq K}|z_{j}-z_{k}|^{-3}\lambda^{\frac{3}{2}}_{k}\lambda_{j}^{\frac{3}{2}}\\
          &+ O\Bigg(\lambda_{1}^{4}+\lambda_{1}^{3}\left(\sum_{k=1}^{K}|x_{k}-z_{k}|\right)+
             \left(\sum_{k=1}^{K}|b_{k}|^{2}\right)\lambda_{1}+\lVert g\rVert^{3}_{\dot{H}^{1}}+\lambda_{1}^{3}\lVert g\rVert_{\dot{H}^{1}}\Bigg) ,
  \end{align*}
  This yields \eqref{energy estimate} and \eqref{estimate for bk}, and then \eqref{refined energy estimate} follows.
 \end{proof}

 \begin{remark}
 Substituting \eqref{cor loc1}, \eqref{energy estimate}, and \eqref{estimate for bk} into \eqref{dervative of pk 1}, the estimate \eqref{dervative of pk 1} can be simplified as follows:
     \begin{equation}\label{sim estimate for pk'}
        p'_{k}(t)\leq -\kappa\sum_{j\neq k}|z_{j}-z_{k}|^{-3}\lambda^{\frac{3}{2}}_{j}\lambda^{\frac{1}{2}}_{k}+O\left(\left(\eta+\frac{1}{\sqrt{M}}\right)\lambda^{2}_{1}+\frac{1}{\lambda_{1}}\sum_{k=1}^{K}\left((a^{+}_{k})^{2}+(a_{k}^{-})^{2}\right)\right).
     \end{equation}
 \end{remark}

 \section{Control of the stable and unstable directions}\label{stable and unstable directions}
 In our case, we also need to control the stable and unstable directions. We first study the equation for $a^{+}_{k}$ and $a^{-}_{k}$.
 \begin{lemma}
     Let $a^{+}_{k}$ and $a^{-}_{k}$ be defined in (\ref{negative directions}) and  let $\lambda_{k}(t)$, $\vec{g}(t)$ be as in Lemma \ref{Mod}. Then we have
     \begin{equation}\label{a+}
         \left|\frac{d}{dt}a^{+}_{k}(t)-\frac{\nu}{\lambda_{k}(t)}a^{+}_{k}(t)\right|\lesssim \lambda^{2}_{k}(t)+\frac{1}{\lambda_{k}(t)}\left(\lVert \dot{g}(t)\rVert^{2}_{L^{2}}+\lVert g(t)\rVert^{2}_{\dot{H}^{1}}\right),
     \end{equation}
     \begin{equation}\label{a-}
                \left|\frac{d}{dt}a^{-}_{k}(t)+\frac{\nu}{\lambda_{k}(t)}a^{-}_{k}(t)\right|\lesssim \lambda^{2}_{k}(t)+\frac{1}{\lambda_{k}(t)}\left(\lVert \dot{g}(t)\rVert^{2}_{L^{2}}+\lVert g(t)\rVert^{2}_{\dot{H}^{1}}\right).
     \end{equation}
 \end{lemma}
 \begin{proof}
     By the definition of $a^{\pm}_{k}$ in (\ref{negative directions}), we compute
     \begin{equation*}
         \frac{d}{dt}a_{k}^{\pm}=\left\langle\partial_{t}\alpha^{\pm}_{\lambda_{k},x_{k}},\vec{g}\right\rangle+\left\langle\alpha^{\pm}_{\lambda_{k},x_{k}},\partial_{t}\vec{g}\right\rangle.
     \end{equation*}
     For the first term,
     \begin{equation*}
\partial_{t}\alpha^{\pm}_{\lambda_{k},x_{k}}=\lambda'_{k}\partial_{\lambda_{k}}\alpha^{\pm}_{\lambda_{k},x_{k}}+x'_{k}\cdot\partial_{x_{k}}\alpha^{\pm}_{\lambda_{k},x_{k}}.
     \end{equation*}
     Since $\mathcal{Y}$ is smooth and exponentially decaying, we obtain from (\ref{lamkbk}), (\ref{estimate for bk}), the estimate
     \begin{align*}
&\left|\left\langle\partial_{t}\alpha^{\pm}_{\lambda_{k},x_{k}},\vec{g}\right\rangle\right|\lesssim \left(\frac{|\lambda'_{k}|}{\lambda_{k}}+\frac{|x'_{k}|}{\lambda_{k}}\right)\lVert\vec{g}\rVert_{\dot{H}^{1}\times L^{2}}\\
&\lesssim \frac{1}{\lambda_{k}}\left(\left|b_{k}\right|+\lVert\dot{g}\rVert_{L^{2}}+\sum_{k=1}^{K}|b_{k}|\lVert g\rVert_{\dot{H}^{1}}\right)\lVert\vec{g}\rVert_{\dot{H}^{1}\times L^{2}}\lesssim \lambda_{k}^{2}+\frac{1}{\lambda_{k}}\lVert\vec{g}\rVert^{2}_{\dot{H}^{1}\times L^{2}}.
     \end{align*}
     For the second term, using (\ref{g}) and (\ref{dotg}),
     \begin{equation*} \left\langle\alpha^{\pm}_{\lambda_{k},x_{k}},\partial_{t}\vec{g}\right\rangle=\left\langle\alpha^{\pm}_{\lambda_{k},x_{k}},J\circ{{\rm{D}}^{2}}E\left(\vec{W}_{\lambda_{k},x_{k}}\right)\vec{g}\right\rangle+\left\langle\alpha^{\pm}_{\lambda_{k},x_{k}},\vec{h}\right\rangle,
     \end{equation*}
     where
     \begin{align*}      \vec{h}=&\left(\sum_{j=1}^{K}\left(b_{j}+\lambda'_{j}\right)\left(\Lambda  W\right)_{\underline{\lambda_{j},x_{j}}}+\sum_{j=1}^{K}x'_{j}\cdot\left(\nabla W\right)_{\underline{\lambda_{j},x_{j}}}, f\left(\sum_{j=1}^{K}W_{\lambda_{j},x_{j}}+g\right)\right.\\
         &\qquad-\sum_{j=1}^{K} f\left(W_{\lambda_{j},x_{j}}\right)-f'\left(W_{\lambda_{k},x_{k}}\right)g-\sum\limits_{j=1}^{K}b'_{j}\left(\Lambda W\right)_{\underline{\lambda_{j},x_{j}}}\\
         &\qquad\qquad\left.+
       \sum\limits_{j=1}^{K}\frac{b_{j}\lambda'_{j}}{\lambda_{j}}\left({\underline{\Lambda}}\Lambda W\right)_{\underline{\lambda_{j},x_{j}}}+\sum_{j=1}^{K}\frac{b_{j}x'_{j}}{\lambda_{j}}\cdot \left(\nabla \Lambda W\right)_{\underline{\lambda_{j},x_{j}}}\right).
     \end{align*}
     By (\ref{negative inner}), the first term on the right-hand side gives the leading-order contribution
     \begin{equation*}
\left\langle\alpha^{\pm}_{\lambda_{k},x_{k}},J\circ{{\rm{D}}^{2}}E\left(\vec{W}_{\lambda_{k},x_{k}}\right)\vec{g}\right\rangle=\pm\frac{\nu}{\lambda_{k}}\left\langle\alpha^{\pm}_{\lambda_{k},x_{k}},\vec{g}\right\rangle=\pm\frac{\nu}{\lambda_{k}}a_{k}^{\pm}.
     \end{equation*}
     Next, for the second term, we expand
     \begin{align*}
\left\langle\alpha^{\pm}_{\lambda_{k},x_{k}},\vec{h}\right\rangle=&\frac{1}{2}\nu\left\langle\lambda^{-2}_{k}\mathcal{Y}_{\lambda_{k},x_{k}},\sum_{j=1}^{K}\left(b_{j}+\lambda_{j}'\right)\left(\Lambda  W\right)_{\underline{\lambda_{j},x_{j}}}+\sum_{j=1}^{K}x'_{j}\cdot\left(\nabla W\right)_{\underline{\lambda_{j},x_{j}}}\right\rangle\\
&\pm\frac{1}{2}\left\langle\mathcal{Y}_{\underline{\lambda_{k},x_{k}}},f\left(\sum_{j=1}^{K}W_{\lambda_{j},x_{j}}+g\right)-\sum_{j=1}^{K} f\left(W_{\lambda_{j},x_{j}}\right)-f'\left(W_{\lambda_{k},x_{k}}\right)g\right.\\
&-\left.\sum\limits_{j=1}^{K}b'_{j}\left(\Lambda W\right)_{\underline{\lambda_{j},x_{j}}}
         +
       \sum\limits_{j=1}^{K}\frac{b_{j}\lambda'_{j}}{\lambda_{j}}\left({\underline{\Lambda}}\Lambda W\right)_{\underline{\lambda_{j},x_{j}}}+\sum_{j=1}^{K}\frac{b_{j}x'_{j}}{\lambda_{j}}\cdot \left(\nabla \Lambda W\right)_{\underline{\lambda_{j},x_{j}}}\right\rangle.
     \end{align*}
    For the first line, using $\left\langle\mathcal{Y},\Lambda W\right\rangle=\left\langle\mathcal{Y},\nabla W\right\rangle=0$ and the estimate \eqref{lamkbk}, we obtain
     \begin{align*}  &\left|\left\langle\lambda^{-2}_{k}\mathcal{Y}_{\lambda_{k},x_{k}},\sum_{j=1}^{K}\left(b_{j}+\lambda_{j}'\right)\left(\Lambda  W\right)_{\underline{\lambda_{j},x_{j}}}+\sum_{j=1}^{K}x'_{j}\cdot\left(\nabla W\right)_{\underline{\lambda_{j},x_{j}}}\right\rangle\right|\\
     \lesssim&\lambda_{k}^{2}\left(\sum_{j\neq k}|b_{j}+\lambda_{j}'|+|x_{j}'|\right)
       \lesssim\lambda_{k}^{2}\lVert \vec{g}\rVert_{\dot{H}^{1}\times L^{2}}.
     \end{align*}
     For the second line, arguing as in the proof of Lemma \ref{mod con}, we have
     \begin{equation*}
\left|\left\langle\mathcal{Y}_{\underline{\lambda_{k},x_{k}}},f\left(\sum_{j=1}^{K}W_{\lambda_{j},x_{j}}+g\right)-\sum_{j=1}^{K} f\left(W_{\lambda_{j},x_{j}}\right)-f'\left(W_{\lambda_{k},x_{k}}\right)g\right\rangle\right|\lesssim\lambda_{k}^{2}+\frac{1}{\lambda_{k}}\lVert g\rVert^{2}_{\dot{H}^{1}}.
     \end{equation*}
     For the third line, from (\ref{lamkbk}), (\ref{bk'}) and (\ref{estimate for bk}),
     \begin{align*}
         &\left|\left\langle\mathcal{ Y}_{\underline{\lambda_{k},x_{k}}},\sum\limits_{j=1}^{K}b'_{j}\left(\Lambda W\right)_{\underline{\lambda_{j},x_{j}}}
         - \sum\limits_{j=1}^{K}\frac{b_{j}\lambda'_{j}}{\lambda_{j}}\left({\underline{\Lambda}}\Lambda W\right)_{\underline{\lambda_{j},x_{j}}}-\sum_{j=1}^{K}\frac{b_{j}x'_{j}}{\lambda_{j}}\cdot \left(\nabla \Lambda W\right)_{\underline{\lambda_{j},x_{j}}}\right\rangle\right|\\
         \lesssim&\sum_{j=1}^{K}|b_{j}'|+\frac{|b_{j}\lambda'_{j}|}{\lambda_{j}}+\frac{|b_{j}x'_{j}|}{\lambda_{j}}\lesssim\lambda_{k}^{2}+\frac{1}{\lambda_{k}}\lVert\vec{g}\rVert^{2}_{\dot{H}^{1}\times L^{2}}.
     \end{align*}
     Finally, combining the above estimates completes the proof.
 \end{proof}
Set 
 \begin{equation*}
     N_{1}(t):=\sum_{k=1}^{K} \left(a^{-}_{k}(t)\right)^{2} \quad {\rm{and}}\quad N_{2}(t):=\sum_{k=1}^{K}\left(a^{+}_{k}(t)\right)^{2}.
 \end{equation*}
Then we have the following lemmas.
 \begin{lemma}\label{bound for N1t1}
     For any $c>0$ and $T_{1}\geq T_{0}$, there exists $t_{1}\geq T_{1}$ such that
     \begin{equation}\label{N1 1}
         N_{1}(t_{1})\leq c\left(\lambda_{1}^{3}(t_{1})+\lVert \vec{g}(t_{1})\rVert^{2}_{\dot{H}^{1}\times L^{2}}\right).
     \end{equation}
 \end{lemma}
 \begin{proof}
     From (\ref{a-}), we have
     \begin{align}\label{der of N1}
         \frac{d}{dt}N_{1}(t)&=2\sum_{k=1}^{K}\left(\frac{d}{dt}a^{-}_{k}\right)a_{k}^{-}\notag\\
        & =2\sum_{k=1}^{K}\left(-\frac{\nu}{\lambda_{k}}a_{k}^{-}+O\left(\lambda_{1}^{2}+\lambda_{1}^{-1}\lVert \vec{g}\rVert^{2}_{\dot{H}^{1}\times L^{2}}\right)\right)a_{k}^{-}\notag\\
        &=-2\nu\sum_{k=1}^{K}\frac{1}{\lambda_{k}}\left(a_{k}^{-}\right)^{2}+2\sum_{k=1}^{K}O\left(|a_{k}^{-}|\left(\lambda_{1}^{2}+\lambda_{1}^{-1}\lVert \vec{g}\rVert^{2}_{\dot{H}^{1}\times L^{2}}\right)\right)\notag\\
        &\leq -2\nu\frac{A}{\lambda_{1}}\sum_{k=1}^{K}\left(a_{k}^{-}\right)^{2}+O\left(\left(\sum_{k=1}^{K}|a_{k}^{-}|\right)\left(\lambda_{1}^{2}+\lambda_{1}^{-1}\lVert \vec{g}\rVert^{2}_{\dot{H}^{1}\times L^{2}}\right)\right).
     \end{align}
     To prove (\ref{N1 1}), we argue by contradiction and suppose that there exist  $C>0$ and $T_{1}\geq T_{0}$ such that for all $t\geq T_{1}$, we have
     \begin{equation}\label{lower bound for N1}
         N_{1}(t)> C\left(\lambda_{1}^{3}(t)+\lVert\vec{g}(t)\rVert^{2}_{\dot{H}^{1}\times L^{2}}\right).
     \end{equation}
     Inserting this into (\ref{der of N1}), we obtain for any $t\geq T_{0}$,
     \begin{equation*}
         \frac{d}{dt}N_{1}(t)\leq -\frac{2\nu A}{\lambda_{1}(t)}N_{1}(t)+D\left(\left(\sum_{k=1}^{K}|a_{k}^{-}(t)|\right)\frac{1}{\lambda_{1}(t)}N_{1}(t)\right),
     \end{equation*}
     where $D>0$ is a fixed constant. Since 
     \begin{equation*}
         \sum_{k=1}^{K}|a_{k}^{-}(t)|\lesssim \lVert\vec{g}(t)\rVert_{\dot{H}^{1}\times L^{2}}\rightarrow 0,\quad {\rm{as}}\quad t\rightarrow+\infty,
     \end{equation*}
     there exists  $t_{0}>T_{1}$ such that for any $t\geq t_{0}$, we have
    \begin{equation*}
        \frac{d}{dt}N_{1}(t)\leq-\frac{\nu A}{\lambda_{1}(t)}N_{1}(t).
    \end{equation*}
     This is equivalent to 
 \begin{equation*}
    \frac{d}{dt}\left(e^{\int_{t_{0}}^{t}\frac{\nu A}{\lambda_{1}(s)}{\rm{d}}s}N_{1}(t)\right)\leq 0.
 \end{equation*}
 Hence,
 \begin{equation*}
     N_{1}(t)\leq N_{1}(t_{0})e^{-\int_{t_{0}}^{t}\frac{\nu A}{\lambda_{1}(s)}{\rm{d}}s}.
 \end{equation*}
 Then, from (\ref{lower bound for N1}), we have
 \begin{equation*}
     \lambda_{1}^{\frac{3}{2}}(t)+\lVert\vec{ g}(t)\rVert_{\dot{H}^{1}\times L^{2}}\lesssim e^{-\int_{t_{0}}^{t}\frac{\nu A}{2\lambda_{1}(s)}{\rm{d}}s}\quad \forall\ t\geq t_{0}.
 \end{equation*}
 On the other hand, by (\ref{lamkbk}) and (\ref{estimate for bk}),
 \begin{align}\label{eq for lam 1}
     \left|\lambda_{1}'(t)\right|\lesssim\left|b_{1}(t)\right|+\lVert\dot{g}(t)\rVert_{L^{2}}+\sum_{k=1}^{K}|b_{k}(t)|\lVert g(t)\rVert_{\dot{H}^{1}}\notag\\
     \lesssim\lambda_{1}^{\frac{3}{2}}(t)+\lVert \vec{g}(t)\rVert_{\dot{H}^{1}\times L^{2}}\lesssim e^{-\int_{t_{0}}^{t}\frac{\nu A}{2\lambda_{1}(s)}{\rm{d}}s}
 \end{align}
 Set 
 \begin{equation*}
    \tilde{\lambda}_{1}=\frac{\nu A}{2\lambda_{1}},
 \end{equation*}
 By (\ref{asy}), 
 \begin{equation*}
     \tilde{\lambda}_{1}(t)\rightarrow+\infty,\quad {\rm{as}}\quad t\rightarrow+\infty.
 \end{equation*}
 Moreover, from \eqref{eq for lam 1}, we have
 \begin{equation*}
     \left|\frac{\tilde{\lambda}_{1}'(t)}{\tilde{\lambda}_{1}(t)}\right|\leq D_{1}\tilde{\lambda}_{1}(t) e^{-\int_{t_{0}}^{t}\tilde{\lambda}(\tau){\rm{d}}\tau}
 \end{equation*}
 for some fixed constant $D_{1}>0$. Hence,
 \begin{equation*}
     \frac{d}{dt}\left(\log \tilde{\lambda}_{1}(t)+D_{1}e^{-\int_{t_{0}}^{t}\tilde{\lambda}_{1}(\tau){\rm{d}}\tau}\right)\leq 0,\quad \forall\ t\geq t_{0},
 \end{equation*}
 from which we obtain
 \begin{equation*}
     \log \tilde{\lambda}_{1}(t)+D_{1}e^{-\int_{t_{0}}^{t}\tilde{\lambda}_{1}(\tau){\rm{d}}\tau}\leq \log \tilde{\lambda}_{1}(t_{0})+D_{1},\quad \forall\ t\geq t_{0}.
 \end{equation*}
 Letting $t\rightarrow+\infty$ contradicts the fact that $\tilde{\lambda}_{1}\rightarrow+\infty$ as $t\rightarrow+\infty$.
 \end{proof}

 \begin{lemma}
     There exists $C_{0}>0$ such that for any $c>0$, there exists $T_{1}\geq T_{0}$ such that for any $t\geq T_{1}$,
     \begin{equation}\label{rough estimate for N2}
         N_{2}(t)\leq c\ \sup_{\tau\geq t}\left(\lambda^{3}_{1}(\tau)+N_{1}(\tau)\right),
     \end{equation}
     \begin{equation}\label{rough estimate for g}
         \lVert \vec{g}(t)\rVert^{2}_{\dot{H}^{1}\times L^{2}}\leq C_{0}\ \sup_{\tau\geq t}\left(\lambda^{3}_{1}(\tau)+N_{1}(\tau)\right).
     \end{equation}
 \end{lemma}
 \begin{proof}
     Let $t\geq T_{0}$ and $t_{1}\geq t$ be such that
     \begin{equation*}
         \lambda_{1}^{3}(t_{1})+\lVert\vec{g}(t_{1})\rVert^{2}_{\dot{H}^{1}\times L^{2}}=\sup_{\tau\geq t}\left(\lambda_{1}^{3}(\tau)+\lVert\vec{g}(\tau)\rVert^{2}_{\dot{H}^{1}\times L^{2}}\right).
     \end{equation*}
     We first prove that, for any $c>0$, if $T_{1}>T_{0}$ is chosen large enough and $t\geq T_{1}$, $t_{1}\geq t$, then
     \begin{equation}\label{bound for N2t1}
         N_{2}(t_{1})\leq c\left(\lambda_{1}^{3}(t_{1})+\lVert\vec{g}(t_{1})\rVert^{2}_{\dot{H}^{1}\times L^{2}}\right).
     \end{equation}
     We argue by contradiction. Suppose (\ref{bound for N2t1}) does not hold; then for some fixed $c>0$, for any $T_{1}>T_{0}$, there exist $t\geq T_{1}$ and $t_{1}\geq t$ such that
     \begin{equation*}
          N_{2}(t_{1})> c\left(\lambda_{1}^{3}(t_{1})+\lVert\vec{g}(t_{1})\rVert^{2}_{\dot{H}^{1}\times L^{2}}\right).
     \end{equation*}
     Let
     \begin{equation*}
         t_{2}:=\max\{\tau\geq t_{1}:N_{2}(\tau)\geq N_{2}(t_{1})\}.
     \end{equation*}
     Then $t_{2}\in[t_{1},+\infty)$, and if $t> t_{2}$,
     \begin{equation*}
         \frac{N_{2}(t_{2})-N_{2}(t)}{t_{2}-t}\leq 0,
     \end{equation*}
     which implies $N_{2}'(t_{2})\leq 0$. On the other hand, by the definition of $t_{1}$ and $t_{2}$, we have
     \begin{equation}\label{estN2}
         N_{2}(t_{2})\geq N_{2}(t_{1})>c\left(\lambda^{3}_{1}(t_{1})+\lVert\vec{g}(t_{1})\rVert^{2}_{\dot{H}^{1}\times L^{2}}\right)\geq c\left(\lambda_{1}^{3}(t_{2})+\lVert \vec{g}(t_{2})\rVert^{2}_{\dot{H}^{1}\times L^{2}}\right).
     \end{equation}
 Then, by (\ref{lamksimlamj}) and (\ref{a+}),
     \begin{align*}
         &\frac{d}{dt}N_{2}(t)\Bigg|_{t=t_{2}}=2\sum_{k=1}^{K}a^{+}_{k}(t_{2})\frac{d}{dt}a_{k}^{+}(t)\Bigg|_{t=t_{2}}\\
         =&2\nu\sum_{k=1}^{K}\frac{1}{\lambda_{k}}(a_{k}^{+})^{2}(t_{2})+O\left(\lVert\vec{g}(t_{2})\rVert_{\dot{H}^{1}\times L^{2}}\left(\lambda_{1}^{2}(t_{2})+\lambda^{-1}_{1}(t_{2})\lVert\vec{g}(t_{2})\rVert_{\dot{H}^{1}\times L^{2}}^{2}\right)\right)\\
         &\geq \frac{2\nu A}{\lambda_{1}(t_{2})}N_{2}(t_{2})+O\left(\lVert\vec{g}(t_{2})\rVert_{\dot{H}^{1}\times L^{2}}\left(\lambda_{1}^{2}(t_{2})+\lambda^{-1}_{1}(t_{2})\lVert\vec{g}(t_{2})\rVert_{\dot{H}^{1}\times L^{2}}^{2}\right).\right.
         \end{align*}
         From \eqref{estN2}, we have
         \begin{align*}
             \frac{d}{dt}N_{2}(t)\Bigg|_{t=t_{2}}\geq& \frac{2\nu Ac}{\lambda_{1}(t_{2})}\left(\lambda_{1}^{3}(t_{2})+\lVert\vec{g}(t_{2})\rVert^{2}_{\dot{H}^{1}\times L^2}\right)\\
             &\qquad+O\left(\frac{\lVert \vec{g}(t_{2})\rVert_{\dot{H}^{1}\times L^{2}}}{\lambda_{1}(t_{2})}\left(\lambda_{1}^{3}(t_{2})+\lVert\vec{g}(t_{2})\rVert^{2}_{\dot{H}^{1}\times L^2}\right)\right)>0,       
     \end{align*}
     provided $T_{1}>T_{0}$ is large enough (since $\lVert \vec{g}(t)\rVert_{\dot{H}^{1}\times L^{2}}\rightarrow 0$ as $t\rightarrow+\infty$ and $t_{2}\geq t_{1}\geq T_{1}$). This contradicts  $N_{2}'(t_{2})\leq 0$. Hence, (\ref{bound for N2t1}) holds.\\
     
     Next, by the energy estimate (\ref{energy estimate}), we have
     \begin{equation*}
         \lVert\vec{g}(t)\rVert^{2}_{\dot{H}^{1}\times L^2}\leq D\left(\lambda^{3}_{1}(t)+N_{1}(t)+N_{2}(t)\right),
     \end{equation*}
     where $D>0$ is a fixed constant. By the previous argument, if we take $c=\displaystyle\frac{1}{2(1+D)}$, then 
     there exists $T_{1}>T_{0}$ such that, for all $t\geq T_{1}$, $t_{1}\geq t$, (\ref{bound for N2t1}) holds.  Then,
     \begin{align*}
         \lambda_{1}^{3}(t_{1})+\lVert\vec{ g}(t_{1})\rVert^{2}_{\dot{H}^{1}\times L^{2}}&\leq (1+D)(\lambda_{1}^{3}(t_{1})+N_{1}(t_{1})+N_{2}(t_{1}))\\
         &\leq \frac{1}{2}\left(\lambda_{1}^{3}(t_{1})+\lVert\vec{g}(t_{1})\rVert^{2}_{\dot{H}^{1}\times L^{2}}\right)+(1+D)\left(\lambda_{1}^{3}(t_{1})+N_{1}(t_{1})\right),
     \end{align*}
     from which we deduce
     \begin{align*}
         \sup_{\tau\geq t}\left(\lambda_{1}^{3}(\tau)+\lVert\vec{ g}(\tau)\rVert^{2}_{\dot{H}^{1}\times L^{2}}\right)=\lambda_{1}^{3}(t_{1})+\lVert\vec{ g}(t_{1})\rVert^{2}_{\dot{H}^{1}\times L^{2}}\leq 2(1+D)\left(\lambda_{1}^{3}(t_{1})+N_{1}(t_{1})\right).
     \end{align*}
     Moreover, for any $t\geq T_{1}$, we have
     \begin{equation*}
         \lambda_{1}^{3}(t)+\lVert\vec{g}(t)\rVert^{2}_{\dot{H}^{1}\times L^{2}}\leq2(1+D)\sup_{\tau\geq t}\left(\lambda_{1}^{3}(\tau)+N_{1}(\tau)\right).
     \end{equation*}
     Hence, taking $C_{0}=2(1+D)$, we obtain
     \begin{equation*}
         \lVert \vec{g}(t)\rVert^{2}_{\dot{H}^{1}\times L^{2}}\leq C_{0}\ \sup_{\tau\geq t}\left(\lambda^{3}_{1}(\tau)+N_{1}(\tau)\right).
     \end{equation*}
     This is (\ref{rough estimate for g}). \\
     
     Finally, we prove (\ref{rough estimate for N2}) by contradiction. 
     Suppose that, for some $c>0$ and any $T_{1}>T_{0}$, there exists $t\geq T_{1}$ such that
     \begin{equation*}
         N_{2}(t)> c\sup_{\tau\geq t}\left(\lambda_{1}^{3}(\tau)+N_{1}(\tau)\right).
     \end{equation*}
     Set
     \begin{equation*}
         t_{3}:=\sup\left\{\tau>t:N_{2}(t)>c\left(\lambda_{1}^{3}(t_{1})+N_{1}(t_{1})\right)\right\}.
     \end{equation*}
     Since $N_{2}(t)\rightarrow 0$ as $t\rightarrow+\infty$, we know that $t_{3}$ is well-defined and $t_{3}\in(t,+\infty)$. 
     Moreover, similar arguments as before, we have $N_{2}'(t_{3})\leq 0$. On the other hand, by the definition of $t_{1}$ and the previous argument,
     \begin{align*}
         \lambda_{1}^{3}(t_{3})+\lVert\vec{g}(t_{3})\rVert^{2}_{\dot{H}^{1}\times L^{2}}\leq \lambda_{1}^{3}(t_{1})+\lVert\vec{g}(t_{1})\rVert^{2}_{\dot{H}^{1}\times L^{2}}\leq C_{0}\left(\lambda_{1}^{3}(t_{1})+N_{1}(t_{1})\right).
     \end{align*}
     Then, by the definition of $t_{3}$,
     \begin{equation*}
         N_{2}(t_{3})=c\left(\lambda_{1}^{3}(t_{1})+N_{1}(t_{1})\right)\geq\frac{c}{C_{0}}\left(\lambda_{1}^{3}(t_{3})+\lVert\vec{g}(t_{3})\rVert^{2}_{\dot{H}^{1}\times L^{2}}\right).
     \end{equation*}
     Hence, computations similar to those above give
     \begin{align*}
         &\frac{d}{dt}N_{2}(t)\Bigg|_{t=t_{3}}\geq \frac{2\nu A}{\lambda_{1}(t_{3})}N_{2}(t_{3})+O\left(\frac{\lVert \vec{g}(t_{3})\rVert_{\dot{H}^{1}\times L^{2}}}{\lambda_{1}(t_{3})}\left(\lambda_{1}^{3}(t_{3})+\lVert\vec{g}(t_{3})\rVert^{2}_{\dot{H}^{1}\times L^2}\right)\right)\\
         \geq& \frac{2\nu Ac}{C_{0}\lambda_{1}(t_{3})}\left(\lambda_{1}^{3}(t_{3})+\lVert\vec{g}(t_{3})\rVert^{2}_{\dot{H}^{1}\times L^{2}}\right)+O\left(\frac{\lVert \vec{g}(t_{3})\rVert_{\dot{H}^{1}\times L^{2}}}{\lambda_{1}(t_{3})}\left(\lambda_{1}^{3}(t_{3})+\lVert\vec{g}(t_{3})\rVert^{2}_{\dot{H}^{1}\times L^2}\right)\right)\\
         >&0
     \end{align*}
     for $t_{3}\geq T_{1}$ large enough. This is a contradiction, and \eqref{rough estimate for N2} follows.
 \end{proof}
 \begin{lemma}
     There exists $C_{0}'>0$ such that for any $c>0$, there exists $T_{1}\geq T_{0}$ such that for any $t\geq T_{1}$,
     \begin{equation}\label{refine estimate for N1}
         N_{1}(t)\leq c\ \sup_{\tau\geq t}\lambda^{3}_{1}(\tau),
     \end{equation}
     \begin{equation}\label{refine estimate for N2}
         N_{2}(t)\leq c\ \sup_{\tau\geq t}\lambda^{3}_{1}(\tau),
     \end{equation}
     \begin{equation}\label{refine estimate for g}
         \lVert\vec{g}(t)\rVert^{2}_{\dot{H}^{1}\times L^{2}}\leq C_{0}' \sup_{\tau\geq t}\lambda^{3}_{1}(\tau).
     \end{equation}
 \end{lemma}
 \begin{proof}
     By Lemma \ref{bound for N1t1}, for any $c>0$ and  any $T_{1}>T_{0}$, there exists $t_{1}\geq T_{1}$ such that 
     \begin{align}\label{bounds for N1T1}
         N_{1}(t_{1})&\leq c\left(\lambda^{3}_{1}(t_{1})+\lVert\vec{g}(t_{1})\rVert^{2}_{\dot{H}^{1}\times L^{2}}\right)\notag\\
         &\leq c\left(\lambda^{3}_{1}(t_{1})+C_{0}\sup_{\tau\geq t_{1}}\left(\lambda^{3}_{1}(\tau)+N_{1}(\tau)\right)\right)\notag\\
         &\leq  c\left(\left(1+C_{0}\right)\sup_{\tau\geq t_{1}}\lambda_{1}^{3}(\tau)+C_{0}\sup_{\tau\geq t_{1}} N_{1}(\tau)\right),
     \end{align}
     where in the second line, we have used (\ref{rough estimate for g}). Set 
     \begin{equation*}
         N_{3}(t):=\lambda_{1}^{3}(t),\quad \tilde{N}_{3}(t):=\sup_{\tau\geq t}N_{3}(\tau),\quad \tilde{N}_{1}(t):=\sup_{\tau\geq t}N_{1}(\tau).
     \end{equation*}
     Then, by the rising sun lemma, $\tilde{N}_{3}(t)$ is differentiable for almost every $t\in[T_{0},+\infty)$. Moreover, by definition,  $\left|\tilde{N}_{3}'(t)\right|\leq  \left|N_{3}'(t)\right|$ at every differentiable points of $\tilde{N}_{3}(t)$ and
     \begin{equation}\label{Lip}
         \left|\tilde{N}_{3}(t)-\tilde{N}_{3}(s)\right|\leq \sup_{\tau\geq T_{0}}\left|N'_{3}(\tau)\right||s-t|\lesssim|s-t|\quad \forall\ s,t\in[T_{0},+\infty).
     \end{equation}
     Fix an arbitrarily small $\eta>0$ and take $\displaystyle c=\frac{\eta}{2(1+C_{0})}$ in (\ref{bounds for N1T1}), we show that, if $t_{1}\geq T_{1}$ is sufficiently large, then for any $t\geq t_{1}$, we have
     \begin{equation}\label{tiN1tiN3}
         \tilde{N}_{1}(t)\leq \eta \tilde{N}_{3}(t).
     \end{equation}
     We argue by contradiction and suppose (\ref{tiN1tiN3}) does not hold. Let $t_{2}>t_{1}$ be such that
     \begin{equation*}
         \tilde{N}_{1}(t_{2})>\eta \tilde{N}_{3}(t_{2}).
     \end{equation*}
     Without loss of generality, we may assume that $\tilde{N}_{1}(t_{2})=N_{1}(t_{2})$ (otherwise, it suffices to replace $t_{2}$ by $\sup\{t\geq t_{2}|N_{1}(t)=\tilde{N}_{1}(t_{2})\}$). We first claim that  if $t$ is sufficiently large and satisfies $N_{1}(t)=\tilde{N}(t)$, then
     \begin{equation}\label{claim for N1}
         N_{1}(t)\geq \eta\tilde{N}_{3}(t)\quad {\rm{implies}}\quad N'_{1}(t)\leq -\frac{\nu A}{\lambda_{1}(t)}N_{1}(t).
     \end{equation}
     In fact, by (\ref{rough estimate for g}),
     \begin{equation*}
         \lVert\vec{g}(t)\rVert^{2}_{\dot{H}^{1}\times L^{2}}\leq C_{0} \left(\tilde{N}_{1}(t)+\tilde{N}_{3}(t)\right)\leq C_{0}\left(1+\eta^{-1}\right)N_{1}(t)
     \end{equation*}
     where in the last inequality, we have used $N_{1}(t)=\tilde{N}(t)$. Then, from (\ref{der of N1}),
     \begin{align*}
         N_{1}'(t)&\leq -\frac{2\nu A}{\lambda_{1}(t)}N_{1}(t)+O\left(\lVert\vec{g}(t)\rVert_{\dot{H}^{1}\times L^2}\left(\lambda_{1}^{2}(t)+\lambda^{-1}_{1}(t)\lVert\vec{g}(t)\rVert^{2}_{\dot{H}^{1}\times L^{2}}\right)\right)\\
         &\leq -\frac{2\nu A}{\lambda_{1}(t)}N_{1}(t)+O\left(\frac{\lVert\vec{g}(t)\rVert_{\dot{H}^{1}\times L^2}}{\lambda_{1}(t)}\left(\eta^{-1}N_{1}(t)+C_{0}\left(1+\eta^{-1}\right)N_{1}(t)\right)\right).
     \end{align*}
     Since $\lVert\vec{g}(t)\rVert^2_{\dot{H}^{1}\times L^{2}}\rightarrow 0$ as $t\rightarrow +\infty$,  if $t$ is sufficiently large,
     \begin{equation*}
         N_{1}'(t)\leq -\frac{\nu A}{\lambda_{1}(t)}N_{1}(t),
     \end{equation*}
     which is (\ref{claim for N1}). Let
     \begin{equation*}
         t_{3}:=\min \left\{t\in[t_{1},t_{2}]:N'_{1}(\tau)\leq-\frac{\nu A}{2\lambda_{1}(\tau)}N_{1}(\tau)\ {\rm{for}}\ {\rm{all}}\ \tau\in[t,t_{2}]\right\}.
     \end{equation*}
     Since $N_{1}(t_{2})>\eta \tilde{N}_{3}(t_{2})$ and $N_{1}(t_{2})=\tilde{N}_{1}(t_{2})$,  for $t_{2}>t_{1}\geq T_{1}$ large enough, it follows from  (\ref{claim for N1}) that
     \begin{equation*}
         N_{1}'(t_{2})\leq -\frac{\nu A}{\lambda_{1}(t_{2})}N_{1}(t_{2}).
     \end{equation*}
     By continuity, we have $t_{3}<t_{2}$. We prove that $t_{3}=t_{1}$.
     Now we argue by contradiction and assume that $t_{3}>t_{1}$. Note  that for almost every $t\geq T_{0}$,
     \begin{align*}
         \left|\tilde{N}_{3}'(t)\right|&\leq \left|N_{3}'(t)\right|=\left|3\lambda'_{1}(t)\lambda_{1}^{2}(t)\right|\lesssim\left(\left|b_{1}\right|+\lVert\vec{g}(t)\rVert_{\dot{H}^{1}\times L^{2}}\right)\lambda_{1}^{2}(t)\\
        & \lesssim\frac{1}{\lambda_{1}(t)}\left(\lVert\vec{g}(t)\rVert_{\dot{H}^{1}\times L^{2}}+\lambda_{1}^{\frac{3}{2}}(t)\right)\lambda_{1}^{3}(t)\ll \frac{1}{\lambda_{1}(t)}\tilde{N}_{3}(t)\quad {\rm{as}}\quad t\rightarrow+\infty.
     \end{align*}
     In particular, we may assume that  
     \begin{equation*}
         \tilde{N}_{3}'(t)\geq-\frac{\nu A}{2\lambda_{1}(t)}\tilde{N}_{3}(t)
     \end{equation*}
     for almost every $t\in[t_{3},t_{2}]$ (provided that $t_{1}\geq T_{1}$ was chosen sufficiently large). 
     We now introduce the auxiliary function
     \begin{equation*}
         \phi(t):= N_{1}(t)-\eta \tilde{N}_{3}(t).
     \end{equation*}
     Then $\phi(t_{2})>0$ and from (\ref{Lip}), $\phi(t)$ is Lipschitz continuous. Moreover, for almost every $t\in[t_{3},t_{2}]$, we have
     \begin{align*}
         \phi'(t)=N_{1}'(t)-\eta \tilde{N}_{3}'(t)\leq -\frac{\nu A}{2\lambda_{1}(t)}N_{1}(t)+\frac{\nu A}{2\lambda_{1}(t)}\eta \tilde{N}_{3}(t)=-\frac{\nu A}{2\lambda_{1}(t)}\phi(t),
     \end{align*}
     which implies  that
     \begin{equation*}
         \frac{d}{dt}\left(e^{\int_{t_{2}}^{t}\frac{\nu A}{2\lambda_{1}(\tau)}{\rm{d}}\tau}\phi(t)\right)\leq 0
     \end{equation*}
     for almost every $t\in[t_{3},t_{2}]$. Since  $ \phi(t)$ is Lipschitz continuous, integrating the above inequality yields
     \begin{equation*}
         e^{\int_{t_{2}}^{t}\frac{\nu A}{2\lambda_{1}(\tau)}{\rm{d}}\tau}\phi(t)\geq \phi (t_{2})\quad \forall\ t\in[t_{3},t_{2}].
     \end{equation*}
     Hence, 
     \begin{equation*}
         \phi (t)>0,\quad \forall\ t\in[t_{3},t_{2}],
     \end{equation*}
     which immediately gives
     \begin{equation*}
         N_{1}(t_{3})>\eta \tilde{N}_{3}(t_{3}).
     \end{equation*}
     By the definition of $t_{3}$, $N_{1}(t)$ is decreasing on $[t_{3},t_{2}]$. Combining this with the fact that $N_{1}(t_{2})=\tilde{N}_{1}(t_{2})$ gives
     \begin{equation*}
         N_{1}(t_{3})=\tilde{N}_{1}(t_{3}).
     \end{equation*}
     Therefore, invoking the claim \eqref{claim for N1} once more, we obtain
     \begin{equation*}
         N'_{1}(t_{3})\leq -\frac{\nu A}{\lambda_{1}(t)}N_{1}(t_{3}),
     \end{equation*}
     which contradicts the definition of $t_{3}$. We thus conclude that $t_{3}=t_{1}$. In particular, this yields
     \begin{equation*}
         N_{1}(t_{1})=\tilde{N}_{1}(t_{1}),\quad N_{1}(t_{1})>\eta \tilde{ N}_{3}(t_{1}).
     \end{equation*}
     However, by the definition of $t_{1}$ and (\ref{bounds for N1T1}), we have
     \begin{equation*}
         N_{1}(t_{1})\leq \frac{\eta}{2(C_{0}+1)}\left[(1+C_{0})\tilde{N}_{3}(t_{1})+C_{0}N_{1}(t_{1})\right].
     \end{equation*}
     It follows that
     \begin{equation*}
         N_{1}(t_{1})\leq\frac{\eta}{2-\frac{C_{0} \eta}{C_{0}+1}}\tilde{N}_{3}(t_{1}),
     \end{equation*}
     which is a contradiction if $\eta>0$ is chosen sufficiently small. Therefore, (\ref{tiN1tiN3}) holds, which implies (\ref{refine estimate for N1}). Finally, (\ref{refine estimate for N2}) and (\ref{refine estimate for g}) follow from (\ref{rough estimate for N2}) and (\ref{rough estimate for g}).
 \end{proof}
 \section{Orders of the modulation parameters and the remainder terms}\label{order of the modulation parameters and the remainder terms}
In this section, we prove that, under the assumptions \eqref{Bub1} and \eqref{asp}, the modulation parameters $\lambda_k(t)$ behave like $t^{-2}$ as $t \to +\infty$, whereas the parameters $b_k(t)$ behave like $t^{-3}$. Moreover, the remainder terms are shown to be of much smaller order than $t^{-3}$ as $t \to +\infty$.
\subsection{Monotonicity of $\zeta_{1}(t)$}
From \eqref{zetalam}, 
 \begin{equation*}
     \left|\frac{\zeta_{k}(t)}{\lambda_{k}(t)}-1\right|\lesssim \sqrt{M}\lVert g\rVert_{\dot{H}^{1}},
 \end{equation*}
 Hence, by \eqref{asy}, for $t$ sufficiently large, $\zeta_{k}(t)\sim \lambda_{k}(t)$, we may replace $\lambda_{k}(t)$ with $\zeta_{k}(t)$ and $\displaystyle\sup_{\tau\geq t}\lambda_{1}^{3}(\tau)$ with $\displaystyle\sup_{\tau\geq t}\zeta_{1}^{3}(\tau)$ in the following arguments. We now prove that $\zeta_{1}(t)$ is in fact a decreasing function for $t$ sufficiently large, so that $\displaystyle\sup_{\tau\geq t}\zeta_{1}^{3}(\tau)=\zeta_{1}^{3}(t)$ for $t$ large enough.
 \begin{proposition}\label{zeta1 decreasing}
     There exists $T_{1}\geq T_{0}$ such that $\zeta_{1}(t)$ is decreasing on $[T_{1},+\infty).$
 \end{proposition}
 \begin{proof}
 It suffices to prove that for $T_{1}$ sufficiently large, for any $t_{1}\geq T_{1}$ and all $t_{1}<t$, 
    \begin{equation*}
        \zeta_{1}(t)<\zeta_{1}(t_{1}).
    \end{equation*}
    We argue by contradiction. Suppose that, for any $T_{1}\geq T_{0}$, there exist   $t_{1}\geq T_{1}$ and $t>t_{1}$ such that $\zeta_{1}(t)\geq \zeta_{1}(t_{1})$. Let
    \begin{equation*}
        t_{2}:=\sup\left\{t:\zeta_{1}(t)=\sup_{\tau\geq t_{1}}\zeta_{1}(\tau)\right\}.
    \end{equation*}
    Since $\zeta_{1}(t)\rightarrow 0$ as $t\rightarrow+\infty$, we know that $t_{2} $ is finite and $t_{2}>t_{1}$ with $\zeta_{1}'(t_{2})=0$.
    Set $\zeta_{0}=\zeta_{1}(t_{2})$ and 
    \begin{equation*}
        t_{3}:=\inf\left\{t\geq t_{2}:\zeta_{1}(t)=\frac{\zeta_{0}}{2}\right\}.
    \end{equation*}
    Using again the fact that $\zeta_{1}(t)\rightarrow 0$ as $t\rightarrow+\infty$ once more, we know $t_{3}$ is finite and $t_{3}>t_{2}$. By the estimate \eqref{refine estimate for g}, for any $t\in[t_{2},t_{3}]$,
    \begin{equation}\label{est for g ex}
        \lVert\vec{g}(t)\rVert^{2}_{\dot{H}^{1}\times L^{2}}\leq C_{0}'\sup_{\tau\geq t_{1}}\zeta_{1}^{3}(\tau)=C_{0}'\zeta_{0}^{3}.
    \end{equation}
    Recall that, by \eqref{sim estimate for pk'}, \eqref{refine estimate for N1} and \eqref{refine estimate for N2}, we have
    \begin{equation*}
        p'_{1}(t)\leq -\kappa\sum_{j\neq 1}|z_{j}-z_{1}|^{-3}\lambda^{\frac{3}{2}}_{j}\lambda^{\frac{1}{2}}_{1}+O\left(\left(\eta+\frac{1}{\sqrt{M}}\right)\lambda^{2}_{1}+\frac{c}{\lambda_{1}}\sup_{\tau\geq t}\lambda_{1}^{3}(\tau)\right).
    \end{equation*}
    On the interval $[t_{2},t_{3}]$, by the definition of $t_{2}$ and $t_{3}$,
    \begin{equation*}
        \frac{\zeta_{0}}{2}\leq \zeta_{1}(t)\leq \zeta_{0},\quad \forall\ t\in[t_{2},t_{3}],
    \end{equation*}
    and since $\zeta_{1}(t)\sim \lambda_{1}(t)$, $\lambda_{1}(t)\sim\lambda_{j}(t)$ $(\forall\ 1\leq j\leq K)$, we have
    \begin{equation*}
        \lambda_{j}(t)\sim \zeta_{0},\quad \forall\ t\in[t_{2},t_{3}]\quad {\rm{and}}\quad 1\leq j\leq K.
    \end{equation*}
    Moreover, since $\displaystyle\eta,\frac{1}{\sqrt{M}},c>0$ can be taken arbitrarily small, there exists a constant $A_{0}>0$, such that
    \begin{equation}\label{ex estimate for pk'}
        p'_{1}(t)\leq -A_{0}\zeta^{2}_{0},\quad \forall\ t\in[t_{2},t_{3}].
    \end{equation}
    By \eqref{zetabk}, \eqref{pkbk}, \eqref{estimate for bk} and \eqref{est for g ex}, we have
    \begin{equation}\label{ex estiamte for zeta'p1}
        \left|\zeta'_{1}(t)+p_{1}(t)\right|\lesssim\frac{1}{\sqrt{M}}\lVert\vec{g}(t)\rVert_{\dot{H}^{1}\times L^{2}}\lesssim\frac{1}{\sqrt{M}}\zeta_{0}^{\frac{3}{2}}\quad \forall\ t\in[t_{2},t_{3}]
    \end{equation}
    In particular, at $t_{2}$, since $\zeta'_{1}(t_{2})=0$,
    \begin{equation*}
        p_{1}(t_{2})\leq \frac{C}{\sqrt{M}}\zeta_{0}^{\frac{3}{2}}
    \end{equation*}
    for some fixed constant $C>0$. Integrating \eqref{ex estimate for pk'} over $[t_{2},t_{3}]$, we obtain for any $t\in[t_{2},t_{3}]$
    \begin{equation}\label{ex estimate for p1}
        p_{1}(t)\leq p_{1}(t_{2})-A_{0}\zeta_{0}^{2}(t-t_{2})\leq \frac{C}{\sqrt{M}}\zeta_{0}^{\frac{3}{2}}-A_{0}\zeta_{0}^{2}(t-t_{2}),
    \end{equation}
    which, combined with \eqref{ex estiamte for zeta'p1}, yields 
    \begin{equation}\label{ex estimate for zeta1'}
        \zeta'_{1}(t)\geq-p_{1}(t)-\frac{C}{\sqrt{M}}\zeta_{0}^{\frac{3}{2}}\geq A_{0}\zeta_{0}^{2}(t-t_{2})-\frac{C}{\sqrt{M}}\zeta_{0}^{\frac{3}{2}}.
    \end{equation}
    Integrating \eqref{ex estimate for zeta1'} on $[t_{2},t_{3}]$ gives
    \begin{align*}
        \zeta_{1}(t_{3})-\zeta_{1}(t_{2})&\geq A\zeta_{0}^{2}\int_{t_{2}}^{t_{3}}(t-t_{2}){\rm{d}}t-\frac{C}{\sqrt{M}}\zeta_{0}^{\frac{3}{2}}(t_{3}-t_{2})\\
        &=\frac{A_{0}\zeta_{0}^{2}}{2}\left(t_{3}-t_{2}\right)^{2}-\frac{C}{\sqrt{M}}\zeta^{\frac{3}{2}}_{0}(t_{3}-t_{2})\geq -\zeta_{0}\frac{C^{2}}{2A_{0}M}.
    \end{align*}
    Thus, for $M$ sufficiently large
    \begin{equation*}
        \zeta_{1}(t_{3})\geq \zeta_{1}(t_{2})-\zeta_{0}\frac{C^{2}}{2A_{0}M}\geq \zeta_{0}-\frac{\zeta_{0}}{4}=\frac{3\zeta_{0}}{4},
    \end{equation*}
    which contradicts the definition of $t_{3}$.
 \end{proof}
  \begin{remark}
  Since $\zeta_{k}(t)\sim\zeta_{1}(t)$, similar arguments show that
      $\zeta_{k}(t)$ is also decreasing for $t\in[T_{1},+\infty)$ if $T_{1}$ is large enough.
 \end{remark}
 \subsection{Proof of the first main conclusion}
In this subsection, we first determine the leading order of the modulation parameters, and then derive estimates for the remainder terms.
 \begin{theorem}\label{order for the parameters}
    Let $\vec{u}(t):[0,+\infty)\rightarrow\dot{H}^{1} (\RR^{5})\times L^{2}(\RR^{5})$ be a solution to (\ref{NLW1}) satisfying (\ref{Bub1}) and (\ref{asp}). Then for $t$ sufficiently large and some fixed constants $C_{1}>0$ and $C_{2}>0$, there exist  modulated parameters $\lambda_{k}(t),b_{k}(t),x_{k}(t)\in C^{1}$  satisfying
\begin{equation}\label{estimates for parameters}
    C_{1}t^{-2}\leq \lambda_{k}(t)\leq C_{2}t^{-2},\ C_{1}t^{-3}\leq b_{k}(t)\leq C_{2}t^{-3},\ \lim_{t\rightarrow+\infty}t^{2}|x_{k}(t)-z_{k}|=0,
\end{equation}
and 
\begin{equation}\label{estimate for remainder terms}
   \lim_{t\rightarrow+\infty}t^{3}\left(\bigg\lVert u(t)-\sum_{k=1}^{K}W_{\lambda_{k},x_{k}}\bigg\rVert_{\dot{H}^{1}(\RR^{5})}+
       \bigg\lVert\partial_{t}u(t)-\sum_{k=1}^{K}b_{k}(\Lambda W)_{\underline{\lambda_{k},x_{k}}}\bigg\rVert_{L^{2}(\RR^{5})}\right)=0.
\end{equation}
\end{theorem}

\begin{remark}
  Substituting the estimates \eqref{estimate for remainder terms} into \eqref{lamkbk} and \eqref{bk'}, we  also obtain the asymptotic ODE system for $\lambda_{k}$ and $b_{k}$,
    \begin{equation*}
		\left\{\begin{aligned}
      & \lambda'_{k}(t)+b_{k}(t)=o(t^{-3}),\\
      & b'_{k}(t)+\kappa \sum_{j\neq k, 1\leq j\leq K}|z_{j}-z_{k}|^{-3}\lambda_{k}^{\frac{1}{2}}(t)\lambda^{\frac{3}{2}}_{j}(t)=o(t^{-4}).
		\end{aligned}\right.\quad {\rm{as}}\quad t\rightarrow+\infty
	\end{equation*}
\end{remark}

\begin{proof}
    \textbf{Step 1: Bounds for $\boldsymbol{\lambda_{k}(t)}$.} For the lower bound, first, from \eqref{zetabk}, \eqref{estimate for bk}, \eqref{refine estimate for g} and Proposition \ref{zeta1 decreasing}, we have
    \begin{equation*}
        \left|\zeta_{1}'(t)\right|\leq \left|b_{1}(t)\right|+\frac{1}{\sqrt{M}}\lVert \dot{g}\rVert_{L^{2}}+\sqrt{M}\left(\sum_{k=1}^{K}\left|b_{k}\right|\right)\lVert g\rVert_{\dot{H}^{1}}\lesssim \left|\zeta_{1}(t)\right|^{\frac{3}{2}}.
    \end{equation*}
    Combining this with the fact that $\zeta_{1}(t)\rightarrow 0$ as $t\rightarrow+\infty$ and $\zeta_{1}(t)>0$, we obtain
    \begin{equation*}
        \zeta_{1}(t)\gtrsim t^{-2} \quad{\rm{as}}\quad t\rightarrow+\infty,
    \end{equation*}
    and since $\zeta_{1}(t)\sim\lambda_{1}(t)$, $\lambda_{1}(t)\sim\lambda_{k}(t)$, we have
    \begin{equation*}
        \lambda_{k}(t)\gtrsim t^{-2} \quad {\rm{as}}\quad t\rightarrow+\infty \quad \forall\ 1\leq k\leq K.
    \end{equation*}
    For the upper bound, from \eqref{sim estimate for pk'}, \eqref{refine estimate for N1} and \eqref{refine estimate for N2},
    \begin{equation*}
        p'_{1}(t)\leq -\kappa\sum_{j\neq 1}|z_{j}-z_{1}|^{-3}\lambda^{\frac{3}{2}}_{j}\lambda^{\frac{1}{2}}_{1}+O\left(\left(\eta+\frac{1}{\sqrt{M}}\right)\lambda^{2}_{1}+\frac{c}{\lambda_{1}}\sup_{\tau\geq t}\lambda_{1}^{3}(\tau)\right).
    \end{equation*}
    Since $\lambda_{1}(t)\sim\lambda_{k}(t)$, $\lambda_{1}(t)\sim\zeta_{1}(t)$ and $\zeta_{1}$ is decreasing, by choosing $\eta$, $\displaystyle\frac{1}{\sqrt{M}}$, $c$ sufficiently small, we obtain there exists some fixed constant $A_{1}>0$ such that 
    \begin{equation}\label{upper bound for p1'}
        p_{1}'(t)\leq -A_{1}\zeta_{1}^{2}(t)\quad \forall\ t\geq T_{1},
    \end{equation}
    where $T_{1}\geq T_{0}$ is sufficiently large. Moreover, from  \eqref{zetabk}, \eqref{pkbk}, \eqref{refine estimate for g} and Proposition \ref{zeta1 decreasing},
    \begin{equation}\label{refine estimate for zeta1'}
        \left|\zeta'_{1}(t)+p_{1}(t)\right|\lesssim\frac{1}{\sqrt{M}}\lVert\vec{g}(t)\rVert_{\dot{H}^{1}\times L^{2}}\lesssim\frac{1}{\sqrt{M}}\zeta_{1}^{\frac{3}{2}}(t),
    \end{equation}
    Taking $M>0$ sufficiently large, we have
     \begin{equation}\label{lower bound for zeta1'}
         \zeta_{1}'(t)\geq-p_{1}(t)-\sqrt{\frac{A_{1}}{3}}\zeta_{1}^{\frac{3}{2}}(t)\quad \forall\ t\geq T_{1}.
     \end{equation}
     We now consider the auxiliary function
     \begin{equation*}
         \varphi_{1}(t):=-p_{1}(t)+\sqrt{\frac{A_{1}}{3}}\zeta_{1}^{\frac{3}{2}}(t).
     \end{equation*}
     It follows from \eqref{upper bound for p1'} and \eqref{lower bound for zeta1'} that
     \begin{align*}
         \varphi_{1}'(t)&=-p_{1}'(t)+\sqrt{\frac{A_{1}}{3}}\cdot\frac{3}{2}\zeta_{1}^{\frac{1}{2}}(t)\zeta'_{1}(t)\\
         &\geq A_{1}\zeta^{2}_{1}(t)+\sqrt{\frac{A_{1}}{3}}\cdot\frac{3}{2}\left(-p_{1}(t)-\sqrt{\frac{A_{1}}{3}}\zeta_{1}^{\frac{3}{2}}(t)\right)\zeta_{1}^{\frac{1}{2}}(t)\\
         &=\frac{1}{2}\sqrt{3A_{1}}\zeta_{1}^{\frac{1}{2}}(t)\left(-p_{1}(t)+\sqrt{\frac{A_{1}}{3}}\zeta_{1}^{\frac{3}{2}}(t)\right)=\frac{1}{2}\sqrt{3A_{1}}\zeta_{1}^{\frac{1}{2}}(t)\varphi_{1}(t),\quad \forall\ t\geq T_{1}.
     \end{align*}
     We claim that $\varphi_{1}(t) \leq 0$ holds for all $t \geq T_{1}$.
Suppose otherwise, then there exists $t_{1}\geq T_{1}$  such that
$\varphi_{1}(t_{1}) > 0$ and $\zeta_{1}(t_{1}) > 0$. It follows that
$\varphi_{1}'(t_{1}) > 0$, and by a continuity argument, we obtain
\begin{equation*}
    \varphi_{1}'(t) > 0, \quad \forall t \in [t_{1}, +\infty).
\end{equation*}
Consequently, for all $t > t_{1}$, we have
\begin{equation*}
    \varphi_{1}(t) > \varphi_{1}(t_{1}) > 0.
\end{equation*}
However, by the definition of $\varphi_{1}(t)$, we have $\varphi_{1}(t) \to 0$ as $t \to +\infty$, which contradicts the fact that $\varphi_{1}(t)>\varphi_{1}(t_{1})>0$ for all $t>t_{1}$. This contradiction establishes our claim that
\begin{equation*}
    \varphi_{1}(t) \leq 0 \quad \text{for all } t \geq T_{1},
\end{equation*}
which immediately yields the bound
\begin{equation*}
    -p_{1}(t) \leq -\sqrt{\frac{A_{1}}{3}}\zeta_{1}^{\frac{3}{2}}(t), \quad \forall\ t \geq T_{1}.
\end{equation*}
Substituting this estimate into inequality \eqref{refine estimate for zeta1'} and choosing $M > 0$ sufficiently large, we arrive at the  inequality
\begin{equation*}
    \zeta_{1}'(t) \leq -\sqrt{\frac{A_{1}}{3}}\zeta_{1}^{\frac{3}{2}}(t) + \frac{1}{2}\sqrt{\frac{A_{1}}{3}}\zeta_{1}^{\frac{3}{2}}(t) = -\frac{1}{2}\sqrt{\frac{A_{1}}{3}}\zeta_{1}^{\frac{3}{2}}(t).
\end{equation*}
Integrating both sides of this inequality over $[t, +\infty)$ and using the positivity condition $\zeta_{1}(t) > 0$ together with the asymptotic behavior $\zeta_{1}(t) \to 0$ as $t \to +\infty$, we deduce that
\begin{equation*}
    \zeta_{1}(t) \lesssim t^{-2} \quad \text{as } t \to +\infty.
\end{equation*}
It  follows that 
\begin{equation*}
    \lambda_{k}(t)\lesssim t^{-2} \quad {\rm{as}}\quad t\rightarrow+\infty \quad \forall\ 1\leq k\leq K.
\end{equation*}
\textbf{Step 2: Bounds for $\boldsymbol{b_{k}(t)}$.}
Combining  \eqref{refine estimate for g} with the previously derived estimate $\lambda_{k}(t) \lesssim t^{-2}$, we immediately obtain 
\begin{equation*}
    \lVert \vec{g}(t) \rVert^{2}_{\dot{H}^{1} \times L^{2}} \lesssim t^{-6}.
\end{equation*}
Substituting this into equation \eqref{estimate for bk} yields the preliminary upper bound
\begin{equation*}
    \left| b_{k}(t) \right| \lesssim t^{-3}.
\end{equation*}
We now turn to the lower bound. Starting from the  inequality \eqref{sim estimate for pk'} and using the asymptotic behavior $\lambda_{j}(t) \sim t^{-2}$, we choose $\eta$, $1/\sqrt{M}$, and $c$ sufficiently small to derive that
\begin{equation*}
    p_{k}'(t) \lesssim -t^{-4}.
\end{equation*}
Integrating both sides over $[t, +\infty)$ and using the vanishing asymptotic condition $p_{k}(t) \to 0$ as $t \to +\infty$, we deduce the lower bound
\begin{equation*}
    p_{k}(t) \gtrsim t^{-3} \quad \text{as } t \to +\infty.
\end{equation*}
Combining this with relation \eqref{pkbk} then gives the corresponding lower bound for $b_{k}(t)$:
\begin{equation*}
    b_{k}(t) \gtrsim t^{-3} \quad \text{as } t \to +\infty.
\end{equation*}
\textbf{Step 3: Estimates for the remainder terms.} We consider the functional
\begin{equation*}
    E_{1}(t):=\sum_{k=1}^{K}\left(p_{k}(t)\right)^{2}-\frac{4}{3}\kappa\sum_{1\leq j<k\leq K}\left|z_{j}-z_{k}\right|^{-3}\zeta_{j}^{\frac{3}{2}}(t)\zeta_{k}^{\frac{3}{2}}(t).
\end{equation*}
We begin by differentiating the functional $E_1(t)$ with respect to $t$ :
\begin{align*}
E'_{1}(t)&=2\sum_{k=1}^{K}p_{k} p'_{k}-\frac{4}{3}\kappa\sum_{1\leq j<k\leq K}\left|z_{j}-z_{k}\right|^{-3}\left(\frac{3}{2}\zeta_{j}'\zeta_{j}^{\frac{1}{2}}\zeta_{k}^{\frac{3}{2}}+\frac{3}{2}\zeta_{k}'\zeta_{k}^{\frac{1}{2}}\zeta_{j}^{\frac{3}{2}}\right).
\end{align*}
From the estimates \eqref{zetalam}, \eqref{zetabk}, \eqref{pkbk} and \eqref{sim estimate for pk'}, we obtain
\begin{align*}
    E'_{1}(t)\leq &2\sum_{k=1}^{K}p_{k}(t)\left(-\kappa\sum_{j\neq k}|z_{j}-z_{k}|^{-3}\lambda^{\frac{3}{2}}_{j}(t)\lambda^{\frac{1}{2}}_{k}(t)+O\left(\left(\eta+\frac{1}{\sqrt{M}}+c\right)t^{-4}\right)\right)\\
    &-\frac{4}{3}\kappa\sum_{1\leq j<k\leq K}\frac{3}{2}\left|z_{j}-z_{k}\right|^{-3}\left[\left(-p_{j}(t)+O\left(\frac{1}{\sqrt{M}}t^{-3}\right)\right)\zeta_{j}^{\frac{1}{2}}(t)\zeta_{k}^{\frac{3}{2}}(t)\right.\\
   &\left.+\left(-p_{k}(t)+O\left(\frac{1}{\sqrt{M}}t^{-3}\right)\right)\zeta_{j}^{\frac{3}{2}}(t)\zeta_{k}^{\frac{1}{2}}(t) \right].   
\end{align*}
Thus, the estimates established in Step 1 and Step 2 give
\begin{align*}
    E'_{1}(t)\leq &-2\kappa\sum_{k=1}^{K}p_{k}(t)\sum_{j\neq k}\left|z_{j}-z_{k}\right|^{-3}\zeta_{j}^{\frac{3}{2}}(t)\zeta_{k}^{\frac{1}{2}}(t)\\
   &+2\kappa\sum_{1\leq j<k\leq K}\left|z_{j}-z_{k}\right|^{-3}\left(p_{j}(t)\zeta_{j}^{\frac{1}{2}}(t)\zeta_{k}^{\frac{3}{2}}(t)+p_{k}(t)\zeta_{j}^{\frac{3}{2}}(t)\zeta_{k}^{\frac{1}{2}}(t) \right)\\
   &+O\left(\left(\eta+\frac{1}{\sqrt{M}}+c\right)t^{-7}\right).
   %=O\left(\left(\eta+\frac{1}{\sqrt{M}}+c\right)t^{-7}\right)
\end{align*}
A direct computation shows that the two leading-order sums cancel. We thus arrive at the remainder-only bound
\begin{equation*}
    E_1'(t) \leq O\left( \left( \eta + \frac{1}{\sqrt{M}} + c \right) t^{-7} \right).
\end{equation*}
We now integrate this  bound from $t$ to $+\infty$. Using the asymptotic vanishing condition $E_1(t) \to 0$ as $t \to +\infty$, we obtain the lower bound
\begin{equation*}
    E_1(t) \gtrsim - \left( \eta + \frac{1}{\sqrt{M}} + c \right) t^{-6}.
\end{equation*}
Recalling the definition of $E_1(t)$ and combining this with \eqref{zetalam} and \eqref{pkbk}, we deduce the lower bound
\begin{equation*}
    \sum_{k=1}^{K} b_k^2(t) \geq \frac{4}{3}\kappa \sum_{1\leq j<k\leq K} |z_j - z_k|^{-3} \lambda_j^{3/2}(t) \lambda_k^{3/2}(t) - C \left( \eta + \frac{1}{\sqrt{M}} + c \right) t^{-6},
\end{equation*}
where $C>0$ is a universal constant independent of $t$, $\eta$, $M$, and $c$. On the other hand, from \eqref{refined energy estimate}, we have
\begin{align*}
   \sum_{k=1}^{K}|b_{k}(t)|^{2}\leq& \frac{\left(2\times15^{\frac{3}{2}}\int W^{\frac{7}{3}}\right)}{ \lVert\Lambda W\rVert^{2}_{L^{2}}}\cdot
             \sum_{1\leq j<k\leq K}|z_{j}-z_{k}|^{-3}\lambda_{j}^{\frac{3}{2}}(t)\lambda_{k}^{\frac{3}{2}}(t)\\
             &+C\Bigg(\sum_{k=1}^{K}\left((a^{+}_{k})^{2}+(a_{k}^{-})^{2}\right)+\lambda_{1}^{4}+\lambda_{1}^{3}\Bigg(\sum_{j=1}^{K}|x_{j}-z_{j}|+
             \lVert g\rVert_{\dot{H}^{1}}\Bigg)+\lVert g\rVert^{3}_{\dot{H}^{1}}\Bigg)\\
             \leq&\frac{4}{3}\kappa\sum_{1\leq j<k\leq K}\left|z_{j}-z_{k}\right|^{-3}\lambda_{j}^{\frac{3}{2}}(t)\lambda_{k}^{\frac{3}{2}}(t)+O\left(c t^{-6}\right).
\end{align*}
Since $\eta>0$, $c>0$, and $1/\sqrt{M}$ can be taken arbitrarily small, we therefore obtain the sharp asymptotic identity
\begin{equation*}
    \sum_{k=1}^{K} b_k^2(t) = \frac{4}{3}\kappa \sum_{1\leq j<k\leq K} |z_j - z_k|^{-3} \lambda_j^{3/2}(t) \lambda_k^{3/2}(t) + o\left( t^{-6} \right), \quad \text{as } t \to +\infty.
\end{equation*}
Substituting this identity back into the refined energy estimate \eqref{refined energy estimate}, we  conclude that
\begin{equation*}
    \lim_{t \to +\infty} t^6 \lVert \vec{g}(t) \rVert_{\dot{H}^1 \times L^2}^2 = 0,
\end{equation*}
which is precisely the estimate \eqref{estimate for remainder terms} we set out to prove. Finally, the bound for $x_{k}$ follows directly from \eqref{lamkbk}, \eqref{asy} and  \eqref{estimate for remainder terms}.
\end{proof}

\section{ODE systems for the parameters $\lambda_{k}$ and $b_{k}$}\label{ode systems for the parameters}
     In the following, we consider the ODE system for $\lambda_{k}(s)$ and $b_{k}(s)$:
      \begin{equation}\label{mud eq}
		\left\{\begin{aligned}
      & \lambda'_{k}(s)+b_{k}(s)=o_{s\rightarrow+\infty}(s^{-3}),\\
      & b'_{k}(s)+\kappa \sum_{j\neq k, 1\leq j\leq K}|z_{j}-z_{k}|^{-3}\lambda_{k}^{\frac{1}{2}}(s)\lambda^{\frac{3}{2}}_{j}(s)=o_{s\rightarrow+\infty}(s^{-4}).
		\end{aligned}\right.
	\end{equation}
    By the previous argument, we have already proved that there exist  constants $C_{1}>0$ and $C_{2}>0$ such that
    \begin{equation}
        C_{1}s^{-2}\leq \lambda_{k}(s)\leq C_{2}s^{-2},\quad C_{1}s^{-3}\leq b_{k}(s)\leq C_{2}s^{-3},
    \end{equation}
    for all $1\leq k\leq K$ and all sufficiently large $s$. Next, we show that the vector $\left(s^{2}\vec{\lambda}(s), s^{3}\vec{b}(s)\right)$ converges to a connected component of a fixed set.\\
    First, we change variables and set
    \begin{equation}
        \alpha_{k}=s^{2}\lambda_{k}(s),\quad  \beta_{k}=s^{3}b_{k}(s),\quad s=e^{t}, 
    \end{equation} 
    Then  (\ref{mud eq}) is equivalent to 
    \begin{equation}
 \left\{\begin{aligned}
      & \alpha'_{k}(t)=2\alpha_{k}(t)-\beta_{k}(t)+o_{t\rightarrow+\infty}(1),\\
      & \beta'_{k}(t)=3\beta_{k}(t)-\kappa\sum_{j\neq k, 1\leq j\leq K}|z_{j}-z_{k}|^{-3}\alpha_{k}^{\frac{1}{2}}(t)\alpha^{\frac{3}{2}}_{j}(t)+o_{t\rightarrow+\infty}(1).
		\end{aligned}\right.
    \end{equation}
    for any $1\leq k\leq K$. And the parameters $\alpha_{k}(t)$ and $\beta_{k}(t)$ satisfy
        \begin{equation}
        C_{1}\leq \alpha_{k}(t)\leq C_{2},\quad C_{1}\leq \beta_{k}(t)\leq C_{2},
    \end{equation}
    We now state the main conclusion of this section.
   \begin{proposition}\label{ODE}
       Suppose  that $(\vec{\alpha}(t),\vec{\beta}(t))\in C^{1}([0,+\infty))$ is a solution of the equation
       \begin{equation}\label{ODE1}
 \left\{\begin{aligned}
      & \alpha'_{k}(t)=2\alpha_{k}(t)-\beta_{k}(t)+\epsilon^{1}_{k}(t),\\
      & \beta'_{k}(t)=3\beta_{k}(t)-\kappa\sum_{j\neq k, 1\leq j\leq K}|z_{j}-z_{k}|^{-3}\alpha_{k}^{\frac{1}{2}}(t)\alpha^{\frac{3}{2}}_{j}(t)+\epsilon_{k}^{2}(t),
		\end{aligned}\right.
    \end{equation}
    where $\epsilon_{k}^{1}(t),\epsilon_{k}^{2}(t)\in C([0,+\infty))$ and 
    \begin{equation}
        \lim_{t\rightarrow+\infty}\epsilon_{k}^{1}(t)=\lim_{t\rightarrow+\infty}\epsilon_{k}^{2}(t)=0.
    \end{equation}
    If there exist  positive constants $C_{1}>0$ and $C_{2}>0$, such that
           \begin{equation}
               C_{1}\leq \alpha_{k}(t)\leq C_{2},\quad C_{1}\leq \beta_{k}(t)\leq C_{2}.
           \end{equation}
           Then, as $t\rightarrow+\infty$, the vector $(\vec{\alpha}(t),\vec{\beta}(t))$ converges to a connected component of the set
           \begin{equation}\label{equ 1}
      Eq(F)=\left\{(\vec{a},\vec{c})=((a_{k})_{1\leq k \leq K},(c_{k})_{1\leq k\leq K})\left|\begin{aligned}
      & 2a_{k}=c_{k},\\
      & 3c_{k}=\kappa\sum_{j\neq k, 1\leq j\leq K}|z_{j}-z_{k}|^{-3}a_{k}^{\frac{1}{2}}a^{\frac{3}{2}}_{j},\\
      & a_{k}>0,\ c_{k}>0,\ \forall\ 1\leq k\leq K.
		\end{aligned}\right.\right\}
 \end{equation}
 \end{proposition}
 \begin{remark}
     The set $Eq(F)$ is non-empty. A proof of this fact was given by Jendrej and Martel \cite{JM} in their construction of multi-bubble solutions for \eqref{NLW1} and we refer the reader to Lemma 3 of their paper for a detailed proof.
 \end{remark}
 \begin{remark}
     Combining Proposition \ref{ODE} with Theorem \ref{order for the parameters}, we  complete the proof of the main theorem \ref{main theorem}.
 \end{remark}
  Before giving the proof of proposition \ref{ODE}, we first introduce some notation that will be used frequently below.
  For  simplicity, we denote
  \begin{align*}
      &X(t):=\left(\vec{\alpha}(t),\vec{\beta}(t)\right)=\left((\alpha_{k}(t))_{1\leq k\leq K},(\beta_{k}(t))_{1\leq k\leq K}\right),\\
      &e(t):=\left((\epsilon^{1}_{k}(t))_{1\leq k\leq K},(\epsilon^{2}_{k}(t))_{1\leq k\leq K}\right),\\
      &F(X):=\left(\left(2\alpha_{k}-\beta_{k}\right)_{1\leq k\leq K}, \left(3\beta_{k}-\kappa\sum_{j\neq k, 1\leq j\leq K}|z_{j}-z_{k}|^{-3}\alpha_{k}^{\frac{1}{2}}\alpha^{\frac{3}{2}}_{j}\right)_{1\leq k\leq K}\right).
  \end{align*}
  Then  (\ref{ODE1}) can be rewritten as 
  \begin{equation}\label{ODES}
      X'(t)=F(X(t))+e(t),\quad {\rm{with}}\quad X(t)\in [C_{1},C_{2}]^{2K},\ e(t) \xrightarrow{t \to +\infty} 0.
  \end{equation}
  Set 
  \begin{equation}
      Eq(F):=\{ F(X)=0,\  {\rm{with}}\ \alpha_{k}>0,\beta_{k}>0,\ \forall 1\leq k\leq K\}
  \end{equation}
  which is equivalent to the set defined in \eqref{equ 1} and is also the set of equilibrium points of the autonomous system:
  \begin{equation}\label{aut ODE}
      X'(t)=F(X(t)).
  \end{equation}
  Let
  \begin{equation*}
      \omega(X):=\{p\in\RR^{2K}: \exists\ t_{n}\rightarrow+\infty, X(t_{n})\rightarrow p\}
  \end{equation*}
  denote the $\omega$-limit set of the solution $X(t)$ for the equation (\ref{ODES}).
  \begin{lemma}\label{omega-limit set}
      Let $\omega(X)$ be defined as above, then $\omega(X)$ is nonempty, compact, connected, and 
      \begin{equation}\label{conve of X(t)}
          \lim_{t\rightarrow+\infty}{\rm{dist}}(X(t), \omega(X))=0,
      \end{equation}
      where ${\rm{dist}}\displaystyle(x,A)=\inf_{y\in A}|x-y|$.
  \end{lemma}
  From (\ref{ODES}), $\omega(X)\subset [C_{1},C_{2}]^{2K}$, 
  then the proof is the same as that for the usual autonomous system (see, e.g, \cite[p.~323 Theorem.]{Walter1998}), so we omit the details.
  
\begin{lemma}\label{Basic properties of Y}
    For any $p\in\omega(X)$, the solution $Y(t)$ of the autonomous system \eqref{aut ODE}  with initial data $Y(0)=p$, which we  denote by $\varphi(t,p)$ below, satisfies
    \begin{itemize}
        \item $Y(t)\in C^{\infty}(\RR)$ and $Y(t)\in [C_{1},C_{2}]^{2K}$ for any $t\in\RR$.
        \item For any $t,s\in\RR$, $p\in\omega(X)$, $\varphi(t,\varphi(s,p))=\varphi(t+s,p)$.
        \item $\omega(X)$ is an invariant set of the flow $\varphi(t,p)$, that is for any $t\in\RR$ and $p\in \omega(X)$, $\varphi(t,p)\in \omega(X)$.
    \end{itemize}
\end{lemma}

\begin{proof}
\textbf{Step 1:}  Since $p\in\omega(X)$,  by the definition of $\omega(X)$, there exists a sequence $t_{n}\rightarrow+\infty$ such that $X(t_{n})\rightarrow p$. 
For any fixed $T>0$ and all sufficiently large $t_n$ such that such that $t_{n}-T\geq 0$,  $X_{n}(s)=X(t_{n}+s)$  solves the ODE
\begin{equation}\label{ODE X_{n}}
    X'_{n}(s)=F(X_{n}(s))+e(t_{n}+s)
\end{equation}
on $[-T,T]$. By  assumption, $X(t)\in [C_{1},C_{2}]^{2K}$, so that $X_{n}(s)\in [C_{1},C_{2}]^{2K}$ for any $s\in[-T,T]$ and $t_{n}$ large enough. Moreover, since $F(X)$ is smooth on $[C_{1},C_{2}]^{2K}$ and $e(s)\in C([0,+\infty))$ with $e(t) \xrightarrow{t \to +\infty} 0$, there exists a fixed constant $M>0$ such that
\begin{equation*}
    \left| F(X)\right|+\left|e(t)\right|\leq M,\quad \forall\ X\in [C_{1},C_{2}]^{2K},\ t\in[0,+\infty).
\end{equation*}
Hence, from (\ref{ODE X_{n}}),
\begin{equation*}
    \left|X'_{n}(s)\right|\leq \left|F(X_{n}(s))\right|+\left| e(t_{n}+s)\right|\leq M
\end{equation*}
is uniformly bounded on $[-T,T]$ for any $n$ large enough. As a result, $\{ X_{n}(s)\}$ is equicontinuous on $[-T,T]$. By Ascoli-Arzel\`a theorem, there exist a subsequence $\{X_{n_{k}}(s)\}$ and some continuous function $\tilde{Y}_{T}(s)\in [C_{1},C_{2}]^{2K}$ such that
\begin{equation}
    X_{n_{k}}(s)\rightarrow \tilde{Y}_{T}(s)\quad {\rm{uniformly}}\ {\rm{on}}\ [-T,T].
\end{equation}
Then, taking $T=1,2,\cdots,m,\cdots$ and using a standard diagonal argument, we can choose a subsequence (which we still denote by $X_{n}(s)$), so that there exists some $\tilde{Y}(s)\in C(\RR)$ and
\begin{equation}
    X_{n}(s)\rightarrow \tilde{Y}(s)\quad {\rm{uniformly}}\ {\rm{on}}\ [-T,T]\ {\rm{for}}\ {\rm{any}}\ T>0,
\end{equation}
and since $X_{n}(s)\in [C_{1},C_{2}]^{2K}$, we have $\tilde{Y}(s)\in[C_{1},C_{2}]^{2K}$ for any $s\in\RR$.\\
\textbf{Step 2:} We claim that $\tilde{Y}(s)\in C^{\infty}(\RR)$ and solves the autonomous system 
\begin{equation*}
    \tilde{Y}'(s)=F(\tilde{Y}(s)).
\end{equation*}
In fact, for any fixed $T>0$, since $\displaystyle\lim_{t\rightarrow+\infty}e(t)=0$
\begin{equation}\label{error}
    \sup_{t\in[-T,T]}|e(t_{n}+t)|\xrightarrow{t_{n}\rightarrow+\infty}0.
\end{equation}
For any fixed $s\in[-T,T]$, integrating (\ref{ODE X_{n}}) gives
\begin{equation}\label{inter}
    X_{n}(s)-X_{n}(0)=\int_{0}^{s}\left(F(X_{n}(\tau))+e(t_{n}+\tau)\right)d\tau.
\end{equation}
Since $X_{n}(s)\rightarrow \tilde{Y}(s)$ uniformly on $[-T,T]$, $F(X)$ is uniformly continuous on $[C_{1},C_{2}]^{2K}$ and (\ref{error}) holds, letting $n\rightarrow +\infty$ on both sides of (\ref{inter}), we obtain
\begin{equation}
    \tilde{Y}(s)-\tilde{Y}(0)=\int_{0}^{s}F(\tilde{Y}(\tau))d\tau\quad \forall\ s\in[-T,T].
\end{equation}
Since $F(X)$ is smooth on $[C_{1},C_{2}]^{2K}$ and $\tilde{Y}(s)\in C(\RR)$, we have
\begin{equation*}
    \tilde{Y}(s)\in C^{1}([-T,T])\quad {\rm{and}}\quad  \tilde{Y}'(s)=F(\tilde{Y}(s))\quad \forall\ s\in[-T,T].
\end{equation*}
Since $T>0$ is arbitrary, 
\begin{equation*}
    \tilde{Y}(s)\in C^{1}(\RR)\quad {\rm{and}}\quad  \tilde{Y}'(s)=F(\tilde{Y}(s))\quad \forall\ s\in\RR.
\end{equation*}
Moreover, $\tilde{Y}(s)\in [C_{1},C_{2}]^{2K}$ for any $s\in\RR$ and by a standard bootstrap argument $\tilde{Y}(s)\in C^{\infty}(\RR)$.\\
\textbf{Step 3:} By step 1 and the uniform convergence of $X_{n}(s)$ on any compact set, we have $\tilde{Y}(0)=p$ and by step 2, $\tilde{Y}(s)$ solves the equation
\begin{equation*}
    \tilde{Y}'(s)=F(\tilde{Y}(s)).
\end{equation*}
Then, by local uniqueness of the ODE (\ref{aut ODE}) (note that $F(X)$ is smooth on $[C_{1},C_{2}]^{2K}$), we obtain
\begin{equation*}
    Y(t)=\tilde{Y}(t),\quad \forall\ t\in\RR.
\end{equation*}
Furthermore, 
\begin{equation*}
     Y(t)\in C^{\infty}(\RR)\quad {\rm{and}}\quad  Y(t)\in [C_{1},C_{2}]^{2K}\quad \forall\ t\in\RR.
\end{equation*}
Then for any $t\in\RR$, $p\in\omega(X)$, $\varphi(t,p)\in [C_{1},C_{2}]^{2K}$ is well-defined. By  local uniqueness of the ODE (\ref{aut ODE}) again, we know 
\begin{equation*}
    \varphi(t,\varphi(s,p))=\varphi(t+s,p),\quad \forall\ t,s\in\RR\ {\rm{and}}\ p\in\omega(X).
\end{equation*}
Finally, we prove that $\omega(X)$ is an invariant set of the flow $\varphi(t,p)$. Indeed, for any $p\in\omega(X)$, by the previous argument, there exists a sequence $t_{n}\rightarrow+\infty$, such that 
\begin{equation*}
    X(t_{n}+t)\xrightarrow{t_{n}\rightarrow+\infty}\varphi(t,p)\quad \forall\ t\in\RR.
\end{equation*}

Hence $\varphi(t,p)\in \omega(X)$ for any $t\in\RR,\ p\in\omega(X)$.
\end{proof}
Next, we  introduce some notation from dynamical systems. 
For simplicity, the following definitions are adapted to our setting; a more general definition can be found, for instance, in \cite{HSZ,MST}. \\
\begin{definition}[$(\varepsilon,T)\ {\rm{chain}}$]
    For  an autonomous flow $\varphi(t,p):\RR\times \omega(X)\rightarrow\omega(X)$, given $\varepsilon>0$ and $T>0$, an $(\varepsilon,T)$ chain from $x$ to $y$ is a finite sequence of points
\[
x=x_0,\; x_1,\; \dots,\; x_n=y\ {\rm{and}}\ x_{i}\in\omega(X)
\]
together with times
\[
t_0,\; t_1,\; \dots,\; t_{n-1}\quad \text{with } t_i\ge T,
\]
such that for each $i=0,1,\dots,n-1$,
\[
|\varphi(t_{i},x_i)-x_{i+1}|<\varepsilon.
\]
\end{definition}
\begin{definition}[$\rm{Internally\ chain\ transitive}\ (ICT)$] Let $S\subset \omega(X)$ be an invariant set (that is $\varphi(t,S)\subset S$ for all $t\in\RR$). We say that $S$ is internally chain transitive if for every $x,y\in S$ and for every $\varepsilon>0$ and $T>0$, there exists an $(\varepsilon,T)$ chain from $x$ to $y$,
\[
x=x_0,\dots,x_n=y,
\]
such that all chain points lie in $S$, i.e.,
\[
x_i\in S\quad (i=0,1,\dots,n),
\]
and the corresponding times satisfy $t_i\ge T$ and
\[
|\varphi(t_{i},x_i)-x_{i+1}|<\varepsilon\quad (i=0,\dots,n-1).
\]
\end{definition}
We now claim that $\omega(X)$ is ICT for the flow $\varphi(t,p):\RR\times \omega(X)\rightarrow\omega(X)$.
\begin{lemma}\label{ICT}
    The $\omega$-limit set $\omega(X)$ is ICT for the flow $\varphi(t,p):\RR\times \omega(X)\rightarrow\omega(X)$.
\end{lemma}
Before giving the proof of Lemma \ref{ICT}, we first prove the following lemma.
\begin{lemma}\label{continuous dependence}
    For any $T>0$, there exists $\delta>0$  such that if ${\rm{dist}}(q,\omega(X))<\delta$, 
 then  the solution of the autonomous system
    \begin{equation*}
        Y'(s)=F(Y(s))\quad {\rm{with}}\quad Y(0)=q
    \end{equation*}
   (which we  denote by $\varphi(t,q)$)  exists on $[0,T]$ and satisfies
   \begin{equation*}
       \varphi(t,q)\in \left[\frac{1}{2}C_{1},2C_{2}\right]^{2K},\quad \forall\ t\in[0,T].
   \end{equation*}
  \end{lemma}
  \begin{proof}
      First, since ${\rm{dist}}(q,\omega(X))<\delta$, there exists  $p\in\omega(X)$ such that $|q-p|<\delta$, and
      \begin{equation*}
          \varphi(t,p)\in\left[C_{1},C_{2}\right]^{2K},\quad \forall\ t \in \RR.
      \end{equation*}
      Take $0<\delta<\delta_{1}$ such that 
      \begin{equation*}
           q\in\left[\frac{3}{4}C_{1},\frac{3}{2}C_{2}\right]^{2K}.
      \end{equation*}
      Set 
      \begin{equation*}
          L=\sup_{X\in\left[\frac{1}{2}C_{1},2C_{2}\right]^{2K}}\left|\nabla F(X)\right|.
      \end{equation*}
      Then, for any $t>0$ such that $\forall\ s\in[0,t]$,
      \begin{equation*}
          \varphi(s,q)\in\left[\frac{1}{2}C_{1},2C_{2}\right]^{2K},
      \end{equation*}
      (By continuity, this holds for some small $t>0$), we have
      \begin{equation*}
          \left|\varphi'(s,p)-\varphi'(s,q)\right|=\left|F(\varphi(s,p))-F(\varphi(s,q))\right|\leq L\left|\varphi(s,p)-\varphi(s,q)\right|,
      \end{equation*}
      which together with Gronwall's inequality gives
      \begin{equation*}
          \left|\varphi(s,p)-\varphi(s,q)\right|\leq e^{Ls}|p-q|\leq e^{Lt}\delta,\quad s\in[0,t].
      \end{equation*}
      Thus, for $\displaystyle 0<\delta<\min \{\delta_{1},\frac{1}{4}C_{1}e^{-LT}\}$, by standard bootstrap argument, we know the solution $\varphi(s,q)$ exists on $[0,T]$ and satisfies
           \begin{equation*}
       \varphi(t,q)\in \left[\frac{1}{2}C_{1},2C_{2}\right]^{2K},\quad \forall\ t\in[0,T].
   \end{equation*}  
  \end{proof}
  \begin{lemma}\label{auto vs unauto}
    For any $T>0$, there exists  $s_{0}>0$ such that for any $s\geq s_{0}$, the solution of the autonomous system
    \begin{equation*}
         Y'(s)=F(Y(s))\quad {\rm{with}}\quad Y(0)=X(s)
    \end{equation*}
    exists on $[0,T]$ and satisfies
    \begin{equation*}
        \varphi(t,X(s))\in \left[\frac{1}{2}C_{1},2C_{2}\right]^{2K},\quad \forall\ t\in[0,T].
    \end{equation*}
  \end{lemma}
  Moreover, there exists a fixed constant $L$, such that for any $t\in[0,T]$,
  \begin{equation}\label{APT}
      \left|\varphi(t,X(s))-X(t+s)\right|\leq \frac{e^{Lt}-1}{L}\sup_{\tau\in[0,t]}\left|e(\tau+s)\right|.
  \end{equation}
\begin{proof}
    For any fixed $T>0$, from (\ref{conve of X(t)}), there exists some $s_{0}>0$ such that for any $s\geq s_{0}$,
    \begin{equation*}
        {\rm{dist}}\left(X(s),\omega(X)\right)<\delta,
    \end{equation*}
    where we take $\delta>0$ as in lemma \ref{continuous dependence} for fixed $T>0$. Hence, it remains to show the estimate (\ref{APT}). For any fixed $s\geq s_{0}$, set
    \begin{equation*}
        \psi_{s}(t):=\varphi(t,X(s))-X(t+s)
    \end{equation*}
    Then  $\psi_{s}(t)$ satisfies the equation
    \begin{equation*}
        \frac{d}{dt}\psi_{s}(t)=F(\varphi(t,X(s))-F(X(t+s))-e(t+s).
    \end{equation*}
    By the previous argument,
    \begin{equation*}
        \varphi(t,X(s)),\ X(t+s)\in \left[\frac{1}{2}C_{1},2C_{2}\right]^{2K},\quad \forall\ t\in[0,T],\ s\geq s_{0}.
    \end{equation*}
    Recalling the definition of $L>0$ in the proof of lemma \ref{continuous dependence}, we have
    \begin{equation*}
        \left|\frac{d}{dt}\psi_{s}(t)\right|\leq L\left|\psi_{s}(t)\right|+|e(t+s)|,
    \end{equation*}
    which together with the Gronwall's inequality and initial condition $\psi_{s}(0)=X(s)-X(s)=0$, gives (\ref{APT}).
\end{proof}
Now, we turn to the proof of Lemma \ref{ICT}:
\begin{proof}
    We only need to prove that for any $a,b\in\omega(X)$ and $\varepsilon>0$, $T>0$, there exists an $(\varepsilon,T)$-chain from $a$  to  $b$ such that all chain points lie in $\omega(X)$. We divide the proof into several steps.\\
    \textbf{Step 1: Basic preparation.} Take $\eta>0$ sufficiently small (which will be chosen below and depends only on $\varepsilon>0,\ T>0$), from (\ref{ODES}), there exists  $t_{0}>0$ such that for any $t\geq t_{0}$,
    \begin{equation*}
        \left|e(t)\right|\leq \eta.
    \end{equation*}
    Since $a,b\in \omega(X)$, for some $t_{a}\geq t_{0}$,
    \begin{equation}\label{start point}
        \left| X(t_{a})-a\right|<\eta,
    \end{equation}
    and some $t_{b}>t_{a}+T$,
    \begin{equation}\label{end point}
        \left| X(t_{b})-b\right|<\eta.
    \end{equation}
    \textbf{Step 2: Choice of the time sequence $\boldsymbol{t_i\geq T}$.} We divide the interval $[t_{a},t_{b}]$ into some smaller intervals with length between $T$ and $2T$. Set $T'=t_{b}-t_{a}\geq T$ and take 
    \begin{equation*}
        n:=\left[\frac{T'}{T}\right]\geq 1,\quad r:=T'-nT\in[0,T).
    \end{equation*}
    We  define
    \begin{equation}
        t_{i}:=T+\frac{r}{n}\in[T,2T)\quad (i=0,\cdots,n-1),
    \end{equation}
    and set
    \begin{equation}
        \tau_{0}:=t_{a},\quad \tau_{i+1}:=\tau_{i}+t_{i}.
    \end{equation}
    Then $\displaystyle \tau_{n}=t_{a}+\sum_{i=0}^{n-1}t_{i}=t_{a}+T'=t_{b}$ and the length of each interval $[\tau_{i},\tau_{i+1}]$ is $\displaystyle t_{i}=T+\frac{r}{n}\in[T,2T)$.\\
    \textbf{Step 3: Choice of the chain points $\boldsymbol{x_{i}\in\omega (X)}$.} Set
    \begin{equation*}
        y_{i}:=X(\tau_{i})\quad (i=0,1,\cdots,n).
    \end{equation*}
    From Lemma \ref{omega-limit set}, as $\tau_{i}\rightarrow+\infty$, ${\rm{dist}}\left(y_{i},\omega(X)\right)\rightarrow 0$. Hence, taking $t_{a}$  sufficiently large and for $\tau_{i}\geq \tau_{0}=t_{a}$,
    \begin{equation}
        {\rm{dist}}(y_{i},\omega(X))< \eta\quad (i=0,1,\cdots,n),
    \end{equation}
     we choose $x_{i}\in\omega(X)$ such that
    \begin{equation}\label{xi-yi}
        |x_{i}-y_{i}|<\eta,\quad (i=0,1,\cdots, n).
    \end{equation}
    Here, we can take $x_{0}=a$, $x_{n}=b$, since from (\ref{start point}) and (\ref{end point}),
    \begin{equation*}
        |y_{0}-a|<\eta,\quad |y_{n}-b|<\eta.
    \end{equation*}
    \textbf{Step 4: Conclusion.} In this step, we prove that 
    the choices of $t_i\geq T$ and $x_i\in\omega(X)$ form an $(\varepsilon,T)$-chain from $a$ to $b$ for $\eta>0$ small enough. From the previous argument, we only need to show the inequality
    \begin{equation}
        |\varphi(t_{i},x_{i})-x_{i+1}|<\varepsilon,\quad (i=0,1,\cdots,n-1)
    \end{equation}
    for $\eta>0$ small enough. First, 
    \begin{equation}\label{3 term}
        |\varphi(t_{i},x_{i})-x_{i+1}|\leq |\varphi(t_{i},x_{i})-\varphi(t_{i},y_{i})|+|\varphi(t_{i},y_{i})-y_{i+1}|+|y_{i+1}-x_{i+1}|.
    \end{equation}
    Then we estimate the three terms on the right-hand side of (\ref{3 term}). For the first term, from the definition of $x_{i}$ and $y_{i}$,
    \begin{equation*}
        x_{i}\in\omega(X)\quad {\rm{and}}\quad {\rm{dist}}(y_{i},\omega(X))\leq |x_{i}-y_{i}|<\eta.
    \end{equation*}
   By Lemma \ref{continuous dependence}, for fixed $T>0$, take $\eta<\delta$ (where $\delta>0$ is defined in Lemma \ref{continuous dependence}), we have
    \begin{equation*}
        \varphi(t,x_{i}),\ \varphi(t,y_{i})\in \left[\frac{1}{2}C_{1},2C_{2}\right]^{2K},\quad \forall\ t\in[0,2T].
    \end{equation*}
    Then
    \begin{equation*}
        |\varphi'(t,x_{i})-\varphi'(t,y_{i})|=\left|F(\varphi(t,x_{i}))-F(\varphi(t,y_{i}))\right|\leq L|\varphi(t,x_{i})-\varphi(t,y_{i})|,
    \end{equation*}
    which together with the initial condition $\varphi(0,x_{i})-\varphi(0,y_{i})=x_{i}-y_{i}$ and Gronwall's inequality gives
    \begin{equation*}
        |\varphi(t_{i},x_{i})-\varphi(t_{i},y_{i})|\leq e^{Lt_{i}}|x_{i}-y_{i}|<e^{2TL}\eta.
    \end{equation*}
    For the second term, note that
    \begin{equation*}
        |\varphi(t_{i},y_{i})-y_{i+1}|=|\varphi(t_{i},X(\tau_{i}))-X(t_{i}+\tau_{i})|
    \end{equation*}
    and for $\tau_{i}\geq t_{a}$ large enough, from Lemma \ref{auto vs unauto},
    \begin{equation*}
        |\varphi(t_{i},X(\tau_{i}))-X(t_{i}+\tau_{i})|\leq\frac{e^{Lt_{i}}-1}{L}\sup_{\tau\in[0,t_{i}]}|e(\tau+\tau_{i})|\leq\frac{e^{2LT}-1}{L}\eta.
    \end{equation*}
    Hence, 
    \begin{equation*}
        |\varphi(t_{i},y_{i})-y_{i+1}|\leq\frac{e^{2LT}-1}{L}\eta.
    \end{equation*}
    For the third term, from (\ref{xi-yi}),
    \begin{equation*}
        |y_{i+1}-x_{i+1}|< \eta.
    \end{equation*}
    Inserting the above three estimates into (\ref{3 term}), we have
    \begin{equation*}
         |\varphi(t_{i},x_{i})-x_{i+1}|<\left(1+e^{2TL}+\frac{e^{2LT}-1}{L}\right)\eta.
    \end{equation*}
    We now choose $\displaystyle\eta<\min \left\{\delta,\frac{\varepsilon}{e^{2LT}+\frac{e^{2LT}-1}{L}+1}\right\}$ and then 
    \begin{equation*}
         |\varphi(t_{i},x_{i})-x_{i+1}|<\varepsilon.
    \end{equation*} 
    As a result, for any $T>0$, $\varepsilon>0$ and $a,b\in\omega(X)$, we have constructed an $(\varepsilon,T)$-chain from $a$ to $b$ with all chain points lying in $\omega(X)$. This shows that $\omega(X)$ is ICT for the flow $\varphi(t,p)$.
\end{proof}
Finally, we return to the proof of Proposition \ref{ODE}. By Lemma \ref{omega-limit set}, we know that $\omega(X)$ is itself a connected set. Thus, to prove Proposition \ref{ODE}, it suffices to establish that $\omega(X) \subset Eq(F)$.
\begin{proposition}\label{omegainEq}
    The $\omega$-limit set $\omega(X)$ is contained in the equilibrium set Eq(F).
\end{proposition}
\begin{proof}
    We divide the proof into three steps.\\
    \textbf{Step 1: Construction of the Lyapunov function.}\\
    Let  $Y(t)=((\alpha_k(t))_{1\le k\le K},(\beta_k(t))_{1\le k\le K})$ be any solution of
\begin{equation}\label{flow}
    Y'(t)=F(Y(t)),\quad Y(0)=p\in\omega(X).
\end{equation}
By Lemma \ref{Basic properties of Y}, we know that  $Y(t)\in C^{\infty}(\RR)$ and $Y(t)\in [C_{1},C_{2}]^{2K}$ for any $t\in\RR$. For simplicity, we write
\begin{equation*}
   A_k:=2\alpha_k-\beta_k,
\qquad
B_k:=3\beta_k-\kappa\sum_{j\ne k}\left|z_{j}-z_{k}\right|^{-3}\alpha_k^{1/2}\alpha_j^{3/2}.
\end{equation*}
Then the autonomous system can be rewritten as
\begin{equation*}
    \alpha_k'=A_k,
\qquad
\beta_k'=B_k,
\end{equation*}
and we define the following Lyapunov function
\begin{equation}\label{lyapounv}
    L(\vec{\alpha},\vec{\beta})=\frac{1}{2}\sum_{k=1}^{K}(2\alpha_{k}-\beta_{k})^{2}+3\sum_{k=1}^{K}\alpha_{k}^{2}-\frac{2}{3}\kappa\sum_{1\leq i<j\leq K}|z_{i}-z_{j}|^{-3}\alpha^{\frac{3}{2}}_{i}\alpha^{\frac{3}{2}}_{j}.
\end{equation}
Since $[C_1,C_2]^{2K}\subset (0,+\infty)^{2K}$, the function $L$ is $C^\infty$ on an open neighborhood of $[C_1,C_2]^{2K}$.
We now differentiate \eqref{lyapounv} along the flow \eqref{flow}:
\begin{align*}
    \frac{d}{dt}L(Y(t))=&\sum_{k=1}^{K}(2\alpha_{k}-\beta_{k})(2\alpha_{k}'-\beta_{k}')+6\sum_{k=1}^{K}\alpha_{k}\alpha_{k}'\\
    &-\kappa\sum_{1\leq i<j\leq K}|z_{i}-z_{j}|^{-3}\left(\alpha_{i}^{\frac{1}{2}}\alpha_{j}^{\frac{3}{2}}\alpha_{i}'+\alpha_{j}^{\frac{1}{2}}\alpha_{i}^{\frac{3}{2}}\alpha_{j}'\right)\\
    =&\sum_{k=1}^{K}A_{k}(2A_{k}-B_{k})+6\sum_{k=1}^{K}\alpha_{k}A_{k}
    -\kappa\sum_{k=1}^{K}A_{k}\sum_{j\neq k}\left|z_{k}-z_{j}\right|^{-3}\alpha_{k}^{\frac{1}{2}}\alpha_{j}^{\frac{3}{2}}.
\end{align*}
Since 
\begin{equation*}
    B_k=3\beta_k-\kappa\sum_{j\ne k}\left|z_{j}-z_{k}\right|^{-3}\alpha_k^{1/2}\alpha_j^{3/2},
\end{equation*}
we get
\begin{align*}
    \frac{d}{dt}L(Y(t))&=\sum_{k=1}^{K}A_{k}(2A_{k}-B_{k})+6\sum_{k=1}^{K}\alpha_{k}A_{k}
    -\sum_{k=1}^{K}A_{k}(3\beta_{k}-B_{k})\\
    &=\sum_{k=1}^{K}2A_{k}^{2}+3\sum_{k=1}^{K}(2\alpha_{k}-\beta_{k})A_{k}\\
    &=\sum_{k=1}^{K}5A_{k}^{2}.
\end{align*}
Thus, for every solution of the autonomous system \eqref{flow},
\begin{equation}\label{nodecreasing of L}
    \frac{d}{dt}L(Y(t))=5\sum_{k=1}^{K}\left(2\alpha_{k}(t)-\beta_{k}(t)\right)^{2}=5\sum_{k=1}^{K}\left(\alpha_{k}'(t)\right)^{2}\geq 0.
\end{equation}
\textbf{Step 2: Properties of the Lyapunov function.}\\
We prove the following two properties of the Lyapunov function $L$ 
which will be used later.\\
(1). For any $p\in\omega(X)$ and any $T>0$,
    \begin{equation}\label{equal}
        L(\varphi(T,p))=L(p)
\quad\Longleftrightarrow\quad
p\in Eq(F),
\end{equation}
where $\varphi(t,p)$ denotes the autonomous flow. In particular, if $p\in\omega(X)$ and $p\notin Eq(F)$, then for every $T>0$,
\begin{equation}\label{inequal}
    L(\varphi(T,p))>L(p).
\end{equation}
(2). $L\big(Eq(F)\cap [C_1,C_2]^{2K}\big)$ has empty interior in $\RR$.\\
For part (1), if $p\in Eq(F)$,  by  uniqueness of the solution, $\varphi(t,p)\equiv p$, and hence $L(\varphi(T,p))=L(p)$. Conversely, assume that for some $T>0$,
\[
L(\varphi(T,p))=L(p).
\]
Then, from \eqref{nodecreasing of L}, we know that $L(\varphi(t,p))$ is increasing on $[0,T]$ which combines with the fact $L(\varphi(T,p))=L(p)$ gives
\begin{equation*}
    0=\frac{d}{dt}L(Y(t))=5\sum_{k=1}^{K}\left(2\alpha_{k}(t)-\beta_{k}(t)\right)^{2}=5\sum_{k=1}^{K}\left(\alpha_{k}'(t)\right)^{2}\quad \text{for all } t\in[0,T].
\end{equation*}
Hence
\[
\alpha_k'(t)=2\alpha_k(t)-\beta_k(t)\equiv 0
\qquad \text{on }[0,T],
\]
so each $\alpha_k$ is constant on $[0,T]$, and hence each $\beta_k=2\alpha_k$ is also constant on $[0,T]$. Therefore
\[
\beta_k'(t)\equiv 0
\qquad \text{on }[0,T].
\]
Substituting into the second equation of the autonomous system yields
\[
0
=
3\beta_k(0)-\kappa\sum_{j\ne k}\left|z_{j}-z_{k}\right|^{-3}\alpha_k(0)^{1/2}\alpha_j(0)^{3/2},
\qquad 1\le k\le K,
\]
which together with $\beta_k(0)=2\alpha_k(0)$ gives that $F(p)=0$, i.e. $p\in Eq(F)$. This proves \eqref{equal}, and \eqref{inequal}  follows.\\
For part (2), we define
\begin{equation}\label{Psi}
\Psi(\vec\alpha)
:=
3\sum_{k=1}^K\alpha_k^2
-\frac23\kappa\sum_{1\le i<j\le K}\left|z_{j}-z_{j}\right|^{-3}\alpha_i^{3/2}\alpha_j^{3/2},
\quad
\vec\alpha\in \Big(\frac{C_1}{2},2C_2\Big)^K.
\end{equation}
Then $\Psi\in C^\infty\!\big((\frac{C_1}{2},2C_2)^K\big)$ and
\[
L(\vec\alpha,\vec\beta)
=
\frac12\sum_{k=1}^K(2\alpha_k-\beta_k)^2+\Psi(\vec\alpha).
\]
Moreover,
\[
\partial_{\alpha_k}\Psi(\vec\alpha)
=
6\alpha_k-\kappa\sum_{j\ne k}|z_{j}-z_{k}|^{-3}\alpha_k^{1/2}\alpha_j^{3/2}.
\]
Hence a point $((a_k)_{1\le k\le K},(c_k)_{1\le k\le K})\in [C_1,C_2]^{2K}$ belongs to $Eq(F)$ if and only if
\[
c_k=2a_k,
\qquad
\partial_{\alpha_k}\Psi(\vec a)=0,
\qquad
1\le k\le K.
\]
Therefore
\[
L\big(Eq(F)\cap [C_1,C_2]^{2K}\big)
\subset
\Psi\big(\{\vec a\in [C_1,C_2]^K:\ \nabla\Psi(\vec a)=0\}\big).
\]
The set on the left-hand side is contained in the image of the critical set of a smooth function $\Psi$. Furthermore, by Sard's theorem, the image of the critical set of $\Psi$ has empty interior in $\RR$, hence the conclusion of the second part follows.\\
\textbf{Step 3: $\boldsymbol{\omega(X)\subset Eq(F)}$.}\\
We now prove that $\omega(X)\subset Eq(F)$. We argue by contradiction and assume that
\[
\omega(X)\not\subset Eq(F).
\]
Then there exists
\[
p\in \omega(X)\setminus Eq(F).
\]
Since $\omega(X)$ is invariant under the autonomous flow, $\varphi(1,p)\in \omega(X)$. Set
\[
q:=\varphi(1,p)\in \omega(X).
\]
Because $p\notin Eq(F)$, \eqref{inequal} gives
\[
L(q)>L(p).
\]
From Step 2 Part (2), we can choose 
\begin{equation}\label{cchoice}
c\in (L(p),L(q))
\setminus
L\big(Eq(F)\cap [C_1,C_2]^{2K}\big).
\end{equation}
Since $\omega(X)$ is connected and $L$ is continuous, the image $L(\omega(X))$ is an interval in $\RR$. Since $L(p)<c<L(q)$ and $p,q\in \omega(X)$, there exists at least one point in $\omega(X)$ with Lyapunov value $c$. Thus
\[
K_c:=\{x\in \omega(X):\ L(x)=c\}
\]
is a nonempty compact subset of $\omega(X)$. By \eqref{cchoice},
\begin{equation}\label{Kc-noeq}
K_c\cap Eq(F)=\varnothing.
\end{equation}
Fix any $x\in K_c$. Since, by \eqref{Kc-noeq}, $x\notin Eq(F)$, we have
\[
L(\varphi(1,x))>L(x)=c.
\]
Define \[
\Delta_x:=L(\varphi(1,x))-L(x)>0.
\]
Since the map
\[
y\mapsto L(\varphi(1,y))-L(y)
\]
is continuous on $\omega(X)$, there exists an open neighborhood $U_x\subset \omega(X)$ of $x$ such that for every $y\in U_x$,
\[
L(\varphi(1,y))-L(y)>\frac{\Delta_x}{2}.
\]
Set
\[
\eta_x:=\frac{\Delta_x}{4}>0.
\]
Since $L(\varphi(t,y))$ is non-decreasing in $t$ by \eqref{nodecreasing of L}, it follows that for all $t\geq 1$ and all $y\in U_x$,
\begin{equation}\label{Uxall}
L(\varphi(t,y))\geq L(y)+2\eta_x.
\end{equation}
Now $\{U_x\}_{x\in K_c}$ is an open cover of the compact set $K_c$, so there exists a finite subcover
\[
K_c\subset \bigcup_{m=1}^N U_{x_m}.
\]
Define
\[
U:=\bigcup_{m=1}^N U_{x_m},
\qquad
\eta:=\min_{1\leq m\leq N}\eta_{x_m}>0.
\]
Then \eqref{Uxall} implies
\begin{equation}\label{uniformU}
L(\varphi(t,y))\geq L(y)+2\eta
\qquad
\forall\, y\in U,\ \forall\, t\geq 1.
\end{equation}
Since $K_c\subset U$ and $U$ is open in $\omega(X)$, there exists $\delta_0>0$ such that
\begin{equation}\label{band}
\{x\in \omega(X):\ |L(x)-c|\leq 2\delta_0\}\subset U.
\end{equation}
In fact, if not, there would exist a sequence $x_n\in \omega(X)\setminus U$ such that
\[
|L(x_n)-c|\leq \frac1n.
\]
Since $\omega(X)$ is compact, after passing to a subsequence $x_n\to x_\ast\in \omega(X)\setminus U$. By continuity of $L$, $L(x_\ast)=c$, hence $x_\ast\in K_c\subset U$, a contradiction. Now choose
\begin{equation}\label{deltachoice}
0<\delta<\min\Big\{
\delta_0,\,
\eta,\,
\frac{c-L(p)}{2},\,
\frac{L(q)-c}{2}
\Big\}.
\end{equation}
Then by the definition of $\delta$,
\[
L(p)<c-2\delta,
\qquad
L(q)>c+2\delta.
\]
Define
\begin{equation}\label{eq:Asets}
A_0:=\{x\in \omega(X) :\ L(x)>c-\delta\},
\qquad
A_1:=\{x\in \omega(X):\ L(x)\geq c+\delta\}.
\end{equation}
Clearly,
\begin{equation}\label{eq:Ainc}
A_1\subset A_0,
\qquad
q\in A_1,
\qquad
p\notin A_0.
\end{equation}

We now claim that
\begin{equation}\label{A0A1}
\varphi(t,y)\in A_1
\qquad
\forall\, y\in A_0,\ \forall\, t\geq 1.
\end{equation}
Let $y\in A_0$ and $t\geq 1$.\\
\smallskip
\noindent\emph{Case 1.} If $L(y)\geq c+\delta$, then $y\in A_1$, and by \eqref{nodecreasing of L},
\[
L(\varphi(t,y))\geq L(y)\geq c+\delta,
\]
hence $\varphi(t,y)\in A_1$.\\
\smallskip
\noindent\emph{Case 2.} If $c-\delta<L(y)<c+\delta$, then
\[
|L(y)-c|<\delta<2\delta_0,
\]
so by \eqref{band}, $y\in U$. Therefore \eqref{uniformU} gives
\[
L(\varphi(t,y))\geq L(y)+2\eta.
\]
Since $\delta<\eta$, we get
\[
L(\varphi(t,y))
>
(c-\delta)+2\eta
>
c+\delta.
\]
Thus $\varphi(t,y)\in A_1$.
This proves \eqref{A0A1}.\\
Since $\omega(X)$ is compact, the sets
\[
A_{1}
=
\{x\in \omega(X):\ L(x)\geq c+\delta\},
\]
and
\[
\omega(X)\setminus A_0
=
\{x\in \omega(X):\ L(x)\leq c-\delta\},
\]
are disjoint compact subsets of $\omega(X)$. Hence
\begin{equation}\label{eq:distpos}
d_\ast
:=
\text{dist}\big(A_{1},\,\omega(X)\setminus A_0\big)
>0.
\end{equation}
Choose
\[
0<\varepsilon<d_\ast.
\]
Because $\omega(X)$ is internally chain transitive, for this $\varepsilon$ and for $T=1$, there exists an $(\varepsilon,1)$-chain from $q$ to $p$. Thus there exist chain points
\[
q=x_0,\ x_1,\dots,x_n=p,
\qquad x_i\in \omega(X),
\]
and times
\[
t_0,\dots,t_{n-1}\geq 1,
\]
such that
\begin{equation}\label{eq:chain}
|\varphi(t_i,x_i)-x_{i+1}|<\varepsilon,
\qquad i=0,\dots,n-1.
\end{equation}

We claim that
\begin{equation}\label{eq:allA0}
x_i\in A_0
\qquad \forall\, i=0,1,\dots,n.
\end{equation}
Indeed, by \eqref{eq:Ainc}, $x_0=q\in A_1\subset A_0$.

Assume $x_i\in A_0$ for some $0\leq i\leq n-1$. Since $t_i\geq 1$, by \eqref{A0A1},
\[
\varphi(t_i,x_i)\in A_1.
\]
Then \eqref{eq:chain} implies
\[
\text{dist}(x_{i+1},A_{1})<\varepsilon<d_\ast.
\]
By the definition of $d_\ast$ in \eqref{eq:distpos}, this is impossible if $x_{i+1}\in \omega(X)\setminus A_0$. Therefore $x_{i+1}\in A_0$. This proves \eqref{eq:allA0}.

In particular,
\[
p=x_n\in A_0,
\]
which contradicts \eqref{eq:Ainc}, where $p\notin A_0$.

This contradiction shows that the assumption $\omega(X)\not\subset Eq(F)$ is false. Hence
\[
\omega(X)\subset Eq(F).
\]
The proof is complete.
\end{proof}

 \appendix
 \section{Rigorous analysis of the algebraic equations}\label{}
 In this section, we study the algebraic equation (\ref{equ 1})
 \begin{equation*}
      \left\{\begin{aligned}
      & 2a_{k}=b_{k},\\
      & 3b_{k}=\kappa\sum_{j\neq k, 1\leq j\leq K}|z_{j}-z_{k}|^{-3}a_{k}^{\frac{1}{2}}a^{\frac{3}{2}}_{j},\\
      & a_{k}>0,\ b_{k}>0,\ \forall\ 1\leq k\leq K.
		\end{aligned}\right.
 \end{equation*}
 First, we simplify the aforementioned equation by substituting the first line into the second line and setting $x_{k}=\sqrt{a_{k}}$. The original equation is then equivalent to the following algebraic system:
\begin{equation}\label{alg}
     \left\{\begin{aligned}
      & 6x_{k}=\kappa\sum_{j\neq k, 1\leq j\leq K}|z_{j}-z_{k}|^{-3}x_{j}^{3},\\
      &x_{k}>0,\quad \forall\ 1\leq k\leq K.
		\end{aligned}\right.
\end{equation}
Next, we  prove that when $K=3$, the solutions to the above algebraic equation are isolated.
\begin{proposition}
When $K=3$, for any three distinct points $z_1, z_2, z_3$ in $\RR^5$, the solutions of equation \eqref{alg} are always isolated.
\end{proposition}
\begin{proof}
 For simplicity, we denote
 \begin{equation*}
		d_{jk}=d_{kj}=\kappa|z_{j}-z_{k}|^{-3}>0,\quad \forall\ j\neq k, 1\leq j,k\leq 3.
	\end{equation*}
    Then the equation can be rewritten as
    \begin{equation*}
	F_{k}(x)=6x_{k}-\sum_{j\neq k}d_{jk}x_{j}^{3}=0,\quad k=1,2,3,\quad \text{and}\ x_{1},x_{2},x_{3}>0.
\end{equation*}
To prove that all solutions to the above equation are isolated, by the Implicit Function Theorem, it suffices to show that for any solution $x=(x_1,x_2,x_3)$ satisfying $x_1,x_2,x_3>0$, the Jacobian $J(x)=\nabla F(x)$ is invertible.\\
First, a direct computation yields, for any $k\neq l$,
 \begin{equation*}
 	\frac{\partial F_{k}}{\partial x_{l}}(x)=-3d_{kl}x^{2}_{l},\quad \frac{\partial F_{k}}{\partial x_{k}}(x)=6.
 \end{equation*}
Hence,
\begin{equation}
J(x)=\begin{pmatrix}
	6& -3d_{12}x_{2}^{2}& -3d_{13}x_{3}^{2}\\
	-3d_{12}x_{1}^{2}& 6& -3d_{23}x_{3}^{2}\\
	-3d_{13}x_{1}^{2}& -3d_{23}x_{2}^{2}& 6
	\end{pmatrix}.
\end{equation}
Let $S=\text{diag}(x_{1},x_{2},x_{3})$. Then, by assumption, $S$ is invertible, and $J(x)$ is similar to
\begin{equation*}
	SJ(x)S^{-1}=6I_{3}-A,
\end{equation*}
where $A$ is a symmetric matrix, and
\begin{equation*}
	A=\begin{pmatrix}
		0& a_{12}& a_{13}\\
		a_{12}& 0& a_{23}\\
		a_{13}& a_{23}& 0
	\end{pmatrix},\quad a_{ij}=3d_{ij}x_{i}x_{j}>0.
\end{equation*}
Then, 
\begin{equation*}
	J(x)\ \text{is invertible}\Leftrightarrow \text{det} (6I_{3}-A)\neq 0.
\end{equation*}
Since $A$ is a symmetric matrix, all its eigenvalues are real. We may assume without loss of generality that $\lambda_{1}\leq \lambda_{2}\leq \lambda_{3}$. On the other hand,
\begin{equation*}
	\text{det}(A)=2a_{12}a_{13}a_{23}>0,\quad \text{tr}(A)=0.
	\end{equation*}
    It follows that
    \[
\lambda_{1}\leq \lambda_{2}<0,\quad \lambda_{3}>0.
\]
Furthermore, since $x=(x_{1},x_{2},x_{3})$ satisfies the equation
\[
6x_{k}=\sum_{j\neq k}d_{jk}x_{j}^{3}.
\]
Multiplying both sides of the equation by $3x_{k}$ yields
\[
18x_{k}^{2}=3x_{k}\sum_{j\neq k}d_{kj}x_{j}^{3}=\sum_{j\neq k}(3d_{kj}x_{k}x_{j})x_{j}^{2}=\sum_{j\neq k}a_{kj}x_{j}^{2}.
\]
Then setting $u=(x_{1}^{2},x_{2}^{2},x_{3}^{2})^{\text{T}}\neq 0$, we have
\[
Au=18u.
\]
Thus, 18 is a positive eigenvalue of $A$, which implies that $\lambda_{3}=18$. Therefore, all eigenvalues of $6I_{3}-A $ are
\[
6-\lambda_{1}>0,\quad 6-\lambda_{2}>0,\quad 6-18=-12<0.
\]
It follows that $\det(6I_{3}-A)\neq 0$. This completes the proof.
\end{proof}
Subsequently, for $K=10$, we present a concrete example showing that the positive solutions of equation \eqref{alg} are not necessarily isolated.
\begin{proposition}[Non-isolated positive solutions]
    Let $K=10$, set \[
    \theta:=\frac{\pi}{5}.
    \]
    For $B>0$, we define ten points in $\RR^{5}$ by
    \begin{equation}\label{definition of zk}
        z_k(B):=\Bigl(\cos((k-1)\theta),\,\sin((k-1)\theta),\,\sqrt B\,\cos(2(k-1)\theta),\,\sqrt B\,\sin(2(k-1)\theta),\,0\Bigr),
    \end{equation}
    where $k=1,\cdots,10.$ Then there exists $B_{0}\in(4.7,4.71)$ such that if we take
    \[z_{k}=z_{k}(B_{0}),\quad k=1,\cdots,10, \]
    then there exist constants $a>b>0$ such that for every $t\in\RR$,
    \begin{equation}
        x_{k}(t):=a+b\cos (t+2(k-1)\theta),\quad k=1,\cdots 10,
    \end{equation}
    satisfies equation \eqref{alg}.
\end{proposition}
\begin{proof}
    We divide the proof into several steps.\\
    \textbf{Step 1: Computation of the coefficient matrix.} We define the cyclic distance
    \[\rho(j,k):=\min \{|j-k|,10-|j-k|\}\in\{1,2,3,4,5\}.\]
    By the definition of $z_{k}(B)$ in \eqref{definition of zk} and by direct computation, we have
    \begin{align*}
    \left|z_{j}(B)-z_{k}(B)\right|^{2}&=4\sin^{2}\frac{(j-k)\theta}{2}+4B\sin ^{2}\left((j-k)\theta\right)\\
    &=4\sin^{2}\frac{\rho(j,k)\pi}{10}+4B\sin^{2}\frac{\rho(j,k)\pi}{5}.
    \end{align*}
    Set 
    \begin{equation}
        \sigma_{r}^{2}(B):=4\sin^{2}\frac{r\pi}{10}+4B\sin^{2}\frac{r\pi}{5},\quad r=1,\cdots,5.
    \end{equation}
    Then 
    \[|z_{j}(B)-z_{k}(B)|=\sigma_{\rho(j,k)}(B).\]
    Using the fact 
    \[
    \cos\frac{\pi}{5}=\frac{1+\sqrt{5}}{4},\quad \cos\frac{2\pi}{5}=\frac{\sqrt{5}-1}{4},
    \]
    we obtain
    \begin{align*}
        &\sigma_{1}^{2}(B)=\frac{3-\sqrt{5}}{2}+\frac{5-\sqrt{5}}{2}B,\\
        &\sigma^{2}_{2}(B)=\frac{5-\sqrt{5}}{2}+\frac{5+\sqrt{5}}{2}B,\\
        &\sigma_{3}^{2}(B)=\frac{3+\sqrt{5}}{2}+\frac{5+\sqrt{5}}{2}B,\\
       & \sigma_{4}^{2}(B)=\frac{5+\sqrt{5}}{2}+\frac{5-\sqrt{5}}{2}B,\\
       & \sigma_{5}^{2}(B)=4.
    \end{align*}
    For $r=1,2,3,4,5$, we define
    \[\delta_{r}(B):=\kappa\sigma_{r}^{-3}(B).\]
    Then the coefficients
    \begin{equation*}
        M_{jk}(B)=\left\{\begin{aligned}
     &\kappa|z_{j}(B)-z_{k}(B)|^{-3}=\delta_{\rho(j,k)}(B),\quad &j\neq k,\\
     &0,\quad &j=k.
		\end{aligned}\right.
    \end{equation*}
    Hence, $M(B)=(M_{jk}(B))_{1\leq j,k\leq 10}$ is a real symmetric circulant matrix whose first row is 
    \[
    (0,\delta_{1}(B),\delta_{2}(B),\delta_{3}(B),\delta_{4}(B),\delta_{5}(B),\delta_{4}(B),\delta_{3}(B),\delta_{2}(B),\delta_{1}(B)).
    \]
    \textbf{Step 2: Eigenvalues and eigenvectors of $\boldsymbol{M(B)}$.}
For any integer $0\leq m\leq 9$, we define the vector
\[v_{m} := \left(e^{im(k-1)\theta}\right)_{k=1}^{10}.\]
By the standard theory of symmetric circulant matrices, $v_{m}$ is an eigenvector of $M(B)$ with eigenvalue
\begin{equation}
    \lambda_{m}(B) = 2\sum_{r=1}^{4}\delta_{r}(B)\cos(mr\theta) + (-1)^{m}\delta_{5}(B).
\end{equation}
Moreover, 
\[\lambda_{10-m} = \lambda_{m} \quad \text{for } 1\leq m\leq 9,\]
and
\[
v_{m}^{s} = \left(\sin\left(m(k-1)\theta\right)\right)_{k=1}^{10}, \quad v_{m}^{c} = \left(\cos\left(m(k-1)\theta\right)\right)_{k=1}^{10}
\]
are two real eigenvectors associated with the eigenvalue $\lambda_{m}$.\\
\textbf{Step 3: Choice of $\boldsymbol{B_{0}}$ and $\boldsymbol{a,b}$.}
In this step, we choose the corresponding parameters. First, a direct numerical computation gives
\[\lambda_{4}(4.70)\approx-1.7242975\times 10^{-3}<0,\]
\[\lambda_{4}(4.71)\approx 5.7146524\times 10^{-3}>0,\]
\[\lambda'_{4}(B)\in[0.7417451,0.7460485]\quad \text{for all}\ B\in[4.70,4.71].\]
Hence, by the intermediate value theorem, there exists a unique $B_{0}\in(4.70,4.71)$ such that
\begin{equation}\label{choice of B0}
    \lambda_{4}(B_{0})=\lambda_{6}(B_{0})=0.
\end{equation}
From now on, we fix $B_{0}$ and write
\[
M:=M(B_{0}),\quad M_{jk}(B_{0}):=M_{jk},\quad  z_{k}=z_{k}(B_{0}),\quad\lambda_{m}:=\lambda_{m}(B_{0}).
\]
Again, numerical computations give
\[\lambda_{0}\approx7.8069722,\qquad \lambda_{2}\approx3.1411361.\]
Therefore
\begin{equation}\label{lambda0lambda2}
    \lambda_{0}>0,\quad \lambda_{2}>0,\quad \frac{3}{2}<\frac{\lambda_{0}}{\lambda_{2}}<3.
\end{equation}
Now, we choose
\begin{equation}\label{definition of a and b}
    a=\sqrt{\frac{12}{5\lambda_{2}}-\frac{6}{5\lambda_{0}}},\quad b=\sqrt{-\frac{8}{5\lambda_{2}}+\frac{24}{5\lambda_{0}}}.
\end{equation}
From \eqref{lambda0lambda2}, $a$ and $b$ are well defined.
Moreover, since
\[
a^{2}-b^{2}=\left(\frac{12}{5\lambda_{2}}-\frac{6}{5\lambda_{0}}\right)-\left(-\frac{8}{5\lambda_{2}}+\frac{24}{5\lambda_{0}}\right)=\frac{2(2\lambda_{0}-3\lambda_{2})}{\lambda_{0}\lambda_{2}}>0,
\]
we have $a>b>0$.\\
\textbf{Step 4: Construction of a continuous family of positive solutions.}
For any $t\in\RR$, we define
\begin{equation}
    x_{k}(t):=a+b\cos \left( t+2(k-1)\theta\right),\quad k=1,\cdots, 10.
\end{equation}
For simplicity, we denote
\[
t_{k}:=t+2(k-1)\theta.
\]
Then $x_{k}(t)$ can be rewritten as
\[
x_{k}(t)=a+b\cos t_{k}.
\]
We now verify that, for the above choice of $z_k$, $a$, $b$, and for any $t\in\RR$, $x_k(t)$ satisfies equation \eqref{alg}.\\
First, by step 3 and the fact that $\cos t_{k}\geq -1$, we have
\begin{equation*}
    x_{k}(t)\geq a-b>0,\quad \text{for all }\ k=1,\cdots ,10\ \text{and }\ t\in\RR.
\end{equation*}
Using the identity
\[
(a+b\cos t)^{3}=a^{3}+\frac{3}{2}ab^{2}+\left(3a^{2}b+\frac{3}{4}b^{3}\right)\cos t+\frac{3}{2}ab^{2}\cos 2t+\frac{1}{4}b^{3}\cos 3t,
\]
we obtain
\begin{equation}
    x_{k}^{3}(t)=A_{0}+A_{1}\cos t_{k}+A_{2}\cos 2t_{k}+A_{3}\cos 3t_{k},
\end{equation}
where
\[A_{0}:=a^3+\frac{3}{2}ab^{2},\quad A_{1}:=3a^2b+\frac{3}{4}b^3,\quad A_{2}:=\frac{3}{2}ab^2,\quad A_{3}:=\frac{1}{4}b^3.\]
Since
\[v_{m}^{s} = \left(\sin\left(m(k-1)\theta\right)\right)_{k=1}^{10}, \quad v_{m}^{c} = \left(\cos\left(m(k-1)\theta\right)\right)_{k=1}^{10}\]
are two real eigenvectors associated with the eigenvalue $\lambda_{m}$, we see that the vectors
\[
\left(\cos(mt_{k}\right)_{k=1}^{10},\qquad m=0,1,2,3
\]
are  eigenvectors associated with the eigenvalue $\lambda_{2m}$. Hence, for any $1\leq k\leq 10$,
\begin{align*}
    \kappa\sum_{j\neq k, 1\leq j\leq 10}|z_{j}-z_{k}|^{-3}x_{j}^{3}(t)&=\sum_{j\neq k}M_{jk}x^{3}_{j}(t)\\
    &=\lambda_{0}A_{0}+\lambda_{2}A_{1}\cos t_{k}+\lambda_{4}A_{2}\cos 2t_{k}+\lambda_{6}A_{3}\cos 3t_{k}.
\end{align*}
It follows from \eqref{choice of B0} that
\begin{equation}\label{alg expand}
     \kappa\sum_{j\neq k, 1\leq j\leq 10}|z_{j}-z_{k}|^{-3}x_{j}^{3}(t)=\lambda_{0}A_{0}+\lambda_{2}A_{1}\cos t_{k}.
\end{equation}
We now compute $\lambda_{0}A_{0}$ and $\lambda_{2}A_{1}$.\\
First, by \eqref{definition of a and b},
\[
A_{0}=a\left(a^2+\frac{3}{2}b^2\right)=a\left[\left(\frac{12}{5\lambda_{2}}-\frac{6}{5\lambda_{0}}\right)+\frac{3}{2}\left(-\frac{8}{5\lambda_{2}}+\frac{24}{5\lambda_{0}}\right)\right]=\frac{6a}{\lambda_{0}}.
\]
Hence
\[
\lambda_{0}A_{0}=\lambda_{0}\frac{6a}{\lambda_{0}}=6a.
\]
A similar computation gives
\[
\lambda_{2}A_{1}=6b.
\]
Substituting this into \eqref{alg expand}, we get
\begin{equation*}
    \kappa\sum_{j\neq k, 1\leq j\leq 10}|z_{j}-z_{k}|^{-3}x_{j}^{3}(t)=6a+6b\cos t_{k}=6x_{k}(t).
\end{equation*}
Thus, we complete the construction of a continuous family of positive solutions to equation \eqref{alg}.
\end{proof}
\section{Asymptotic vanishing of the scaling parameters}
In this appendix, we show that the concentration condition
\[
        \mu_k(t)\to0,\qquad 1\leq k\leq K,
\]
in \eqref{asp} is in fact redundant.  More precisely, it follows from the
pure multi-bubble decomposition \eqref{Bub1}, the comparability of the scales,
and the convergence of the centers to distinct points.

\begin{lemma}[Asymptotic vanishing of the scaling parameters]
\label{lem:asymptotic-vanishing}
Let  $(u,\partial_tu):[0,+\infty)\to
        \dot H^1(\mathbb R^5)\times L^2(\mathbb R^5)$
be a solution of \eqref{NLW1}.  Assume that there exist continuous functions
$    \mu_k:[0,+\infty)\to(0,+\infty),
        \
        y_k:[0,+\infty)\to\mathbb R^5,
        \ 1\leq k\leq K,$
with $K\geq2$, such that the pure multi-bubble decomposition \eqref{Bub1}
holds.  Assume moreover that the scales are comparable and the centers have
distinct limits, namely
\begin{equation}\label{asp-without-vanishing}
\begin{aligned}
    &C_1\leq \frac{\mu_k(t)}{\mu_j(t)}\leq C_2,
    \qquad  1\leq k,j\leq K,\quad t\geq0,\\
    &\lim_{t\to+\infty}y_k(t)=z_k,
    \qquad  1\leq k\leq K,
\end{aligned}
\end{equation}
where $C_1,C_2>0$ and $z_1,\ldots,z_K$ are pairwise distinct points in
$\mathbb R^5$.  Then
\begin{equation}\label{muk 0}
    \lim_{t\to+\infty}\mu_k(t)=0,
    \qquad 1\leq k\leq K.
\end{equation}
\end{lemma}
\begin{proof}
    From \eqref{asp-without-vanishing}, it suffices to prove that
    \begin{equation*}
        \mu_{1}(t)\rightarrow 0\quad \text{ as}\quad t\rightarrow+\infty.
    \end{equation*}
    We argue by contradiction and assume that $\mu_{1}(t)\nrightarrow 0$ as $t\rightarrow+\infty$. Then there exists a constant $\varepsilon_{0}>0$ and a sequence $t_{n}\rightarrow+\infty$ such that 
    \begin{equation}
        \mu_{1}(t_{n})\geq \varepsilon_{0}\quad \forall\ n\in\NN.
    \end{equation}
    We split the argument into two cases.\\
    \textbf{Case 1: $\boldsymbol{\mu_{1}(t_{n})}$ is bounded.} In this case, passing to a subsequence (which we still denote by $\mu_{1}(t_{n})$), we may assume 
    \begin{equation*}
        \mu_{1}(t_{n})\rightarrow \ell_{1}\quad \text{for some }\ \ell_{1}\in(0,+\infty).
    \end{equation*}
    Now, from \eqref{asp-without-vanishing}, we obtain for any $2\leq k\leq K$,
    \begin{equation*}
        C_1\mu_{1}(t_{n})\leq \mu_{k}(t_{n})\leq C_{2}\mu_{1}(t_{n}).
    \end{equation*}
    Hence, $\{\mu_{k}(t_{n})\}_{n\in\NN}$ is bounded and away from 0. Thus, passing to another subsequence, we may assume
    \begin{equation*}
        \mu_{k}(t_{n})\rightarrow\ell_{k}\quad \text{for some}\ \ell_{k}\in(0,+\infty),\quad 1\leq k\leq K.
    \end{equation*}
    From \eqref{asp-without-vanishing}, we also have
    \begin{equation*}
        y_{k}(t_{n})\rightarrow z_{k},\quad 1\leq k\leq K.
    \end{equation*}
 For simplicity, we define
 \begin{equation*}
     Q(x):=\sum_{k=1}^{K}W_{\ell_{k},z_{k}}(x).
 \end{equation*}
 Then, by \eqref{Bub1} and the continuity of the map
 \[(\lambda,z)\mapsto W_{\lambda,z}\]
 from $(0,+\infty)\times\RR^{5}$ to $\dot{H}^{1}(\RR^{5})$, we have
 \begin{equation}\label{asytn}
     \left\lVert u(t_{n})-Q\right\rVert_{\dot{H}^{1}(\RR^{5})}+\left\lVert\partial_{t}u(t_{n})\right\rVert_{L^{2}(\RR^{5})}\rightarrow 0.
 \end{equation}
 Now, define for $s\geq 0$,
 \begin{equation*}
     w_{n}(s,x):=u(t_{n}+s,x),
 \end{equation*}
Then, by the translation invariance of \eqref{NLW}, each $w_{n}$ solves
 \[
 \partial^{2}_{s} w_{n}-\Delta w_{n}=f(w_{n})
 \]
 with initial data
 \[
 \left( w_{n}(0),\partial_{s}w_{n}(0)\right)=\left(u(t_{n}),\partial_{t}u(t_{n})\right)\rightarrow(Q,0)\quad \text{in}\ \dot{H}^{1}(\RR^{5})\times L^{2}(\RR^{5})\quad \text{as}\ t_{n}\rightarrow+\infty
 \]
 by \eqref{asytn}. By the standard local well-posedness theory, there exists a unique local solution
 \[
 U\in C([0,\tau_1];\dot{H}^{1})\cap C^1([0,\tau_1];L^2(\RR^5))
\cap L^{7/3}\bigl([0,\tau_1];L^{14/3}(\RR^5)\bigr)
 \]
 of 
 \[
 \partial^{2}_{s}U-\Delta U=f(U),\quad (U(0),\partial_{s}U(0))=(Q,0)
 \]
 for some $\tau_{1}>0$. Moreover, by the continuous dependence of the solution, we have
 \begin{equation}\label{asytau1}
     (w_{n}(s),\partial_{s}w_{n}(s))\rightarrow(U(s),\partial_{s}U(s))\quad \text{in}\quad C([0,\tau_{1}];\dot{H}^{1}\times L^{2})\quad \text{as}\quad t_{n}\rightarrow+\infty.
 \end{equation}
 On the other hand, by the assumption \eqref{Bub1} and the definition of $w_{n}(s)$, we obtain for any fixed $s\in[0,\tau_{1}]$,
 \begin{equation*}
     \left\lVert \partial_{s}w_{n}(s)\right\rVert_{L^{2}(\RR^{5})}=\left\lVert\partial_{t}u(t_{n}+s)\right\rVert_{L^{2}(\RR^{5})}\rightarrow 0\quad \text{as}\quad t_{n}\rightarrow+\infty,
 \end{equation*}
 which combines with \eqref{asytau1} gives that
 \begin{equation*}
 \partial_s U(s)=0\quad \text{for all } s\in[0,\tau_{1}].
 \end{equation*}
 Hence,
 \begin{equation*}
     U(s)=Q\quad \text{in}\ L^{2}(\RR^{5}),\quad\forall s\in[0,\tau_1].
 \end{equation*}
Furthermore, the equation 
\begin{equation}\label{ell}
    \Delta Q + f(Q) = 0
\end{equation}
holds in $\mathcal{D}'(\mathbb{R}^{5})$. By definition, the function
\[
Q(x) := \sum_{k=1}^{K} W_{\ell_{k}, z_{k}}(x)
\]
is smooth, so equation \eqref{ell} actually holds in the classical sense. However, since each $W_{\ell_{k}, z_{k}}$ solves $\Delta W_{\ell_{k}, z_{k}} + W_{\ell_{k}, z_{k}}^{\frac{7}{3}} = 0$, we have
\begin{equation*}
    \Delta Q + f(Q) = -\sum_{k=1}^{K} W_{\ell_{k}, z_{k}}^{\frac{7}{3}} + \left( \sum_{k=1}^{K} W_{\ell_{k}, z_{k}} \right)^{\frac{7}{3}}.
\end{equation*}
On the other hand, since each $W_{\ell_{k}, z_{k}}(x) > 0$ for all $x \in \mathbb{R}^{5}$ and $K \geq 2$, it follows that for every $x \in \mathbb{R}^{5}$,
\[
\left( \sum_{k=1}^{K} W_{\ell_{k}, z_{k}}(x) \right)^{\frac{7}{3}} > \sum_{k=1}^{K} W_{\ell_{k}, z_{k}}(x)^{\frac{7}{3}},
\]
which contradicts \eqref{ell}.\\
\textbf{Case 2: $\boldsymbol{\mu_{1}(t_{n})}$ is unbounded.} In this case, we may assume that
\begin{equation*}
    \lambda_{n}:=\mu_{1}(t_{n})\rightarrow+\infty.
\end{equation*}
Now define the rescaled solutions
\[v_{n}(s,x):=\lambda_{n}^{\frac{3}{2}}u(t_{n}+\lambda_{n}s,\lambda_{n}x),\quad s\geq 0,\ x\in\RR^{5}.\]
Equivalently,
\[
v_{n}(s)=u(t_{n}+\lambda_{n}s)_{\lambda_{n}^{-1},0}.
\]
Since \eqref{NLW} is invariant under scaling, each $v_{n}$ is again a solution of
\begin{equation*}
    \partial_{s}^{2}v_{n}-\Delta v_{n}=f(v_{n})\quad \text{on}\ [0,+\infty)\times \RR^{5}.
\end{equation*}
For each $k$, define
\[\gamma_{k,n}:=\frac{\mu_{k}(t_{n})}{\lambda_{n}},\quad \sigma_{k,n}:=\frac{y_{k}(t_{n})}{\lambda_{n}}.\]
From \eqref{asp-without-vanishing}, we have
\[
C_{1}\leq \gamma_{k,n}\leq C_{2}\quad \text{for all}\ k,n.
\]
Hence, passing to a subsequence, we may assume that
\begin{equation*}
    \gamma_{k,n}\rightarrow\gamma_{k}\in [C_{1},C_{2}],\quad 1\leq k\leq K.
\end{equation*}
Moreover, since $y_{k}(t_{n})\rightarrow z_{k}$ and $\lambda_{n}\rightarrow+\infty$, for each $1\leq k\leq K$, we have
\[
\sigma_{k,n}\rightarrow0\quad \text{as}\ n\rightarrow+\infty.
\]
Then, from \eqref{Bub1} and the scaling invariance of the $\dot{H}^{1}$ and $L^{2}$ norm,
\begin{equation}\label{asy infty}
    \left\lVert v_{n}(0)-\sum_{k=1}^{K}W_{\gamma_{k,n},\sigma_{k,n}}\right\rVert_{\dot{H}^{1}(\RR^{5})}+\left\lVert \partial_{s}v_{n}(0)\right\rVert_{L^{2}(\RR^{5})}\rightarrow0.
\end{equation}
Now, we define
\[
Q_{\infty}(x):=\sum_{k=1}^{K}W_{\gamma_{k},0}(x).
\]
From \eqref{asy infty}, we have
\begin{equation*}
    \left(v_{n}(0),\partial_{s}v_{n}(0)\right)\rightarrow ( Q_{\infty},0)\quad \text{in}\ \dot{H}^{1}(\RR^{5})\times L^{2}(\RR^5).
\end{equation*}
Then by the continuous dependence of the solution again and arguing similar as in Case 1, we obtain
\begin{equation*}
    \Delta Q_{\infty}+Q_{\infty}^{\frac{7}{3}}=0\quad \text{in}\ \mathcal{D}'(\RR^{5}).
\end{equation*}
Since $Q_{\infty}$ is smooth, this identity holds in the classical sense. But this is impossible for $K\geq 2$ by the same argument as in Case 1.\\
Finally, combining the previous arguments, the desired result follows immediately.
\end{proof}

\end{document}